\documentclass[a4paper]{amsart}
\usepackage[mathscr]{euscript}
\usepackage[all]{xy}
\usepackage{enumitem}

\usepackage{amsmath,amssymb,amsfonts,amscd}

\theoremstyle{plain}
\newtheorem{theorem}{Theorem}[section]
\newtheorem{lemma}[theorem]{Lemma}

\newtheorem{corollary}[theorem]{Corollary}
\newtheorem{definition}[theorem]{Definition}

\newtheorem{notation}[theorem]{Notation}

\theoremstyle{definition}

\newtheorem{remark}[theorem]{Remark}
\newtheorem{finalremark}[theorem]{Final remarks}

\newtheorem{remarks}[theorem]{Remarks}
\newtheorem{notes}[theorem]{Notes}
\def\<#1>{\langle\, #1\,\rangle}

\newcommand{\mB}{\mathscr{B}}
\newcommand{\mC}{\mathscr{C}}
\newcommand{\mE}{\mathscr{E}}
\newcommand{\mA}{\mathscr{A}}
\newcommand{\U}{\mathscr{U}}
\newcommand{\hh}{\mathbb{H}}
\newcommand{\mc}{\mathcal}
\newcommand{\card}[1]{\vert #1 \vert}

\newcommand{\g}{{\mbox{\tiny $G$}}}

\newcommand\restr[2]{{% we make the whole thing an ordinary symbol
  \left.\kern-\nulldelimiterspace % automatically resize the bar with \right
  #1 % the function
  \vphantom{|} % pretend it's a little taller at normal size
  \right|_{_{#2}} % this is the delimiter
  }}

\newcommand{\drr}{\mathrm{d\!-\!rank}}

\newcommand{\R}{\mathbb{R}}
\newcommand{\Q}{\mathbb{Q}}
\newcommand{\C}{\mathbb{C}}
\newcommand{\T}{\mathbb{T}}
\newcommand{\N}{\mathbb{N}}

\newcommand{\A}{\mathscr{A}}

\newcommand{\Z}{\mathbb{Z}}

\DeclareMathOperator{\supp}{supp}

\newcommand{\luc}{\mathit{LUC}(G)}

\renewcommand{\emptyset}{\varnothing}

\usepackage{tikz-cd}
\usepackage{mathtools}

\usepackage[stable]{footmisc}

\font\seis=cmr6
\def\CB{\mathscr{CB}}
\def\ruc{{\seis{\mathscr{RUC}}}}
\def\uc{\mathscr{UC}}
\def\luc{{\seis{\mathscr{LUC}}}}
\def\wap{{\seis{\mathscr{WAP}}}}

\def\lc{{\seis{\mathscr{LC}}}}

\def\im{{\rm im\,}}

\newcommand{\mW}{\mathscr{WAP}}

\begin{document}

\title[Arens irregularity in harmonic analysis]{$\ell^1$-bases,   algebraic structure and strong Arens irregularity of Banach algebras in harmonic analysis$^1$
}
%Factorization in  the Fourier algebra of a compact group and regularity--New property $(Z^\prime)$.}

\author[Filali and Galindo]{M. Filali \and  J. Galindo}


\address{\noindent Mahmoud Filali,
Department of Mathematical Sciences\\University of Oulu\\Oulu,
Finland. \hfill\break \noindent E-mail: {\tt mfilali@cc.oulu.fi}}
\address{\noindent Jorge Galindo, Instituto Universitario de Matem\'aticas y
Aplicaciones (IMAC)\\ Universidad Jaume I, E-12071, Cas\-tell\'on,
Spain. \hfill\break \noindent E-mail: {\tt jgalindo@mat.uji.es}}


\date{\today}

\thanks{$^1$The results were presented in 2019 IMAC AHA Days (Castellon, Spain) and in the conference on Banach Algebras and Applications, 2017 (Oulu, Finland). }
%The results were included in the survey \cite{FGsurvey}}

\begin{abstract}
A long standing problem in abstract harmonic analysis concerns the strong Arens irregularity (sAir, for short) of the Fourier algebra  $A(G)$ of a locally compact group $G.$
The groups for which $A(G)$ is known to be sAir are all amenable. So far
this class includes the abelian groups, the discrete amenable groups,
the second countable amenable groups $G$ such that $\overline{[G, G]}$ is
not open in $G,$
the groups of the form $\prod_{i=0}^\infty G_i$ where each $G_i$, $i\ge1$, is a non-trivial metrizable compact group and $G_0$ is an amenable  second countable locally compact group,
the groups of the form $G_0\times G$, where $G$ is a compact group whose local weight $w(G)$ has uncountable
cofinality and $G_0$ is any locally compact amenable group with $w(G_0)\le w(G)$,
and the compact group $SU(2).$

We were primarily concerned with the groups for which $A(G)$ is sAir.
We introduce a new class of $\ell^1$-bases in Banach algebras. These new $\ell^1$-bases enable us, among other results, to unify most of
 the results related to Arens products proved in the past seventy years since Arens defined
his products. This includes the strong Arens irregularity of algebras in harmonic analysis, and in particular almost all the cases mentioned above for the Fourier algebras.
In addition, we also show that $A(G)$ is sAir  for compact connected groups with an infinite dual rank. The $\ell^1$-bases for the Fourier algebra are constructed with coefficients of certain
irreducible representations of the group.
With this new approach using $\ell^1$-bases, the rich algebraic structure
of the algebras and semigroups under study such as the second dual of Banach algebras with an Arens product or certain semigroup compactifications (the Stone-\v Cech compactification of an infinite discrete group, for instance) is also unveiled
\bigskip
\end{abstract} \maketitle

\setcounter{tocdepth}{1}
\tableofcontents

\section{Introduction} A Banach algebra $\mA$ is Arens regular when the two Arens products defined on its second Banach dual $\mA^{**}$
coincide. This is equivalent to saying that one of the products (and so both products are) is separately weak$^*$-continuous making $\mA^{**}$  a semitopological algebra when equipped with the weak$^*$-topology. It is worthwhile to note, that due to Grothendieck's criterion, this is in turn equivalent to that every element in the Banach dual $\mA^*$
is weakly almost periodic, i.e., $\mA^*=\wap(\mA).$ This fact is not going to be used in the present article, but was essential in our study of Arens irregularity
of Banach algebras in harmonic analysis in the articles \cite{enar1} and \cite{enar2}.
As seen already in these past two articles and as we shall further see presently, unlike C$^*$-algebras, the usual algebras  in harmonic analysis such as the group algebra, and   the Fourier algebra in most cases, have turned out to be all non-Arens regular
(even in an extreme way).
At the very beginning,  about seventy years ago,
Richard Arens  proved in \cite{Arens51} that the convolution algebra $\ell^1$ is Arens irregular.
A long exciting story followed, and so much has already been written on the subject ever since.
For more details, see the  surveys \cite{DH} and \cite{filali-singh}, the book \cite{dales}, the memoirs \cite{DaLa05} and \cite{dales-lau-strauss}, and our most recent survey \cite{FGsurvey} where the results of the actual article were announced.

There are two different types of utmost non-Arens regularities, arising naturally from the way the algebras are decided to be non-Arens regular.

Since $\mA$ is Arens regular if and only if $\mA^*=\mW(\mA)$, the degree of non-Arens regularity of an algebra $\mA$ may be measured
by the size of the quotient space $\frac{\mA^*}{\wap(\mA)}.$ In the extreme situation (i.e.,  when  $\frac{\mA^*}{\wap(\mA)}$ is as large as $\mA^*$) Granirer called such an algbera $\mA$  extremely non-Arens regular (enAr for short) and proved that this was the case for the Fourier algebra of a non-discrete second countable locally compact group $G$ (see \cite[Corollaries 6 and 7]{granirer}).
Granirer's result was extended by Hu in \cite{Hu97} to locally compact groups $G$ having their local weight $w(G)$ greater  or equal to their compact covering $\kappa(G).$

On the dual hand, the group algebra  is
isometrically enAr, and so enAr, for any infinite locally compact group $G$. This was proved in \cite{FV} when $G$ is discrete, in \cite{BF} when $\kappa(G)\ge w(G)$, and finally in \cite{FG} for the general case.
The semigroup algebra $\ell^1(S)$ is isometrically enAr for any infinite weakly cancellative discrete semigroup
(see \cite{FV}). Here, a Banach algebra is isometrically enAr when $ \frac{\mA^*}{\wap(\mA)}$ contains an isometric copy of $\mA^*.$

This topic  was our main focus in \cite{enar1} and \cite{enar2}, where
$\ell^1$-bases approximating lower and upper triangles of products of elements in $\mA$ were the reason behind the non-Arens regularity of $\mA$, and in many cases the extreme  non-Arens regularity of $\mA$.
This unified approach dealing with the extreme non-Arens regularity helped us to improve and extend the above-cited results to a large class of Banach algebras. The combinatorics of the locally compact group $G$ and its dual object $\widehat G$ were the tools to construct these $\ell^1$-bases for many algebras which are associated to $G$ in harmonic analysis.

Non-Arens regularity may also be measured by the degree of defect of the separate weak$^*$-continuity of
Arens products. In the extreme situation (i.e., when separate weak$^*$-continuity of the product in $\mA^{**}$ happens only when one of the elements in the product is in $\mA$),  Dales and Lau called the algebra $\mA$ strongly Arens irregular
( see \cite{DaLa05}). This will be our main focus in the present paper.
A more refined  $\ell^1$-bases in $\mA$ are necessary to show the strong Arens irregularity of Banach algebras associated to a locally compact group $G.$
Again the reason making these refined $\ell^1$-bases available are the combinatorics in $G$ and $\widehat G.$

The paper is organized as follows. After a preliminary section recalling the basic definitions and results needed in the paper, we present our general theory in Sections 3 and 4. The new definitions of some special $\ell^1$-bases, we call F$\ell^1(\eta)$-base, are given for any Banach algebra $\mA$ ($\eta$ is an infinite cardinal which is going to be the Mazur property level of the algebra $\mA$ when we deal with the topological centres). The theorems in these two  sections show how a presence of these F$\ell^1(\eta)$-bases in $\mA$ yield strong Arens irregularity of the algebra $\mA$ and  induces a rich algebraic structure in $\mA^{**}$. Similar results are also proved in $\luc(\mA)^*$ when $\mA$ has a bounded approximate identity. Under certain conditions the theorems are proved also for the right multiplier algebra $M_r(\mA)$ of $\mA.$

We illustrate our general theory with three sections, starting with the group algebra and the measure algebra of a non-compact locally compact group and the semigroup algebra of an infinite weakly cancellative, right cancellative, discrete semigroup. Then come the weighted analogues of these algebras. These are studied when the weight is diagonally bounded.
Using  some special subsets of $G$ constructed inductively,  the required F$\ell^1(\eta)$-bases are produced in each case, where $\eta$ is the compact covering of  the locally compact group,
 or the cardinality of the discrete semigroup.

  The penultimate section deals with the Fourier algebra of a compact group.
The cases we study first are compact groups with local weight having
 uncountable cofinality, then infinite product of compact groups, and finally  compact connected  groups with a certain  infinite dimension (this will be defined in time).
Using  the irreducible representations of $G$,  we shall produce in each case the required F$\ell^1(\eta)$-bases, where $\eta$ is the local weight of $G$.

 With the availability of the  required F$\ell^1(\eta)$-base in each of the cases studied in Sections \ref{Examples}, \ref{wexamples} and \ref{Aexamples},
the topological centers as well as the algebraic structure of all the algebras in question will be an immediate consequence of   Sections \ref{FC}, \ref{FF} and \ref{algebra}.

 In the last section, we shall extend our results to the case when the group is the  product $G\times H$, where $G$ is any
of the previous compact groups and $H$ is an amenable locally compact group with local weight
smaller or equal to that of $G$.

The algebraic structures and the topological centres of the Banach algebras $A(G\times H)^{**}$, $\uc_2(G\times H)^*$ or $C^*(G\times H)^{\perp}$ will be our final conclusion.

Some of these results have been proved in the past by various authors as we shall indicate as we go along.
In addition to the unification of all these results, there will be many other new results.

\section{Some basic definitions and notation}
Let $\mA$ be a Banach algebra.
The elements in $\mA$ will be denoted by miniscule Roman letters.
The elements in the Banach dual $\mA^* $ of $\mA$ will be denoted by majuscule Roman letters.
 Majuscule  Greek letters will be used to denote the elements in the second dual of the algebra.
With the exceptions of the product in $\mA$ and its first Arens extensions to  $\mA^{**}$
and $\luc(\mA)^*$, all the module actions will  be denoted by "$\cdot$".
With the elements denoted as such, there will be no confusion as to which module action is considered.

\subsection{Arens products} Recall that the Arens module actions are defined via the adjoint of continuous bilinear mappings (see \cite{Arens51} and \cite{A2}). That is, if
\[\mE_1\times \mE_2\to \mE_3\quad\text{ is a continuous bilinear mapping}\] on some normed spaces $\mE_i$ ($i=1,2,3$), then its (3{\scriptsize 1}2)-adjoint given by \[\mE_3^*\times \mE_1\to \mE_2^*\quad\text{is again a continuous bilinear mapping},\] and so in turn, its (3{\scriptsize 1}2)-adjoint may be be taken. The process can be repeated indefinitely.

Here are the module actions we shall be using throughout the paper.
We start with a Banach algebra $\mA$ and regard it as a left $\mA$-module with its operation
as the left module action of $\mA$ on itself; that is
\[\mA\times \mA\to\mA,\quad\text{where}\quad a\cdot b=ab.\] Then by taking the (3{\scriptsize 1}2)-adjoint,  we shall regard
$\mA^*$ as a right $\mA$-module, where the natural right module action of $\mA$ on $\mA^*$ is given by
\begin{equation}\label{action1}\mA^*\times \mA\to \mA^*,\;\mbox{ where }\;\langle S\cdot a, b\rangle=\langle S, a b\rangle\;\mbox{ for }\; S\in \mA^*,\; a,b\in\mA.\end{equation}
 This is followed (taking again the (3{\scriptsize 1}2)-adjoint) by regarding $\mA^{*}$ as a left $\mA^{**}$-module, where the natural left module action of $\mA^{**}$ on $\mA^*$ is given by
\begin{equation}\label{action2}
\mA^{**}\times \mA^*\to \mA^*,\;\text{where}\;\langle \Psi\cdot S,a\rangle=\langle \Psi, S\cdot a\rangle\quad\text{for}\; \Psi\in \mA^{**},\; S\in \mA^*, a\in\mA.\end{equation}
This induces the first (left) Arens product in $\mA^{**}$ making $\mA^{**}$ a Banach algebra. For $\Phi,\Psi\in \mA^{**},$ the product is given by
\begin{equation}\label{actionthree}
\mA^{**}\times \mA^{**}\to \mA^{**},\quad\text{where}\quad\langle \Phi\cdot \Psi,S\rangle=\langle \Phi, \Psi\cdot S\rangle\quad\text{for}\quad  S\in \mA^*.
\end{equation}
The first  Arens product in $\mA^{**}$ is the unique extension of the product in $\mA$   such that
\begin{enumerate}
\item  [(i)] $\Phi\mapsto \Phi\Psi:\mA^{**}\to\mA^{**}\quad\text{ is weak$^*$-weak$^*$-continuous}$ for each $\Psi\in \mA^{**},$
\item [(ii)] $\Phi\mapsto a\Phi:\mA^{**}\to\mA^{**}\quad\text{is weak$^*$-weak$^*$-continuous}$ for each $a\in \mA.$
\end{enumerate}
The first topological centre of $\mA^{**}$ may therefore be defined as
\[T_1Z(\mA^{**})=\{\Psi\in \mA^{**}: \Phi\mapsto \Psi\Phi:\mA^{**}\to\mA^{**}\quad\text{is weak$^*$-weak$^*$-continuous}\}.\]
Note that Item (ii) above means that $\mA\subseteq T_1Z(\mA^{**}).$

If we start with $\mA$ as a right $\mA$-module with the right module action given by the transpose of the operation in $\mA$, that is $a\odot b=ba$ for $a, b\in \mA,$ then taking the transpose of the left action
of  $\mA^{**}$ on itself reached at the third step, will give the second (right) Arens product.
We make an exception here and denote the second Arens product of two elements $\Phi$ and $\Psi$ in $\mA^{**}$
by $\Phi\diamond \Psi$ and the actions leading to this product by $"\odot"$ so that, for every $S\in \mA^*$ and $a,b\in \mA$, we have
\begin{equation}\label{second}\begin{split}&\langle \Phi\diamond \Psi, S\rangle=\langle \Psi\odot\Phi,S\rangle=
\langle \Psi, \Phi\odot S\rangle,\quad \text{where}\\&
\langle\Phi\odot S, a\rangle=\langle\Phi, S\odot a\rangle\quad\text{and}\\&
\langle S\odot a, b\rangle= \langle S, a\odot b\rangle=\langle S, ba\rangle\end{split}.\end{equation}

We note that the second Arens product in $\mA^{**}$ is the unique extension of the product in $\mA$   such that
\begin{enumerate}
\item [(i)]  $\Psi\mapsto \Phi\diamond\Psi:\mA^{**}\to\mA^{**}\quad\text{is weak$^*$-weak$^*$-continuous}$ for each $\Phi\in \mA^{**},$
\item [(ii)] $\Phi\mapsto \Phi\diamond a:\mA^{**}\to\mA^{**}\quad\text{is weak$^*$-weak$^*$-continuous}$ for each $a\in \mA.$
\end{enumerate}
The second topological centre of $\mA^{**}$ may therefore be defined as
\[T_2Z(\mA^{**})=\{\Psi\in \mA^{**}: \Phi\mapsto \Phi\diamond\Psi:\mA^{**}\to\mA^{**}\quad\text{is weak$^*$-weak$^*$-continuous}\}.\]
Again Item (ii) means that $\mA\subseteq T_2Z(\mA^{**}).$

When these two ways yield the same product, i.e., if these two products coincide, we say that the algebra $\mA$ is {\it Arens regular}.
It is clear that the algebra is Arens regular if and only if one  of the Arens products  (and so both are) is separately weak$^*$-weak$^*$-continuous, i.e., \[T_1Z(\mA^{**})=T_2Z(\mA^{**})=\mA^{**}.\]
Such a situation happens for instance when $\mA$ is a $C^*$-algebra. This was first deduced from the fact that the second dual of a $C^\ast$-algebra can be seen as an algebra of operators  on a Hilbert space, a result proved    by Sherman in \cite{Sh}   and  Takeda in \cite{T}. Later on,    in their seminal paper \cite{CiYo}, Civin and Yood reproduced  Takeda's  proof and  brought up explicitly the Arens regularity of $C^\ast$-algebras.
A different proof of this can also be found in \cite[Theorem 38.19]{BD}.

When the Banach algebra is not Arens regular, i.e., when
  \[\mA\subseteq T_iZ(\mA^{**}) \subsetneq \mA^{**}\quad \text{for}\quad i=1\quad\text{or}\quad 2,\]
	  or equivalently when $\mA^* \ne  \mW(\mA)$,
we say that the algebra is {\it Arens irregular} or {\it non-Arens regular}.

In the extreme situation of $T_1Z(\mA^{**}) = \mA$ (respectively, $T_2Z(\mA^{**})=\mA$),
Dales and Lau called the algebra $\mA$ \emph{ left }(respectively, \emph{right}) \emph{strongly Arens irregular}.
  In the very extreme situation of $T_1Z(\mA^{**}) =T_2Z(\mA^{**})=\mA$, the algebra $\mA$
is {\emph strongly Arens irregular (sAir)}, see \cite{DaLa05}.
Note that when $\mA$ is commutative (as itis the case with the Fourier algebra $A(G)$ studied in the last two sections),  $T_1Z(\mA^{**})$ and  $T_2Z(\mA^{**})$
coincide with the algebraic centre of $\mA^{**}$, that is,
\begin{align*} T_1Z(\mA^{**}) =T_2Z(\mA^{**})&=\{\Psi\in \mA^{**}: \Phi\Psi=\Psi\Phi\quad\text{for every}\quad\Phi\in \mA^{**}\}\\&=
\{\Psi\in \mA^{**}: \Phi\diamond\Psi= \Psi\diamond\Phi\quad\text{for every}\quad\Phi\in \mA^{**}\}
.\end{align*} In this case, we will be talking just about  the strongly Arens irregularity.

It should be also recalled  that $T_1Z(\mA^{**})$ and $T_2Z(\mA^{**})$ may be different, and it may even happen that $T_1Z(\mA^{**}) = \mA\ne T_2Z(\mA^{**})$, i.e., $\mA$ may be only one-sided strongly Arens irregular (see \cite{DaLa05}, \cite{EF},  \cite{GMM} and \cite{Neufang09}).

\subsection{The space $\luc(\mA)$}
We shall also be concerned with the dual of the following closed subspace of $\mA^*.$
This is the closed linear span in $\mA^*$ of $\mA^*\cdot\mA$, which we denote as usual by
$\luc(\mA)$. By Cohen's factorization theorem, this is $\mA^*\cdot \mA$ when $\mA$ has a bounded right approximate identity.

Now since  $(S\cdot a)\cdot b=S\cdot (ab)$ for every $S\in\mA^*$ and $a,b\in \mA,$ we see that $\luc(\mA)$ is an $\mA^*$-submodule, where the
module action is given as in (\ref{action1}) (but with $\luc(\mA)$ instead of $\mA^*$) by
\[\luc(\mA)\times \mA\to \luc(\mA):(S\cdot a, b)\mapsto S\cdot (ab).\]
After taking the (3{\scriptsize 1}2)-adjoint, we obtain the action $\luc(\mA)^*\times \luc(\mA)\to \mA^*.$
But as known and easy to check (noting that $\Phi\cdot (S\cdot a)=(\Phi\cdot S)\cdot a$ for every $\Phi\in \luc(\mA)^*$, $S\in \mA^*$ and $a\in \mA$), the range of this  mapping is in  $\luc(\mA).$
In  other words, we have in fact $\luc(\mA)$ as a right $\luc(\mA)^*$-module
\[\luc(\mA)^*\times \luc(\mA)\to \luc(\mA)\] (in the literature, a  subspace of $\mA^*$ with such a property is said to be left introverted),
and so as in (\ref{actionthree}), this induces by taking the (3{\scriptsize 1}2)-adjoint, the first Arens product in
$\luc(\mA)^*$ making $\luc(\mA)^*$ a Banach algebra. As in (\ref{actionthree}), the product of  $\Phi,$ $\Psi\in \luc(\mA)^*$ is given by
\[\langle\Phi\Psi, S\rangle=\langle \Phi,\Psi\cdot S\rangle\quad\text{for all}\quad S\in \luc(\mA).\]
This product is the unique extension of the product in $\mA$  to $\luc(\mA)^*$ such that
\begin{enumerate}
\item  [(i)] $\Phi\mapsto \Phi\Psi:\luc(\mA)^{*}\to\luc(\mA)^{*}\quad\text{is weak$^*$-weak$^*$-continuous}$ for each $\Psi\in \luc(\mA)^{*},$
\item [(ii)] $\Phi\mapsto a\Phi:\luc(\mA)^{*}\to\luc(\mA)^{*}\quad\text{is weak$^*$-weak$^*$-continuous}$ for each $a\in \mA.$
\end{enumerate}

The (first) topological centre of $\luc(\mA)^{*}$ may therefore be defined as
\begin{align*}T_1Z(\luc(\mA)^{*})=\{\Psi\in \luc(\mA)^{*}: \Phi\mapsto\Psi\Phi:&\luc(\mA)^{*}\to\luc(\mA)^{*}\\&
\quad\text{is weak$^*$-weak$^*$-continuous}\}.\end{align*}

It may be also worthwhile to note that when $\mA$ has a brai ( the definition is given is Section 3),
$\luc(\mA)^*$ and the quotient $\frac{\mA^{**}}{\luc(\mA)^\perp}$ are isomorphic as Banach algebras,
where \[\luc(\mA)^\perp=\{\Psi\in \mA^{**}:\langle \Psi, S\rangle =0\quad\text{for every}\quad S\in \luc(\mA)\}\] is a closed ideal in $\mA^{**}$.
For more details, see \cite{balapy}.

Note finally that by the definition of $\luc(\mA)$, the module action (\ref{action1}) may be put as
\[\mA^*\times \mA\to \luc(\mA),\] and so the (3{\scriptsize 1}2)-adjoint of this bilinear mapping makes
$\mA^{*}$ as a left $\luc(\mA)^{*}$-module, where the left action of $\luc(\mA)^{*}$ on $\mA^*$ is given by
\begin{equation}\label{action3}\begin{split}&\luc(\mA)^{*}\times \mA^*\to \mA^*,\text{where}\\&\langle \Psi\cdot S,a\rangle=\langle \Psi, S\cdot a\rangle,\;\mbox{ for }\; \Psi\in \luc(\mA)^{*},\; S\in \mA^*,\; a\in\mA.\end{split}\end{equation}
Of course this action extends in turn to
\begin{equation}\label{action4}\begin{split}&\mA^{**}\times \luc(\mA)^{*}\to \mA^{**},\text{where}\\&\langle \Phi.\Psi,S\rangle=\langle \Phi, \Psi.S\rangle\;\mbox{ for }\; \Phi\in\mA^{**},\;
\Psi\in \luc(\mA)^*,\; S\in \mA^*.\end{split}\end{equation}
The actions (\ref{action3}) and (\ref{action4}) do not conflict with the actions  (\ref{action2}) and (\ref{actionthree}) giving the first Arens product for $\mA^{**}$. In fact, if
 $\pi:\mA^{**}\to \luc(\mA)^*$ is the adjoint of the inclusion map
$\luc(\mA)\hookrightarrow \mA^*$, then   for $S\in \mA^*,$ $\Phi\in \mA^{**}$ and $\Psi\in \luc(\mA)^*$, we have
\begin{equation}\label{trick}\Psi.S=\widetilde\Psi.S\quad\text{and}\quad \Phi.\Psi=\Phi\widetilde\Psi,\end{equation} where $\widetilde\Psi$ is picked in $\mA^{**}$ such that $\pi(\widetilde\Psi)=\Psi$ (any Hahn-Banach extension of $\Psi$ to $\mA^*$ will do).
So we keep, as agreed, the dot to denote all the actions in the paper.
No tilde will be used for the elements in $\mA$, so miniscule Roman letter will be used
whether $\mA$ is regarded as a subalgebra of $\mA^{**}$ or $\luc(\mA)^*$.

Note that if $(\Psi_\alpha)$ is a bounded net in $\luc(\mA)^*$ with weak$^*$-limit $\Psi$
in  $\luc(\mA)^*$ and  $(\widetilde{\Psi_\alpha})$ is the corresponding bounded net of extensions in $\mA^{**}$
with weak$^*$-limit $\Phi$ in $\mA^{**}$ , then by (\ref{trick})
\[\widetilde{\Psi_\alpha}\cdot S=\Psi_\alpha\cdot S\quad\text{for every}\;\alpha\;\text{and}\; S\in \mA^*,\] and so
\begin{equation*}\Phi\cdot S=\Psi\cdot S\quad\text{for every}\quad S\in \mA^*.\end{equation*}

\medskip

The analogue of the second Arens product defined on $\mA^{**}$ in (\ref{second})
and so the (second) topololgical centre of $\luc(\mA)^{*}$ cannot be defined for $\luc(\mA)^*$
unless $S\odot a$ and $\Phi\odot S$ are in $\luc(\mA)$ for every  $a\in\mA$, $S\in \mA$ and $\Phi\in \luc(\mA)^*$,
which happens for example when $\mA$ is commutative.
Here, we must use the space $\ruc(\mA)$ instead. This is the closed linear span in $\mA^*$ of $\mA^*\odot \mA$, which again by Cohen's factorization theorem, is equal to $\mA^*\odot \mA$ when $\mA$ has a bounded left approximate identity.

Since
$\ruc(\mA)\odot\mA\subseteq \ruc(\mA)$ and $\ruc(\mA)^*\odot\ruc(\mA)\subseteq \ruc(\mA)$,
the same procedure as before yields the second Arens product in $\ruc(\mA)^*$
 making $\ruc(\mA)^*$ a Banach algebra:  As in (\ref{second}), the product of  $\Phi,$ $\Psi\in \ruc(\mA)^*$ is given by
\[\langle\Phi\diamond\Psi, S\rangle=\langle\Psi\odot\Phi, S\rangle=
\langle \Psi,\Phi\odot S\rangle\quad\text{for all}\quad S\in \ruc(\mA).\]
This product is the unique extension of the product in $\mA$  to $\ruc(\mA)^*$ such that
\begin{enumerate}
\item  [(i)] $\Psi\mapsto \Phi\diamond\Psi:\ruc(\mA)^{*}\to\ruc(\mA)^{*}\quad\text{is weak$^*$-weak$^*$-continuous}$ for each $\Phi\in \luc(\mA)^{*},$
\item [(ii)] $\Phi\mapsto \Phi \diamond a:\ruc(\mA)^{*}\to\ruc(\mA)^{*}\quad\text{is weak$^*$-weak$^*$-continuous}$ for each $a\in \mA.$
\end{enumerate}

The (second) topololgical centre of $\ruc(\mA)^{*}$ may therefore be defined as
\begin{align*}T_2Z(\ruc(\mA)^{*})=\{\Psi\in \ruc(\mA)^{*}: \Phi\mapsto\Psi\diamond \Phi:&\ruc(\mA)^{*}\to\ruc(\mA)^{*}\\&
\quad\text{is weak$^*$-weak$^*$-continuous}\}.\end{align*}
Note again that $\mA\subseteq T_2Z(\ruc(\mA)^{*}).$

 Note also   if $\Phi\odot S$ and $\Phi\odot\Psi$
are the actions of $\ruc(\mA)^*$ on $\mA^*$ and $\mA^{**}$ as defined with $\luc(\mA)$ in (\ref{action3}) and (\ref{action4}), respectively,  then
\begin{equation*}\Phi\odot S=\widetilde\Phi\odot S\quad\text{and}\quad
\Psi\odot\Phi=\widetilde\Phi\diamond\Psi,\end{equation*} where $\widetilde\Phi$ is picked by Hahn-Banach Theorem in $\mA^{**}$ such that $\pi(\widetilde\Phi)=\Phi$.
Here $\pi:\mA^{**}\to \ruc(\mA)^*$ is the adjoint of the inclusion map
$\ruc(\mA)\hookrightarrow \mA^*$.

\section{Factorizations of $\mA^*$ and $\luc(\mA^*)$ and topological centres}\label{FC} Let $\mA$ be a Banach algebra and $\eta$ be a cardinal number. Then the Banach dual
$\mA^*$ of $\mA$ (respectively, $\luc(\mA)$)
is said to have the $\eta$-factorization property if for every bounded family of functionals
$(S_\lambda )_{\lambda <\eta}$ contained in $\mA^*$,
(respectively, $\luc(\mA)$),
there exist a bounded family $(\Psi_\lambda )_{\lambda<\eta}$ in $\luc(\mA)^*$ (usually in the unit ball of $\luc(\mA)^*$) and a single functional $T \in \mA^{*}$ (respectively, in $\luc(\mA)$)
such that the factorization formula $S_\lambda =\Psi_\lambda \cdot T$
holds for all $\lambda<\eta$.

This factorization was obtained and used, together with the so-called Mazur property, by Neufang in \cite{Neufang04B} to show that the group algebra of a non-compact locally compact group is strongly Arens irregular. In some cases, the same result was also obtained by the same author for the measure algebra in \cite{Neufang05}.

The factorization property (and some of its variants) was also used to study ideals and right cancellation in Banach algebras and semigroup compactifications with an Arens product such as
$L^1(G)^{**}$, $\luc(G)^*$ and the $\luc$-compactification $G^{LUC}$ of non-compact locally compact groups, $\ell^\infty(S)^*$ and the Stone-\v Cech compactification of infinite discrete semigroups $S$ having some cancellation properties,   $A(G)^{**}$ and $\uc_2(G)^*$ for a class of non-metrizable locally compact groups. See for example, \cite{F},  \cite{MMM2}, \cite{FiPy03},  \cite{FiSa07}
and \cite{FiSa07a}.

We define next a weaker factorization property, but we keep the same terminology.
 This   weak $\eta$-factorization property  was introduced and used in \cite{MMM1}. It was used, together with the Mazur property, to show that the Fourier algebra of a certain class of non-metrizable groups is strongly Arens irregular.

\begin{definition}\label{wfp}	Let $\mA$ be a Banach algebra,  $\eta$ be a cardinal number and $X_1,...,X_m$ be  subsets of $\luc(\mA)^*$ each with cardinality at least $\eta$. Then the Banach dual
$\mA^*$ of $\mA$
is said to have the $\eta$-factorization property through $\prod_{k=1}^mX_k$ if for every bounded family of functionals
$\{S_\lambda:\lambda <\eta\}$ contained in $\mA^*$,
there exist a bounded family $\{(\Psi_{\lambda k})_{k=1}^m :\lambda<\eta\}$ in $\prod_{k=1}^mX_k$
and a finite family of functionals $\{T_k:k=1,...,m\}$ in $\mA^{*}$
such that the factorization formula \begin{equation}\label{F}S_\lambda =\sum_{k=1}^m\Psi_{\lambda k} \cdot T_k\end{equation}
holds for all $\lambda<\eta$.

If the elements $\{\Psi_\lambda=(\Psi_{\lambda k})_{k=1}^m :\lambda<\eta\}$ can be chosen in $\prod_{k=1}^mX_k$  independently of the elements $S_\lambda$, then we say that $\mA^*$ has the uniform
$\eta$-factorization property. (This is going to be so in all the cases treated in the paper.)
\end{definition}
	\medskip

\begin{remarks}\label{mirror}

\begin{enumerate}
~

\item  Putting
$T=(T_k)_{k=1}^m\in {\mA^*}^m$,  we may write the factorization formula (\ref{F}) as
\[S_\lambda= \Psi_\lambda\cdot T \quad\text{for all}\quad \lambda<\eta.\]

\item The $\eta$-factorization property just defined is one-sided and should have been labeled "left" since we are using the left module action $\Psi \cdot T$
of $\luc(\mA)^*$ on $\mA^*$ (see \ref{action3}).
The mirror $\eta$-factorization property may also be defined using the action $\Psi \odot T$ of $\ruc(\mA)^*$ on $\mA^*$. Noting that
the left action $\Psi \odot T$ of $\ruc(\mA)^*$ on $\mA^*$ is the same as the right action $T\cdot \Psi$   of $\ruc(\mA)^*$ on $\mA^*$ induced by the left action $a\cdot T$ of $\mA$ on $\mA^*$, where $\langle a\cdot T, b\rangle=\langle T, ba\rangle$ for $b\in \mA,$
we may say that $\mA^*$ has the right $\eta$-factorization property through $\prod_{k=1}^mX_k$
when the identity (\ref{F}) in Definition \ref{wfp} is changed to
\[S_\lambda =\sum_{k=1}^m   \Psi_{\lambda k}\odot T_k =\sum_{k=1}^m   T_k\cdot\Psi_{\lambda k}=T\cdot \Psi_\lambda \]
 for all $\lambda<\eta$.
\end{enumerate}
\end{remarks}

In order to apply this factorization property in Theorems \ref{Lcentre} and \ref{lucentre} dealing with the topological centres of $\mA^{**}$ and $\luc(\mA)^*$,
we recall the following definition from \cite{Neufang05}.

\begin{definition}\label{de:090908B}
Let $\mE$ be a Banach space and $\eta$ be an infinite  cardinal number.
\begin{enumerate}
\item A functional $\Psi\in \mE^{**}$ is called $\eta$-weak$^*$-continuous if for every infinite
net $(S_\lambda )_\lambda$ in the unit ball of $\mE^*$ of cardinality at most $\eta$ with $S_\lambda \to 0$ in the weak$^*$-topology,
we have $\langle\Psi,S_\lambda \rangle\to 0$.
\item We say that $\mE$ has the $\eta$-Mazur property  if every $\eta$-weak$^*$-continuous
functional $\Psi\in \mE^{**}$ is an element of $\mE$.
\end{enumerate}
\end{definition}
\medskip

\begin{notation}\label{notation}~\normalfont
 Throughout the paper, with the set $\prod_{k=1}^mX_k$ in $\luc(\mA)^*$ through which the factorization of $\mA^*$ is made, we  let $X=\bigcup\limits_{k=1}^mX_k$ and
 $\widetilde{X}$ be the set made of some fixed extensions
of the elements in $X$ to $\mA^*$  such that $\|\widetilde{\Psi}\|=\|\Psi\|$ as already mentioned
in (\ref{trick}).
\end{notation}

Theorems \ref{Lcentre} and \ref{lucentre} below were proved  by Neufang (see \cite[Theorem 2.3 and Theorem 4.2]{Neufang05}) with with $m=1,$ $\mathscr S=\mA^{**}$ and $ X$ in the unit ball of $\luc(\mA)^{*}$.
The proof of the first theorem is a modification of the proof given in  \cite[Theorem 2.3]{Neufang05}, while we rely on  \cite{laul}
for the proof of the second theorem. Both proofs are included so that the role of the set $\prod_{k=1}^mX_k$ through which the factorization of $\mA^*$ is made for the determination of the topological centres becomes clear. This will be needed in
our applications in later sections.
For the proof of the second theorem as well as many other theorems in the paper, we shall need
the right multiplier algebra $M_r(\mA)$  of $\mA,$ this is the algebra of all bounded operators $R$ on
$\mA$ satisfying \[ R(ab)=aR(b)\quad\text{for each}\quad a\mbox{ and } b\;\text{ in}\;\mA.\] The operation in $M_r(\mA)$ is the opposite of the operation $\circ$ of composition of operators, i.e., $R_1R_2=R_2\circ R_1$.

We shall also prove in Theorems \ref{diamondcentre} and \ref{ducentre} the analogue of these results for the second toplogical centres.
For this we need the set $X$ to be in the unit ball of $\ruc(\mA)^*$ and we use the left multiplier algebra $M_l(\mA)$ of $\mA.$
  A bounded operator on $\mA$ is a {\it left  multiplier} of $ \mA $ if $ L(ab) = L(a)b$ for $a,b \in \mA$.
The composition of operators is the operation in $M_l(\mA)$.

Recall also that in a Banach algebra $\mA$, a  {\it bounded right approximate identity} (brai for short)
is a  bounded net $(e_\alpha)$ with  \[\lim_\alpha\|ae_\alpha-a\|=0\quad\text{
for each}\quad a\in  \mA.\]
Any weak$^*$-limit of a brai is a right identity in $\mA^{**}$ when $\mA^{**}$ has the first Arens product. Bounded left approximate
identities (blai) are defined analogously. The net $(e_\alpha)$ is a bounded approximate identity
(bai) if it is both brai and blai.

When $\mA$ has a bai bounded by $1$, $\mA$ is a closed right ideal in  $M_r(\mA)$ and a closed left ideal in $M_l(\mA)$. It  will also be seen by Lemmas \ref{Laul} and \ref{Laur} that $M_r(\mA)$ and $M_l(\mA)$  are closed subalgebras of $\luc(\mA)^*$ and $\ruc(\mA)^*,$ respectively.

For more details on $M_r(\mA)$ and $M_l(\mA)$, see \cite{dales}, particularly Theorems 2.5.12 and  2.9.49.

\begin{theorem}\label{Lcentre}
Let $\mA$ be a Banach algebra, $\eta$ be an infinite cardinal and  $X_1, ..., X_m$ be subsets of $\luc(\mA)^*$ each of cardinality at least  $\eta$.
Suppose that $\mA$ has the $\eta$-Mazur property  and its dual $\mA^*$ has the $\eta$-factorization property through $\prod_{k=1}^mX_k$. Then
\begin{enumerate}
\item   for any weak$^*$-closed subsemigroup $\mathscr S$ of $\mA^{**}$ with $\widetilde X\subseteq \mathscr S$, we have $T_1Z(\mathscr S)=\mathscr S\cap \mA.$
\item In particular, $T_1Z(\mA^{**})=\mA$,  that is, $\mA$ is left strongly Arens irregular.
\end{enumerate}
\end{theorem}

\begin{proof}
Let $\Psi\in T_1Z(\mathscr S)$ and  $(S_\lambda )_\lambda$ be a net in the unit ball of $\mA^*$ of cardinality at most $\eta$ with $S_\lambda \to 0$ in the weak$^*$-topology. Factorize $(S_\lambda )_\lambda$ through $\prod_{k=1}^mX_k$ as \[S_\lambda =\sum_{k=1}^m\Psi_{\lambda k} \cdot T_k\]
for each $\lambda<\eta$, where $\{(\Psi_{\lambda k})_{k=1}^m :\lambda<\eta\}$ is a bounded net in
$\prod_{k=1}^mX_k$  and $\{T_k:k=1,...,m\}\subset\mA^{*}$.
Let $\widetilde{\Psi_{\lambda k}}$ be the corresponding extensions in $\widetilde X$,
and take $m$ subnets of $(\widetilde{\Psi_{\lambda k}})_{\lambda<\eta}$ (indexed by the same set) with weak$^*$-limit $\widetilde{\Psi_k}$ in $\mathscr S$ for each  $k=1,...,m.$
Denote for simplicity the $m$ subnets also by  $(\widetilde{\Psi_{\lambda k}})_{\lambda<\eta}$,  $k=1,...,m,$ and keep  $(S_\lambda)$ to denote the corresponding subnet of the original net in the unit ball of $\mA^*$.
It is then easy to check that $\sum_{k=1}^m\widetilde{\Psi_{k}} \cdot T_k$ is the weak$^*$-limit of  the subnet $(S_\lambda)$, and so it is zero.
Since  $\Psi\in T_1Z(\mathscr S)$, using (\ref{trick}), this gives
\begin{equation}\begin{split}\lim_\lambda\langle
\Psi, S_\lambda\rangle&=\lim_\lambda\left\langle \Psi,
\sum_{k=1}^m\Psi_{\lambda k}. T_k\right\rangle
=\lim_\lambda\sum_{k=1}^m\left\langle \Psi.\Psi_{\lambda k},  T_k\right\rangle\\&=\sum_{k=1}^m\lim_\lambda\left\langle \Psi\widetilde{\Psi_{\lambda k}},  T_k\right\rangle=
\sum_{k=1}^m\left\langle \Psi\widetilde{\Psi_{k}},  T_k\right\rangle\\&
=\sum_{k=1}^m\left\langle \Psi,\widetilde{\Psi_{k}}\cdot  T_k\right\rangle=
\left\langle \Psi,\sum_{k=1}^m\widetilde{\Psi_{k}}\cdot  T_k\right\rangle
=0.\label{facto}\end{split}\end{equation}
The $\eta$-Mazur property implies then that $\Psi\in\mA,$ as wanted.
\end{proof}

\begin{theorem}\label{diamondcentre}
Let $\mA$ be a Banach algebra, $\eta$ be an infinite cardinal and  $X_1, ..., X_m$ be subsets of $\ruc(\mA)^*$ each of cardinality at least  $\eta$.
Suppose that $\mA$ has the $\eta$-Mazur property  and its dual $\mA^*$ has the right $\eta$-factorization property through $\prod_{k=1}^mX_k$. Then
\begin{enumerate}
\item   for any weak$^*$-closed subsemigroup $\mathscr S$ of $\mA^{**}$ with $\widetilde X\subseteq \mathscr S$, we have $T_2Z(\mathscr S)=\mathscr S\cap \mA.$
\item In particular, $T_2Z(\mA^{**})=\mA$,  that is, $\mA$ is right strongly Arens irregular.
\end{enumerate}
\end{theorem}
\begin{proof}
Proceed as in the proof of Theorem \ref{Lcentre}.
\end{proof}

For the proof of the second theorem we need the following lemma
due essentially  to Lau and \"Ulger in \cite{laul}.
Let \[I(\mA)=\{\Psi\in \luc(\mA)^*: \mA \Psi\subseteq \mA\}.\]

\begin{lemma}\label{Laul} If $\mA$ is a Banach algebra with a bai $(e_\alpha)$, then $I(\mA)\subseteq T_1Z(\luc(\mA)^*).$
When the bai is bounded by $1$, $M_r( \mA)\cong I(\mA)$.
\end{lemma}

\begin{proof} It is clear that $I(\mA)$ is a norm-closed subalgebra of $\luc(\mA)^*$. To see that $I(\mA)\subseteq T_1Z(\luc(\mA)^*),$ let $\Psi\in I(\mA)$ and $(\Phi_\alpha)$ be a net
in $\luc(\mA)^*$ with weak$^*$-limit $\Phi$.
Then for every $S\in\mA^*$ and $a\in\mA,$ we have
\begin{align*}\lim_\alpha\langle \Psi \Phi_\alpha, S.a\rangle&=\lim_\alpha\langle (a \Psi)\Phi_\alpha, S\rangle=\lim_\alpha\langle (a \Psi),\Phi_\alpha\cdot S\rangle=\lim_\alpha\langle \Phi_\alpha,S\cdot (a \Psi)\rangle\\&=\langle \Phi,S\cdot (a \Psi)\rangle=\langle (a\Psi)\Phi,S\rangle=
\langle \Psi\Phi,S.a\rangle,\end{align*} where the fourth equality is due to the fact that
 $a\Psi\in \mA$.
Therefore, $\Psi$ is in the first topological centre of $\luc(\mA)^*,$ as wanted.

Next we proceed to identify  $I(\mA)$ with $M_r(\mA).$
Note first that  $(e_\alpha)$ is a weak$^*$-convergent net in $\luc(\mA)^*$
and has the identity in $\luc(\mA)^*$ as its limit.
To see this, take any weak$^*$-cluster point $E$ of $(e_\alpha)$ in $\luc(\mA)^*$.
Since $(e_\alpha)$ is a brai, $E$ is a right identity in $\luc(\mA)^*$,
and so $E$ belongs in particular to $I(\mA)$. Since $(e_\alpha)$ is blai, we also have
$Ea=a$ for every $a\in\mA.$ Therefore $E\Psi=\Psi$ for every $\Psi\in\luc(\mA)^*$
since $E\in I(\mA)\subseteq  T_1Z(\luc(\mA)^*).$ Being the identity, $E$ must be the unique
weak$^*$-limit of $(e_\alpha).$

Now if $T\in M_r(\mA)$, then $(Te_\alpha)$ is a net in $\mA$ which is weak$^*$-convergent in $\luc(\mA)^*$.
To see this, take two subnets  $(Te_\beta)$ and $(Te_\gamma)$  with weak$^*$-limits
$\Psi$ and $\Psi'$ in $\luc(\mA)^*$, respectively (this is possible since
 $(Te_\alpha)$ is bounded in $\mA$).

Then again since $(e_\beta)$ is a brai, we have \[Te_\gamma=\lim_\beta T(e_\gamma e_\beta)\quad\text{for each}\quad \gamma,\] and so for every $S.a\in \luc(\mA)$, we have
 \begin{align*}\langle\Psi,S.a\rangle=\langle E\Psi ,S.a\rangle&=\lim_\gamma\lim_\beta\langle e_\gamma Te_\beta,S.a\rangle= \lim_\gamma\lim_\beta\langle T( e_\gamma e_\beta),S.a\rangle\\&=
\lim_\gamma\langle Te_\gamma,S.a\rangle=\langle\Psi',S.a\rangle,
 \end{align*}
as wanted.
 Denote the weak$^*$-limit of $(Te_\alpha)$ by $\Psi$ (note that $\Psi$ depends on $T$).

Consider now for each $\Phi\in I(\mA)$, the right multiplier   $R_\Phi\in M_r(\mA)$ given by
\[R_\Phi(a)=a \Phi\quad\text{ for every}\quad a\in \mA\]
and consider the linear map \[\mathscr I: I(\mA) \to M_r(\mA),\quad \mathscr I(\Phi)=R_\Phi.\]
Since, for every $T\in M_r(\mA)$ and $a\in\mA$, we have
 \begin{equation*}Ta=\lim_\alpha T(ae_\alpha)=\lim_\alpha a Te_\alpha=a\Psi=R_\Psi(a), \label{identify}\end{equation*}
we see that  $\mathscr I$
 is surjective (remember $\Psi$ depends on $T$). Moreover \[\|\mathscr I(\Phi)\|\le \|\Phi\|=\|E\Phi\|=
\lim_\alpha \|R_\Phi e_\alpha\|\le \|R_\Phi\|=\|\mathscr I(\Phi)\|,\] where the second inequality is due to the condition that $(e_\alpha)$  is bounded by $1$.
So when the bai is bounded by $1$, $I(\mA)$ and $M_r(\mA)$ are isometrically isomorphic
and $I(\mA)\subseteq T_1Z(\luc(\mA)^*).$
\end{proof}

\begin{theorem}\label{lucentre} Let $\mA$ be a Banach algebra with a bai, $\eta$ be an infinite cardinal and  $X_1,...,X_m$ be subset of $\luc(\mA)^*$ each with cardinality at least  $\eta$.
If $\mA$ has the $\eta$-Mazur property  and its dual $\mA^*$ has the $\eta$-factorization property through $\prod_{k=1}^mX_k$, then
\begin{enumerate}
\item   for any weak$^*$-closed subsemigroup $\mathscr S$ of $\luc(\mA)^*$ with $ X\subseteq \mathscr S$, we have $T_1Z(\mathscr S)=\mathscr S\cap I(\mA)$.
\item If in addition the bai is bounded by $1$, then   $T_1Z(\mathscr S)$ is isometrically
isomorphic to $\mathscr S\cap M_r(\mA)$.
\item In particular, if the bai is bounded by $1$, then   $T_1Z(\luc(\mA)^*)$ is isometrically
isomorphic to the right multiplier algebra $M_r(\mA)$.
\end{enumerate}
\end{theorem}

\begin{proof}
Let $\Psi\in T_1Z(\mathscr S)$ and $a$ be any element in $\mA.$
Let $(S_ \lambda)_{ \lambda<\eta}$ be any net in $\mA^*$ with weak$^*$-limit $0$ and let its elements be  factorized as in the previous theorem as \[S_ \lambda =\sum_{k=1}^m\Psi_{ \lambda k} \cdot T_k\quad\text{for each}\quad \lambda<\eta,\]
 where $\{(\Psi_{ \lambda k})_{k=1}^m : \lambda<\eta\}$ is a bounded net in
$\prod_{k=1}^mX_k$ and $\{T_k:k=1,...,m\}\subset\mA^{*}$.
Take $m$ subnets of $(\Psi_{ \lambda k})_ \lambda<\eta$ with weak$^*$-limit $\Psi_k$ in $\mathscr S$ for each  $k=1,...,m.$
Denote the subnets also by  $(\Psi_{ \lambda k})_{ \lambda<\eta}$.
It is then easy to check that $\sum_{k=1}^m\Psi_{k} \cdot T_k$ is the weak$^*$-limit of the net $(S_ \lambda)_ \lambda$, and so it is zero. Then, as in (\ref{facto}),
since  $\Psi\in T_1Z(\mathscr S)$, we have
 \begin{align*}\lim_ \lambda\langle
a\Psi, S_ \lambda\rangle   &=  \lim_ \lambda\left\langle a\Psi,
\sum_{k=1}^m\Psi_{ \lambda k}\cdot  T_k\right\rangle= \lim_ \lambda\sum_{k=1}^m\left\langle (a\Psi)\Psi_{ \lambda k},  T_k\right\rangle\\& =\lim_ \lambda\sum_{k=1}^m\langle \Psi\Psi_{ \lambda k},  T_k\cdot a\rangle=
\sum_{k=1}^m\lim_ \lambda\langle \Psi\Psi_{ \lambda k},  T_k\cdot a\rangle
\\&=\sum_{k=1}^m\langle \Psi\Psi_{k},
 T_k\cdot a\rangle=\sum_{k=1}^m\langle a(\Psi\Psi_{k}),  T_k\rangle\\&=
 \langle a\Psi,\sum_{k=1}^m\Psi_{k}\cdot  T_k\rangle
=0.\label{ufacto}\end{align*}
The $\eta$-Mazur property implies then that $a\Psi\in\mA,$ and so $\Psi\in \mathscr S\cap I(\mA)$.
Lemma \ref{Laul} finishes the proof of the theorem.
\end{proof}
\medskip

Let \[J(\mA)=\{\Psi\in \ruc(\mA)^*: \mA\diamond \Psi\subseteq \mA\}.\]

As in Lemma \ref{Laul}

\begin{lemma}\label{Laur} If $\mA$ is a Banach algebra with a bai $(e_\alpha)$, then $J(\mA)\subseteq T_2Z(\ruc(\mA)^*).$
When the bai is bounded by $1$, $M_l( \mA)\cong J(\mA)$.
\end{lemma}

\begin{theorem}\label{ducentre} Let $\mA$ be a Banach algebra with a bai, $\eta$ be an infinite cardinal and  $X_1,...,X_m$ be subset of $\ruc(\mA)^*$ each with cardinality at least  $\eta$.
If $\mA$ has the $\eta$-Mazur property  and its dual $\mA^*$ has the right $\eta$-factorization property through $\prod_{k=1}^mX_k$, then
\begin{enumerate}
\item   for any weak$^*$-closed subsemigroup $\mathscr S$ of $\ruc(\mA)^*$ with $ X\subseteq \mathscr S$, we have $T_2Z(\mathscr S)=\mathscr S\cap J(\mA)$.
\item If in addition the bai is bounded by $1$, then   $T_2Z(\mathscr S)$ is isometrically
isomorphic to $\mathscr S\cap M_l(\mA)$.
\item In particular, if the bai bounded is by $1$, then   $T_2Z(\ruc(\mA)^*)$ is isometrically
isomorphic to the algebra $M_l(\mA)$.
\end{enumerate}
\end{theorem}

\begin{proof}
Proceed as in the proof of Lemma  \ref{Laul} and Theorem \ref{lucentre}.
\end{proof}

\section{F$\ell^1(\eta)$-bases in $\mA$ and factorizations of $\mA^*$, $\luc(\mA)$
and $M_r(\mA)^*$}\label{FF}
 In a given Banach algebra $\mA$ with a bai, we introduce next certain sets of cardinality $\eta$ which will provide $\mA^*$,  as well as $M_r(\mA)^*$  in certain cases, with the $\eta$-factorization property.
These sets are based on the standard $\ell^1$-bases, and are formed using some especially constructed sets in the multiplier algebra $M_r(\mA)$ together with an increasing family of sets in $\mA$.
We shall refer to them as {\it F$\ell^1(\eta)$-bases}, where the letter "F" stands for factorization.
The elements $\Psi_{\lambda k} $ needed  for the factorization of operators in $\mA^*$
 in Definition \ref{wfp}
are taken from the weak$^*$-closure in $\luc(\mA)^*$ of the sets constructed in the multiplier algebra $M_r(\mA)$.

In the three definitions given below, $\eta$ is an infinite cardinal.
We  start in our first definition by adapting the standard definition of an $\ell^1$-base which was originally given by Rosenthal for sequences in \cite{rosen}.

\begin{definition} \label{l1}
{\sc $\ell^1(\eta)$-base}: Let  $\eta$ be an infinite cardinal. In a normed space $\mE$,
a bounded set $\mathscr L:=\{a_\alpha\colon \alpha <\eta\}$ is  an $\ell^1(\eta)$-base with a constant $K>0$
when  the inequality
\[\sum_{n=1}^p |z_n|\le K\left\|\sum_{n=1}^p z_na_{\alpha_n}\right\|\]
holds for all $p\in\N$ and for every possible choice of scalars $z_1,...,z_p$ and  (distinct) elements $a_{\alpha_1},...,a_{\alpha_p}$ in $\mathscr L.$
\end{definition}

\begin{remarks}\label{some} If  $\mathscr L:=\{a_\alpha\colon \alpha <\eta\}$ is an $\ell^1(\eta)$-base with a constant $K>0$ in a normed space $\mE,$ then we note the following:
\begin{enumerate}  \item
For every function $c=(c_\alpha)_{\alpha<\eta}$ in $\ell^\infty(\eta)$, there corresponds $T\in \mE^*$ such that \[\langle T, a_\alpha\rangle=c_\alpha\quad\text{for every}\quad \alpha<\eta.\]
 This is obtained simply by letting
\[\left\langle T, \sum_{n=1}^p z_n a_{\alpha_n}\right\rangle=\sum_{n=1}^p z_nc_{\alpha_n}\quad \text{for every}\quad\sum_{n=1}^p z_n a_{\alpha_n}\in \langle\mathscr L\rangle,\] and noting that
\[\left|\langle T, \sum_{n=1}^p z_n a_{\alpha_n}\rangle\right|=
\left|\sum_{n=1}^p z_n c_{\alpha_n}\right|\le \|c\|_\infty\sum_{n=1}^p |z_n|\le \|c\|_\infty K \left\|\sum_{n=1}^p z_n a_{\alpha_n}\right\|.\]
Therefore we have a continuous linear functional $\langle\mathscr L\rangle\to \C$, which may be extended  by Hahn-Banach Theorem to an element in $\mE^*$.
\item
  The closed linear span of an $\ell^1(\eta)$-base is isomorphic to $\ell^1(\eta).$
\item
The elements $a_\alpha$ in $\mathscr L$ satisfy $\|a_\alpha\|\ge\frac1K$
for every $\alpha<\eta.$
\end{enumerate}
\end{remarks}

Next we relax a bit the definition of $\ell^1(\eta)$-bases. This will enable us to include more
algebras that are strongly Arens irregular in our study.
Observe that if  $\mathscr L_\alpha=\{ a_\alpha\}$ and for some constant $K>0$, the inequality  $\|a_\alpha\|\ge \frac1K$ holds for every $\alpha<\eta,$ then next definition is precisely Definition \ref{l1}. For,
\begin{align*} \sum_{n=1}^p |z_n|=\sum_{n=1}^p \frac{1}{\|a_{\alpha_n}\|}\|z_na_{\alpha_n}\|
\le
K\sum_{n=1}^p \|z_na_{\alpha_n}\|\le K\left\|\sum_{n=1}^p z_na_{\alpha_n}\right\|.\end{align*}
The condition that $\|a_\alpha\|\ge \frac1K$ for some constant $K>0$ and other similar conditions (see (\ref{iv}) and (\ref{v}) in the definitions below)   will be crucial in the study.
The conditions are automatically satisfied when we are dealing with $\ell^1$-bases by  Remark \ref{some} (iii). In general, however,  as the reader will see, these conditions are the thorns in the investigation.

\begin{definition} \label{sel} {\sc $\ell^1(\eta)$-selective base}: In a normed space  $\mE$, we say that a bounded set $\mathscr L$ is an  $\ell^1(\eta)$-selective base with constant $K_1>0$ when
 it is of the form $\bigcup\limits_{\alpha<\eta} \mathscr L_\alpha,$ where $\{\mathscr L_\alpha:\alpha<\eta\}$ is a family of nonempty subsets of $\mE$ such that
\begin{enumerate}
\item the spans $\langle \mathscr L_\alpha\rangle,$ $\alpha<\eta$, have zero pairwise intersections,
\item the inequality
\begin{equation}\sum_{n=1}^p \|a_n\|\le K_1\left\|\sum_{n=1}^p a_n\right\|\label{Csel}\end{equation}holds
for every $\sum\limits_{n=1}^p a_n\in\langle\mathscr L\rangle,$ where $a_n\in\langle\mathscr L_{\alpha_n}\rangle$ for each
$n=1,...,p.$\end{enumerate}
\end{definition}

\begin{remarks} \label{somme}
~

\begin{enumerate}\item
Clearly every $\ell^1(\eta)$-base is an $\ell^1(\eta)$-selective base since we have
\begin{align*}\sum_{n=1}^p \|a_{n}\|=\sum_{n=1}^p \|z_na_{\alpha_n}\|&\le\sup_{\alpha<\eta}\|a_\alpha\| \sum_{n=1}^p |z_n|\\&\le
K\sup_{\alpha<\eta}\|a_\alpha\|\left\|\sum_{n=1}^p z_na_{\alpha_n}\right\|=K\sup_{\alpha<\eta}\|a_\alpha\|\left\|\sum_{n=1}^p a_{n}\right\|\end{align*}for every
$a_n=z_na_{\alpha_n}\in \langle a_{\alpha_n}\rangle$, $n=1,...,p.$

\item
Let  $\mathscr L=\bigcup\limits_{\alpha<\eta} \mathscr L_\alpha$ be an $\ell^1(\eta)$-selective base in a normed space $\mE.$
Pick for each $\alpha<\eta,$ $a_\alpha\in \mathscr L_\alpha$ and consider the set $\{a_\alpha:\alpha<\eta\}.$
Then, as in  Remark \ref{some}(i), every function $c=(c_\alpha)_{\alpha<\eta}$ in $\ell^\infty(\eta)$
 defines a continuous linear functional $T\in \mE^*$ such that \[\langle T, a_\alpha\rangle=c_\alpha\quad\text{for every}\quad \alpha<\eta.\]\end{enumerate}
\end{remarks}

\begin{lemma} \label{selective} Let $\mathscr E$ be a Banach space with an $\ell^1(\eta)$-selective base $\mathscr L=\cup_{\alpha<\eta}\mathscr L_\alpha$ with constant $K_1>0$. Let,  for each $\alpha<\eta,$
%$\mathcal M=\overline{\langle \mathscr L\rangle}$ and
$\mathcal M_\alpha=\overline{\langle \mathscr L_\alpha\rangle}.$
Then, for each $(S_\alpha)_{\alpha<\eta}\in \oplus_\infty\mathcal M_\alpha^*,$ there exists
$T\in \mathscr E^*$ such that $T_{|\mathcal M_\alpha}=S_\alpha.$
\end{lemma}

\begin{proof} Let $(S_\alpha)_{\alpha<\eta}\in \oplus_\infty\mathcal M_\alpha^*.$
Since the spans $\langle\mathscr L_\alpha\rangle$ have zero pairwise intersections, each $x\in\langle \mathscr L\rangle$ is of the form $\sum\limits_{n=1}^pa_{n}$ for some $p\in\N$ and $a_n\in \langle\mathscr L_{\alpha_n}\rangle,$ $n=1,...,p,$ and so we may naturally define $T$  on $\langle \mathscr L\rangle$ by \[\langle T, \sum_{n=1}^pa_{n}\rangle=
\sum_{n=1}^p\langle S_{\alpha_n},a_{n}\rangle.\]
 The map $T$ is well defined and linear on $\langle \mathscr L\rangle.$
To check that it is bounded, let $x=\sum_{n=1}^p a_n\in\langle \mathscr L\rangle,$ where $a_n\in \mathscr L_{\alpha_n}$
 for each $n=1,...,p.$
Then
 \begin{align*}\left|\langle T,x\rangle\right|=\left|\langle T,\sum_{n=1}^p a_n\rangle\right|&
=\left|\sum_{n=1}^p\langle S_{\alpha_n},a_{n}\rangle\right|
\\&\le  \sup_\alpha\|S_\alpha\| \sum_{n=1}^p\| a_n\|\\&
\le K_1\sup_\alpha\|S_\alpha\| \|x\|, \label{key}
\end{align*}
where the last inequality is due to the fact that $\mathscr L$ is an $\ell^1(\eta)$-selective base with constant $K_1.$
This shows that $T$ is bounded on $\langle \mathscr L\rangle,$  and so it has a bounded extension to $\overline{\mathscr L}.$
Extend it then to a bounded functional, denoted again by  $T$, to the whole space by Hahn-Banach theorem. By definition, it is clear that
$T_{|\mathcal M_\alpha}=S_\alpha.$
\end{proof}

Now we come to the main definitions needed in the paper.
Recall that $\mA$ is a right ideal in $M_r(\mA)$ when $\mA$ is faithful, so in the definitions below, the sets $A_\alpha a_\alpha^k$ and $A_\alpha a_\alpha$ are
contained in $\mA$ for each $\alpha<\eta$ and $k=1,...,m$.

\begin{definition} \label{Fil}
 In a Banach algebra $\mA$ with a bai, a zooming set  is a set of the form
 $\mathscr L=\bigcup\limits_{\alpha<\eta} A_\alpha a_\alpha$, where
\begin{enumerate}
\item   $\{A_\alpha\}_{\alpha<\eta}$ is an increasing family of subsets of $\mA$ such that $A_\eta=\bigcup_{\alpha<\eta}A_\alpha$ is bounded (we may assume the bound is $1$)
with its closed linear span equal to $\mA$,
\item   $\{a_\alpha:\alpha<\eta\}$ is  a bounded set in the multiplier algebra $M_r(\mA)$,
\item  the spans $\langle A_\alpha a_\alpha\rangle,$ $\alpha<\eta$, have zero pairwise intersections.
\end{enumerate}
\end{definition}
\bigskip

\begin{definition} \label{Fel}
{\sc F$\ell^1(\eta)$-base of type 1}.  In a Banach algebra $\mA$ with a bai, a zooming $\ell^1(\eta)$-base is called an F$\ell^1(\eta)$-base of type 1.
 \end{definition}

 Note that when $\mathscr L$ is  an F$\ell^1(\eta)$-base of type 1, then  the property
 $1\le K\|aa_\alpha\|$ for every $\alpha<\eta$ satisfied by the elements  $a\in A_\alpha,$ $\alpha<\eta$
(see Remark \ref{some} (iii)),
 is equivalent to
\begin{equation}||a||\le K||aa_\alpha||\quad\text{for every}\quad a\in \langle A_\alpha\rangle\;\text{and}\quad \alpha<\eta.\tag{F2}\label{P}\end{equation}
For, since we are assuming  that $A_\eta$ is bounded by $1,$ we find that for every
 scalars $z_1,...,z_p$ and  (distinct) elements $a_{1},...,a_{p}$ in $A_\alpha,$
\[\|\sum_{1}^pz_na_n\|\le \sum_{n=1}^1|z_n|\le K\left\|\sum_{n=1}^p z_na_na_{\alpha}\right\|.\]
In other words, \[\|a\|\le K\|aa_\alpha\|\quad\text{for every}\quad a\in \langle A_\alpha\rangle.\]

In our applications to the algebraic structure and the topological centres of the Banach algebras
with an Arens product, we will see that F$\ell^1(\eta)$-base of type 1 are available in the group algebra of infinite discrete groups and the semigroup algebra of some infinite discrete semigroups
as well as their weighted versions when the weight is diagonally bounded.
But in general, these F$\ell^1(\eta)$-bases are not available.
The next definition relaxes the condition on being  an $\ell^1(\eta)$-base to an $\ell^1(\eta)$-selective base having Property (\ref{P}) . This second type of F$\ell^1(\eta)$-base will be found in more general situations.

\begin{definition} \label{Fell}
{\sc F$\ell^1(\eta)$-base of type 2}.  In a Banach algebra $\mA$ with a bai, an F$\ell^1(\eta)$-base of type 2 is a zooming $\ell^1(\eta)$-selective base (with constant $K_1>0$)
such that for some constant $K_2>0$,
\begin{equation}||a||\le K_2||aa_\alpha||\quad\text{for every}\quad a\in \langle A_\alpha\rangle\;\text{and}\quad \alpha<\eta.\tag{F2}\label{iv}\end{equation}
\end{definition}

\medskip

In certain cases (see Subsection \ref{nonmetrizable}), although zooming $\ell^1(\eta)$-selective bases are available,  we are not able to check the validity of Inequality (\ref{iv})
above. This has forced us to relax this inequality  to Condition (\ref{v}) below, which is easier to have.
  Condition (iii) in Definition \ref{Fil} for the zooming property means in the definition below
	that the spans $\langle A_\alpha \rangle a_\alpha^k$  have zero pairwise intersections for each $k=1,...,m.$

\begin{definition} \label{Felll}
{\sc F$\ell^1(\eta)$-base of type 3}.  In a  Banach algebra  $\mA$ with a bai,
an F$\ell^1(\eta)$-base of type 3 is  a zooming $\ell^1(\eta)$-selective base (with constant $K_1>0$)  of the form
\[\mathscr L=\bigcup_{\alpha<\eta}\bigcup_{k=1}^m A_\alpha a_\alpha^k=\bigcup_{\alpha<\eta}A_\alpha \{a_\alpha^k:k=1,...m\},\]  where
$\{{\bf a_\alpha}=(a_\alpha^k)_{k=1}^m:\alpha<\eta\}$ is a bounded set in $M_r(\mA)^m$
such that for some $K_2>0$,
\begin{equation}\|a\|\le K_2\|a{\bf a_\alpha}\|\quad\text{for every}\quad a\in \langle A_\alpha\rangle\;\text{and}\quad \alpha<\eta.\tag{F3}\label{v}\end{equation}
Here $a{\bf a_\alpha}=(aa_\alpha^k)_{k=1}^m\in \mA^m$ and $\|a{\bf a_\alpha}\|$ is any of the $p$-norms of the vector $a{\bf a_\alpha}$
in $M_r(\mA)^m$, $p=1,2,...,\infty$.
\end{definition}
\medskip

\begin{remarks}~\normalfont

\begin{enumerate}
\item
 The importance of  Conditions (\ref{iv}) and (\ref{v}) is not    a mere surmise. Already with $m=1$, the condition yields the inequality \[\frac 1{K_2}\|\Phi\|\le \|\Phi\Psi\|\le \|\Phi\|\quad\text{for every}\quad \Phi\in \mA^{**}\] whenever $\Psi$ is weak*-cluster point of $(a_\alpha)$.
When $K_2=1,$ these elements satisfy  $\|\Phi\Psi\|=\|\Phi\|$ for every $\Phi\in \mA^{**}$,  they are called "right isometries" in the literature. They were essential in each of the proofs
on the first topological centres of $\luc(G)^*$ and $L^1(G)^{**}$ given in the past by various authors,
see for instance \cite{grosser-losert}, \cite{Lau86}, \cite{LaLo88}, \cite{FiSa} and \cite{pym-et-al:one-pt}.
They came again in \cite{FiSa-b} studying the first topological centres of weighted algebras, with  the constant $K_2$ as the diagonal bound of the weight (and so not necessarily equal to $1$). Further details can be found in \cite{FiSa}.

The condition (\ref{v}) (and so  (\ref{iv}) when $m=1$) is   essential and needed right at the beginning at  Inequality (\ref{isometry}) in the proof of the factorization Theorem \ref{gfactori}.

\item The right F$\ell^1(\eta)$-bases are defined similarly as in the three definitions above but with $a_\alpha^kA_\alpha$ and $a_\alpha^k\in M_l(\mA),$ $\alpha<\eta$ and $k=1,...,m$.
\end{enumerate}
\end{remarks}

\begin{notation}\label{through}~\normalfont

\begin{enumerate}
\item
 We use the term F$\ell^1(\eta)$-base to designate any of the three types of F$\ell^1(\eta)$-bases.
For an F$\ell^1(\eta)$-base in $\mA$, let $\eta$ have the discrete topology and $\beta \eta$ be the Stone-\v Cech compactification
of $\eta$ and let $\mathscr{C}(\eta)$ be the set of all cofinal ultrafilters on $\eta$. Recall that an ultrafilter
$p$ on $\eta$ is cofinal if every set $I$ in $p$ is cofinal in $\eta$, that is, $\sup I=\eta$ (cf.\ \cite{vanDouwen91}).
 By a result of van  Douwen \cite{vanDouwen91}
\begin{equation*}\label{eq:280509A}
\left|\mathscr {C}(\eta)\right| =2^{2^{\eta}}.
\end{equation*}
Since for each $k=1,...,m$, the  set $\{a_\alpha^k:\alpha<\eta\}$ is bounded in   $M_r(\mA),$
the bounded map
\begin{equation}\label{map}
\eta \to \luc(\mA)^*\colon\quad \alpha \mapsto a_\alpha^k
\end{equation}
has a continuous extension to $\beta \eta$. So for each $k=1,...,m$ and each ultrafilter $p$ in $\beta\eta,$
there corresponds an element $\Psi_{p k}$ in $\luc(\mA)^*.$
(If the elements $a_\alpha^k$ are bounded by $1$, then the elements   $\Psi_{p k}$
are in the closed unit ball in $\luc(\mA)^{*}$.)
For each $k=1,...m,$ we may consider the sets in the unit ball of $\luc(\mA)^{*}$
given by  \[X_k=\{\Psi_{p k}:p\in \mathscr{C}(\eta)\}\quad\text{and let}\quad X=\bigcup_{k=1}^mX_k.\] \medskip

\item When we deal with $\mA^{**}$, as already agreed $\widetilde X$ is the set made of fixed extensions of the elements in $X$ to $\mA^{*}$.
\item  When our concern is with $M_r(\mA)$, instead of the map (\ref{map}), we consider the bounded map
\[
\eta \to M_r(\mA)^{**}\colon\quad \alpha \mapsto a_\alpha^k,
\] and its extension to $\beta\eta.$
So our elements $\Psi_{p k}$ corresponding to the cofinal ultrafilters, as well as the set $X$, shall be regarded in the unit ball of $M_r(\mA)^{**}$.
\end{enumerate}
\end{notation}

\begin{remarks} \label{fellrem}~\normalfont

\begin{enumerate}
\item
As already noted in Remark \ref{somme}(i), every $\ell^1(\eta)$-base is an $\ell^1(\eta)$-selective base.
  Moreover, as noted after Definition \ref{Fel},
Condition (\ref{iv}) in Definition \ref{Fell} is automatically satisfied for an F$\ell^1(\eta)$-base of type 1.
Accordingly, an  F$\ell^1(\eta)$-base of type 1 is also of type 2.

The converse is however not true in general. The F$\ell^1(\eta)$-bases found  in Theorem \ref{ex1} below  are of type 2 but  they are not usually of type 1 since they are not $\ell^1(\eta)$-bases in general. To see this, let in Theorem \ref{ex1} our group be $\R$ and keep the same notation.
With no loss of generality, we may just consider $V_0=]-1,1[$,  $x_0=0$ the identity in $\R,$ and consider the
 sequence $(\varphi_n)_{n<w}$ in $A_0$ be given by \[\varphi_n(x)=\begin{cases} \frac12-\frac1n,\qquad\text{if}\; 0\le x\le\frac12-\frac1n\\x,\qquad\qquad\text{if}\; \frac12-\frac1n\le x\le\frac12+\frac1n\\ \frac12+\frac1n,\qquad\text{if}\; \frac12+\frac1n\le x\le1\\0,\quad\qquad\quad\text{otherwise}.\end{cases}\] Then  for every odd natural number $p,$
\[\sum_{n=1}^p(-1)^{n+1}\varphi_n(x)=\begin{cases} 1-\sum_{n=1}^p(-1)^{n+1}\frac1n,\qquad\text{if}\; 0\le x\le\frac12-\frac1n\\x,\qquad\qquad\qquad\qquad\qquad\text{if}\;\frac12-\frac1n\le x\le\frac12+\frac1n\\
1+\sum_{n=1}^p(-1)^{n+1}\frac1n,\qquad\text{if}\; \frac12+\frac1n\le x\le1\\0,\qquad\qquad\qquad\qquad\qquad\text{otherwise},\end{cases}\] and so
$\left|\left|\sum_{n=1}^p(-1)^{n+1}\varphi_n\right|\right|_1$ is close to $1$ when $p$ is sufficiently large.  Therefore, there no $K>0$ for which \[p \le K\left|\left|\sum_{n=1}^p(-1)^{n+1}\varphi_n\right|\right|_1\]
for every odd natural number $p.$

\item   If together with the bounded set
$\{{\bf a}_\alpha=(a_\alpha^k)_{k=1}^m:\alpha<\eta\}$  in $M_r(\mA)^m$, there is another  bounded  set $\{{\bf b}_\alpha=(b_\alpha^k)_{k=1}^m:\alpha<\eta\}$ in $M_r(\mA)^m$
such that \begin{equation} \label{vi}{\bf a}_\alpha {\bf b}_\alpha=\sum_{k=1}^ma_\alpha^kb_\alpha^k=1\;\;(\text {the identity in}\;\; M_r(\mA))\;\text{for each}\;\alpha<\eta,\tag{f3}\end{equation}
then Condition (\ref{v}) holds easily since
 \[||a||=||a{\bf a}_\alpha {\bf b}_\alpha||\le ||a{\bf a}_\alpha||_p||{\bf b}_\alpha||_q\le
K||a{\bf a}_\alpha||_p
\quad\text{for every}\quad a\in \langle A_\alpha\rangle\; \text{and}\; \alpha<\eta,\]
where $\frac 1p+\frac1q=1$ when $p>1,$ $q=\infty$ when $p=1$, and
 $K=\sup_\alpha\|{\bf b}_\alpha\|_q$.

Therefore a zooming $\ell^1(\eta)$-selective base with Condition (\ref{vi}) is an F$\ell^1(\eta)$-base with constant $K_2=\sup_\alpha\|{\bf b}_\alpha\|$, it is  of type 2 if $m=1$ and of type 3 if $m$ is any positive integer.

Conditions (\ref{v}) and (\ref{vi}) are however not equivalent in general (even for $m=1$) as Condition (\ref{v}) holds in each of the cases studied in Sections \ref{Examples}, \ref{wexamples} and \ref{Aexamples},
 but Condition (\ref{vi}) is clearly not true   in many situations
%Sections \ref{Examples}, \ref{wexamples} and \ref{Aexamples},
starting with Theorem \ref{ex2} when $G$ is not a group.

\item  In most examples in our application sections \ref{Examples}, \ref{wexamples} and \ref{Aexamples},  we will be needing just the lighter Definitions \ref{Fel} and \ref{Fell}.
The rather technical conditions in Definition \ref{Felll} were forced into our approach because of Theorem \ref{ex5} and the results related to this theorem, where we cannot find F$\ell^1(\eta)$-bases of type 1 or 2. \end{enumerate}
\end{remarks}

\begin{remark} \label{luc} Let $\mc{L}=\bigcup\limits_{\alpha<\eta}A_\alpha\{ a_\alpha^k:k=1,...,m\}$ be an F$\ell^1(\eta)$-base in $\mA$, and suppose that  $\mc L$ is of type 1 or 2 if $m=1,$ and it is of type 3, otherwise.
Fix $a\in A_0$.
Then $aa_{\alpha}^k\in A_{\alpha}a_{\alpha}^k$ for each $\alpha<\eta$ and each $k=1,...,m.$ Since $\bigcup\limits_{\alpha<\eta}A_\alpha a_\alpha^k$ is an $\ell^1(\eta)$-selective base, it follows from Remark \ref{somme}(ii),  that for each $k=1,...,m,$ every function $c=(c_\alpha)_{\alpha<\eta}$ in $\ell^\infty(\eta)$   defines a continuous linear functional $T_k\in \mA^*$ such that \[\langle T_k, aa_\alpha^k\rangle=c_\alpha\quad\text{for every}\quad \alpha<\eta.\] In other words, for each $k=1,...,m,$ the function $c\in \ell^\infty(\eta)$ defines an operator $T_k\cdot a$ in $\luc(\mA)$ such that
 \[\langle T_k\cdot a, a_\alpha^k\rangle=\langle T_k, aa_\alpha^k\rangle=c_\alpha\quad\text{for every}\quad \alpha<\eta.\]

Similarly, if $\mc L$ is a right F$\ell^1(\eta)$-base in $\mA$, the function $c\in \ell^\infty(\eta)$ defines for each $k$ an operator $T_k\odot a=a\cdot T_k$ in $\ruc(A)$ such that
 \[\langle a\cdot T_k , a_\alpha^k\rangle=\langle T_k, a_\alpha^k a\rangle=\textsc{}c_\alpha\quad\text{for every}\quad \alpha<\eta.\]
\end{remark}

\medskip

\begin{theorem}\label{gfactori} Let $\mA$ be a Banach algebra with a bai and  an F$\ell^1(\eta)$-base for some infinite cardinal $\eta$ and some $m\in\N$, and let $X_1,...,X_m$ be the subsets in $\luc(\mA)^*$ as defined in Notation \ref{through}. Then $\mA^*$   has the uniform $\eta$-factorization property through $\prod_{k=1}^m X_k.$
\end{theorem}

\begin{proof}  We prove first that $\mA^*$  has the $1$-factorization property, i.e.,
for every $S\in \mA^*$,  there exists a  family of functionals $\{T_k:k=1,...,m\}$ in $\mA^{*}$ such that the factorization formula \begin{equation}\label{1identity}\Psi\cdot T=\sum_{k=1}^m\Psi_k \cdot T_k=S\end{equation} holds
for every element $(\Psi_k)_{k=1}^m$ in $\prod_{k=1}^mX_k$, where $m=1$ if the F$\ell^1(\eta)$-base is of type 1 or 2,  and $m$ is a finite integer if  the F$\ell^1(\eta)$-base is of type 3.

Let  our F$\ell^1(\eta)$-base $\mathscr L$  have the form $\bigcup\limits_{k=1}^m\bigcup\limits_{\alpha<\eta}A_\alpha a_\alpha^k$,
and consider in the Banach algebra $\mA^m$ (the operations are defined coordinatewise),
the set ${\bf{ L}}=\bigcup\limits_{\alpha<\eta}A_\alpha {\bf a_\alpha}$,
where $A_\alpha{\bf a_\alpha}=\{a{\bf a_\alpha}=(a a_\alpha^1,..., a a_\alpha^m): a\in A_\alpha\}$. Since $\mathscr{L}$ is an $\ell^1(\eta)$-selective base in $\mA$, it is easy to check that ${\bf  L}$ is an $\ell^1(\eta)$-selective base in $\mA^m$.

 As in Lemma \ref{selective}, we consider  in $\mA^m$ the subspaces
 \[\mathcal M_{\alpha}= \overline{\langle A_\alpha {\bf a_\alpha}\rangle}=\overline{\langle A_\alpha\rangle}\,{\bf a_\alpha}\;\text{ for each}\;\alpha<\eta.\]

    Let   $S$ be any element in $\mA^\ast$. 	Then, for each $\alpha<\eta$, define $S^{\alpha}$ on $A_\alpha {\bf a_\alpha}$
  by setting  \begin{equation*}\label{change1}\langle S^{\alpha}, a{\bf a_\alpha} \rangle=\langle S, a\rangle\quad\text{for each}\quad a\in A_\alpha.\end{equation*} (We are changing the notation $S_\alpha$ used in Lemma \ref{selective} to
	$S^\alpha$ to avoid confusion with the operators in the family $\{S_\lambda\}_{\lambda<\eta}$
	coming up later on in the proof.) Then,
	we have
	\begin{equation}|\langle S^{\alpha}, a{\bf a_\alpha}\rangle|=|\langle S, a\rangle|\le  \|S\|\|a\|\le K_2\|S\|\|a{\bf a_\alpha}\|\quad\text{for every}\quad a\in \langle A_\alpha\rangle,\label{isometry}\end{equation}
	where $K_2>0$  is the constant required by Inequality (\ref{iv}) if $m=1$ and by (\ref{v}) if $m\in \N$
	is arbitrary.
	
  Thus, each $S^{\alpha}$ is linear and bounded on $\langle A_\alpha \bf a_\alpha\rangle,$ and so it may be extended to an element
  in $\mathcal M_{\alpha }^*$ which we denote again by $S^{\alpha}.$
    We also deduce from the inequality above that  	$(S^{\alpha})_{\alpha<\eta}\in \oplus_\infty\mathcal M_{\alpha }^*.$  Since $\bf L$ is an $\ell^1(\eta)$-selective base in $\mA^m$, it follows from Lemma \ref{selective}, there is  $T=(T_k)_{k=1}^m\in {\mA^m}^*={\mA^*}^m$ 	such that ${T}_{|\mathcal M_{\alpha}}=S^{\alpha}$ for every $\alpha<\eta,$ that is,
\begin{equation}\langle T, a {\bf a_\alpha}\rangle=\langle S^{\alpha}, a{\bf a_\alpha}\rangle=\langle S, a\rangle, \quad\text{for every}\quad   a\in A_\alpha\quad\text{and}\quad\alpha<\eta,\label{c4}\end{equation}
where \begin{equation}\label{dtc}\|T\|\le K_1\sup_\alpha\|S^\alpha\| \le K_1K_2\|S\|.\end{equation}
Let now  $\Psi=(\Psi_k)_{k=1}^m$ be in $\prod_{k=1}^mX_k$,
  We claim  that \[\langle\sum_{k=1}^m\Psi_k \cdot T_k,a\rangle=\langle \Psi\cdot T, a\rangle=\langle S, a\rangle\quad\text{for every}\quad  a\in \mA.\]
Pick 	 a  net  from $\{{\bf a_\alpha}:\alpha<\eta\}$  that weak$^*$-converges to
	$\Psi$ in $(\luc(\mA)^{\ast})^m$.
	This net may be written as $({\bf a}_{h(\gamma)})_{\gamma \in \Lambda}$, where $\Lambda$  is a directed set and  $h\colon \Lambda \to \eta $ is  a monotone cofinal function.
Let $a$ be arbitrary in $A_\eta$.  Then $a\in A_\beta$ for some $\beta<\eta$, and so $a\in A_\alpha$ for every $\alpha\ge \beta.$ Accordingly,
  \begin{align}   \langle \Psi\cdot T , a\rangle=     \langle \Psi, T\cdot  a\rangle=
	\lim_\gamma  \langle  {\bf a}_{h(\gamma)} , T\cdot  a \rangle
=\lim_\gamma  \langle T, a {\bf a}_{h(\gamma)} \rangle.\label{c3}
      \end{align}

Note now that  if $\gamma\in \Lambda$ is such that  $h(\gamma)\ge \beta$, then $a {\bf a}_{h(\gamma)}\in A_{h(\gamma)}{\bf a}_{h(\gamma)},$ and so (\ref{c4}) implies that
\[\langle T, a {\bf a}_{h(\gamma)} \rangle=\langle S^{h(\gamma)}, a{\bf a}_{h(\gamma)}\rangle=\langle S,a\rangle.\]
Thus, continuing the argument in (\ref{c3}), we get
\begin{equation*}\langle\Psi\cdot T, a\rangle=\lim_\gamma  \langle T, a {\bf a}_{h(\gamma)} \rangle =\langle S,a\rangle\quad\text{for every}\quad
a\in A_\eta.\end{equation*}
Since  $\sum_{k=1}^m\Psi_k \cdot T_k$  and $S$ agree on $A_\eta$ and  the linear span of   $A_\eta$ is dense in $\mA,$
we see that $\sum_{k=1}^m\Psi_k \cdot T_k$  and $S$ are equal.

We prove now that $\mA^*$ has the uniform $\eta$-factorization property. So let $\mathcal S=\{S_\lambda\}_{\lambda<\eta}$ be a bounded family of functionals in $\mA^*.$ Then partition $\eta$ into $\eta$ many subsets $I_\lambda$ such that each $I_\lambda$ is cofinal in $\eta$
(see \cite[Lemma, page 61]{vanDouwen91}).
 We see our  F$\ell^1(\eta)$-base as \[\mathscr L=\bigcup\limits_{\lambda<\eta}\bigcup\limits_{\alpha\in I_\lambda} A_\alpha \{a_\alpha^k:k=1,...,m\}\] and its corresponding
$\ell^1(\eta)$-selective base
 in $\mA^m$ as \[\mathscr {\bf L}=\bigcup\limits_{\lambda<\eta}\bigcup\limits_{\alpha\in I_\lambda} A_\alpha {\bf a}_\alpha.\]

Note, that  for each   $\lambda<\eta$, the sets $A_\alpha{\bf a}_\alpha$, $\alpha\in I_\lambda$,  have zero pairwise intersections, and
the family  $\{A_\alpha\}_{\alpha\in I_\lambda}$ is increasing with  $A_\eta=\bigcup_{\alpha\in I_\lambda}A_\alpha$. So we may proceed as in the previous case with the 1-factorization, but this time by taking
for every
$\lambda<\eta,$ the element $\Psi_\lambda= (\Psi_{\lambda k})_{k=1}^m$ in $\prod_{k=1}^mX_k$ as any weak$^*$-cluster point of the set $\{{\bf a}_\alpha:\alpha\in I_\lambda\}$, i.e., based on our Notation \ref{through}, $\Psi_\lambda=\Psi_p$ for some cofinal ultrafilter $p$ having $I_\lambda$ as a member. Then let
 $\{{\bf a}_{h_\lambda(\gamma)}\}_{\gamma \in \Lambda}$ be the net converging to $\Psi_\lambda$ (here $h_\lambda:\Lambda\to I_\lambda$ is monotone and cofinal.)

Then, define for each $\alpha\in I_\lambda$ and $\lambda<\eta,$ $S^\alpha$ on $A_\alpha {\bf a}_\alpha$ by
\begin{equation}\label{op3}\langle S^{\alpha}, a{\bf a}_\alpha\rangle=\langle S_\lambda, a\rangle\quad\text{for every}\quad a\in A_\alpha.\end{equation}

Again applying Lemma \ref{selective}, we obtain as in (\ref{c4}), a functional $T=(T_k)_{k=1}^m\in {\mA^*}^m$
such that \begin{equation}\label{etafac1}\langle T, a {\bf a}_\alpha\rangle=\langle S^{\alpha}, a{\bf a}_\alpha\rangle=\langle S_\lambda, a\rangle \;\text{for every}\;   a\in A_\alpha,\;\alpha\in I_\lambda\;\text{and}\;\lambda<\eta.\end{equation}
We claim that \[\Psi_\lambda\cdot T= S_\lambda\quad\text{for every}\quad\lambda<\eta.\] So fix $\lambda<\eta$, let $a\in A_\eta$ be arbitrary and pick $\beta<\eta$ with $a\in A_\beta$. Then there exists $\gamma_0\in\Lambda$ such that
$h_\lambda(\gamma)\ge\beta$ for every $\gamma\ge\gamma_0$ (possible since $I_\lambda$ is cofinal in $\eta$),
and so $a\in A_{h_\lambda(\gamma)}$ for every $\gamma\ge\gamma_0.$  Since $h_\lambda(\gamma)\in I_\lambda,$
we deduce from (\ref{etafac1}), that
 \begin{equation*}\label{op4}\langle\Psi_\lambda\cdot T,  a\rangle=
\lim_\gamma\langle {\bf a}_{h_\lambda(\gamma)}\cdot T, a\rangle=
\lim_\gamma\langle  T,a {\bf a}_{h_\lambda(\gamma)}\rangle=\langle S_\lambda, a\rangle\quad\text{for every}\quad a\in\mA_\eta.\end{equation*} As previously with the $1$-factorization,
 the wanted factorization \begin{equation*}\label{etafac}\langle\sum_{k=1}^m\Psi_{\lambda k} \cdot T_k,a\rangle=\langle\Psi_\lambda\cdot T,{\bf a}\rangle=\langle S_\lambda, a\rangle\quad\text{ for every}\quad a\in\mA\end{equation*}
 follows for every $\lambda<\eta$.
\end{proof}

\vskip2.5cm

\begin{remarks} \label{ortho}~\normalfont
\begin{enumerate}
\item Let for some $\lambda<\eta$,  $I_\lambda$  be one of the cofinal subsets from the partition of $\eta.$ Then
instead of the element $S^\alpha$ in (\ref{op3}) defined on $A_\alpha$ for $\alpha<\eta$, we
may define  $S^\lambda$ on $\mathcal L=\bigcup_{\alpha<\eta} A_\alpha{\bf a}_\alpha$
by \begin{equation}\label{etafac0}\langle S^{\lambda}, a{\bf a}_\alpha \rangle=\begin{cases}\langle S_\lambda, a\rangle,\quad\text{if}\quad \alpha\in I_\lambda\\
0,\quad\quad\quad\;\text{if}\quad \alpha\notin I_\lambda,\end{cases}\end{equation}
i.e., starting in (\ref{op3}) with $S_\mu=0$ for every $\mu\ne\lambda$.

If this is the case, then the operator $T=(T_k)_{k=1}^m\in {\mA^*}^m$ comes out defined on $\bigcup_{\alpha<\eta} A_\alpha {\bf a}_\alpha$ as
 \begin{equation*}\label{etaFac1}\langle T, a {\bf a}_\alpha\rangle=\begin{cases}\langle S^{\lambda}, a{\bf a}_\alpha\rangle=\langle S_\lambda, a\rangle \quad\text{for every}\quad\alpha\in
I_\lambda\\0,\quad\quad\quad\;\text{if}\quad \alpha\notin I_\lambda,\end{cases}.\end{equation*}
If instead of $T,$ we denote this operator by $T_\lambda$, then
 the 1-factorization identity  (\ref{1identity}) becomes
 \begin{equation*}\label{newproof} \Psi_\mu\cdot T_\lambda=\begin{cases} S_\lambda,\quad\text{if}\quad \mu=\lambda\\0,\quad\;\;\text{otherwise}.\end{cases} \end{equation*}
When $\mathscr L$ is of type 3, one sees from the way the operator $T_\lambda=(T_{\lambda k})_{k=1}^m\in {\mA^*}^m$
is defined that
 \begin{align*}&\Psi_\lambda\cdot T_\lambda=\sum_{k=1}^m\Psi_{\lambda k}\cdot T_{\lambda k} = S_\lambda,\quad\text{but}\quad  \\&\Psi_{\mu k}\cdot T_{\lambda k} = 0\quad\text{for each}\;k=1,...,m,\;\text{if}\; \mu\ne\lambda.\end{align*}
This means that for every $p\in \mathscr C(\eta)$, we have
 \begin{align*}&\Psi_p\cdot T_\lambda=\sum_{k=1}^m\Psi_{p k}\cdot T_{\lambda k} = S_\lambda\quad \text{if}\quad
I_\lambda\in p, \quad\text{but}\quad  \\&\Psi_{p k}\cdot T_{\lambda k} = 0\quad\text{for each}\;k=1,...,m,\;\text{if}\; I_\lambda\notin p.\end{align*}
The same modification may also be applied for the factorization identities in $\luc(\mA)$ and $M_r(\mA)^*$ obtained below in Theorems \ref{lucfactori} and \ref{mfactori}.
We shall make use of these remarks in the theorems in the forthcoming sections studying the algebraic structure of the algebras.

\item
The elements $S^{\lambda}$ defined in (\ref{etafac0}) for the 1-factorization of each $S_\lambda$ can  be used to provide a slightly different proof for the $\eta$-factorization in $\mA^*$. For, let \begin{equation}\label{newS} S=\sum_{\lambda<\eta}S^\lambda.\end{equation}
Then this is well defined on $\mathcal L.$
Furthermore, if $x\in \langle \mathcal L \rangle,$ then
\[x=\sum_{j=1}^p\sum_{n\in I_j} a_na_{\alpha_n},\] where
$I_j=\{n:\alpha_n\in I_{\lambda_j}\}$ and $a_n\in \langle A_{\alpha_n}\rangle$ for $j=1,..., p.$
Accordingly,
\begin{align*}|\langle S, x\rangle|&=
|\langle \sum_{\lambda<\eta} S^{\lambda},  \sum_{j=1}^p\sum_{n\in I_j} a_na_{\alpha_n}\rangle|
=|\langle  \sum_{j=1}^p\sum_{n\in I_j}S_{\lambda_j},   a_n\rangle|
\\&\le \sup\|S_\lambda\| \sum_{j=1}^p\sum_{n\in I_j}\|a_n\|
\le K_2\sup\|S_\lambda\|\sum_{j=1}^p\sum_{n\in I_j}\|a_n a_{\alpha_n}\|
\\&\le K_1 K_2 \sup\|S_\lambda\|\|x\|,\end{align*} where $K_1>0$ is the constant as in Lemma \ref{selective} and $K_2>0$ is the constant required by Inequality (\ref{iv}) if $m=1$ and by (\ref{v}) if $m\in \N$
	is arbitrary.
Since $S$ is a bounded operator on $\langle \mathcal L\rangle,$
we may extend it to $T\in \mA^*$.

Then, using (\ref{etafac0}), we find for every $\lambda<\eta$ and $\alpha\in I_\lambda,$
\[\langle T, a {\bf a}_\alpha\rangle  =\langle S_\lambda, a\rangle .\] Therefore,
 \begin{equation*}\label{tobeused} \langle\Psi_\lambda\cdot T,  a\rangle=
\lim_\gamma\langle {\bf a}_{h_\lambda(\gamma)}\cdot T, a\rangle=\lim_\gamma\langle  T,
 a {\bf a}_{h_\lambda(\gamma)}\rangle
=\langle S_\lambda, a\rangle, \end{equation*}
as wanted.

\item
When a Banach algebra $\mA$ is equipped with an F$\ell^1(\eta)$-base, $\mA^*$ has the  $\eta$-factorization property as it has just been  proved, but we have been unable to deduce that $\luc(\mA)$ satisfies this property as well, that is for instance when $m=1$, whether any bounded family of functionals
$\{S_\lambda:\lambda <\eta\}$ contained in $\luc(\mA)$ can be fatorized as
$\Psi_\lambda\cdot T$ for some bounded family $\Psi_\lambda\in \luc(\mA)^* $ and $T\in\luc(\mA)$
for every $\lambda<\eta.$

We can nevertheless show below that the $1$-factorization is true in $\luc(\mA)$ whenever an F$\ell^1(\eta)$-base is available in the algebra $\mA$, which is enough as we shall see for our study of the algebra and the topological centre of $\luc(\mA)^*$. To be able to deduce the $\eta$-factorization for $\luc(\mA)$ one needs to make sure that
%for the family of operators $T_\lambda$ giving the 1-factorization,
 the extension $T$ of the operator $S$ defined in (\ref{newS}) is in $\luc(\mA)$.
\end{enumerate}
\end{remarks}

\begin{theorem}\label{lucfactori} If a Banach algebra $\mA$ with a bai has an F$\ell^1(\eta)$-base for some infinite cardinal $\eta,$ then $\luc(\mA)$ has the $1$-factorization property  through $\prod_{k=1}^m X_k.$
\end{theorem}

\begin{proof}
Suppose that the F$\ell^1(\eta)$-base is of type 2 (type 1 is then inculded).
 Let $P\in 	\luc(\mA)$ be arbitrary and choose $S\in \mA^*$ and $a\in \mA$ such that $P=S\cdot a$. Then,  factorize $S$ as $S= \Psi\cdot T$, where $\Psi\in X$ and $T\in \mA^*$ are as obtained  in Theorem \ref{gfactori}.
Since $T\cdot a\in \luc(\mA)$ by definition, we only need to check that $P=\Psi\cdot (T\cdot a).$
So let $b\in \mA$ be arbitrary.
Then
\begin{align*}\langle\Psi\cdot (T\cdot a), b\rangle &=\left\langle \Psi,(T.a).b\right\rangle
=\left\langle \Psi,T.(ab)\right\rangle=\left\langle \Psi\cdot T, ab\right\rangle=\left\langle S, ab\right\rangle\\&=\left\langle S.a, b\right\rangle=\left\langle P,b\right\rangle,\end{align*}
 as wanted.

The proof is similar when the F$\ell^1(\eta)$-base is of type 3.
\end{proof}

\medskip

Recall that a dual Banach algebra is a Banach algebra $\mB$ such that $\mB=\mC^*$ for some Banach space $\mC$ (not necessarily unique), which  is a submodule of the dual $\mB$-bimodule $\mB^*$ (see \cite[page 194]{GhaLa}).

The condition imposed below on the net in $\mC$ happens for instance when the net  is a left bounded approximate identity  and $\mC$ is Arens regular.

 Our set $X$ defined in Notation \ref{through} is now seen as a bounded subset in $M_r(\mA)^{**}$.

\begin{theorem} \label{mfactori} Let $\mA$ be a Banach algebra with a bai and an F$\ell^1(\eta)$-base for some infinite cardinal $\eta.$ Suppose that $M_r(\mA)$ is a dual Banach algebra with predual algebra $\mC$ having a net in $\mA\cap \mC$  weak$^*$-converging to a left identity in $M_r(\mA)^*$.
Then $M_r(\mA)^*$ has the $\eta$-factorization property.
\end{theorem}

\begin{proof} Let $\mc{L}:=\bigcup\limits_{\alpha<\eta}A_\alpha\{ a_\alpha^k:n=1,...,m\}$ be an F$\ell^1(\eta)$-base in $\mA$. The case $m=1$ will correspond to both types of the F$\ell^1(\eta)$-bases 1 and 2.
We prove that $M_r(\mA)^*$ has the $1$-factorization property. So let $S$ be any element in $M_r(\mA)^*$ and let $(\Psi_k)_{k=1}^m$ be any element in $\prod_{k=1}^mX_k$.
 Restrict $S$ to $\mA$  and consider the family  $\{T_k:k=1,...,m\}$ in $\mA^{*}$, found in Theorem \ref{gfactori} for the factorization of $S$ so that
  \begin{equation}\label{Afacto}\langle\sum_{k=1}^m\Psi_k\cdot T_{k}, a\rangle= \langle S,a\rangle\quad\text{for every}\quad a\in\mA.\end{equation}
	Then by Hahn-Banach Theorem, regard this family $\{T_k:k=1,...,m\}$ in $M_r(\mA)^*$, and regard $\{\Psi_k:k=1,...,m\}$ in  $M_r(\mA)^{**}$ (the unit ball if the elements $a_\alpha^k$ are bounded by $1$), as explained in Notation \ref{through}.
	Let
  $(\varphi_i)\subset \mC\cap\mA$ be the net in $\mC$ with a left identity $E$ as a weak*-limit in $M_r(\mA)^*.$ Then
 using the fact that $\mA$ is a right ideal in $M_r(\mA),$ we apply (\ref{Afacto}) for the fourth
of the following equalities
\begin{equation*}\label{op5}\begin{split}\langle \sum_{k=1}^m\Psi_k\cdot T_{k}, b\rangle&=\langle E(\sum_{k=1}^m\Psi_k\cdot T_{k}), b\rangle= \lim_i\langle \varphi_i(\sum_{k=1}^m\Psi_k\cdot T_{k}), b\rangle\\&= \lim_i\langle \sum_{k=1}^m\Psi_k\cdot T_{k}, b\varphi_i \rangle=
\lim_i\langle S,b\varphi_i \rangle\\&=\lim_i\langle \varphi_i S, b\rangle=\langle ES,b\rangle=\langle S, b\rangle \end{split}\end{equation*} for every $b\in M_r(\mA),$
as required.

The $\eta$-factorization property is proved following the same steps as in the proof of Theorem \ref{gfactori}.
\end{proof}

\begin{remark} Flipping definitions and conditions from left to right, the reader may check that $M_l(\mA)^*$ has the right $\eta$-factorization property when $\mA$ has a right F$\ell^1(\eta)$-base for some infinite cardinal $\eta.$
\end{remark}

				\section{The algebraic structure of $\mA^{**}$, $\luc(\mA)^*$ and $M_r(\mA)^{**}$} \label{algebra}
				Various algebraic properties were proved in the algebras in harmonic analysis related to
	 locally compact groups $G$ and to some infinite, discrete semigroups
$S$	such as $L^1(G)^{**}$,  $\luc(G)^{*}$, $\ell^\infty(S)^{**}$, $A(G)^{**}$ and  $\uc_2(G)^*$.

Algebraic properties were also found in the Stone-\v Cech compactification $\beta S$ of some discrete semigroups $S$ and the $\luc$-compactification $G^{LUC}$ of a locally compact group $G$. See Notes \ref{notes} in next section.
Most of these properties such as right cancellation, the number of topologically left invariant means or the number  of left ideals in $G^{LUC}$ and $\beta S$ when $S$ is a discrete semigroup are in fact related to the non-Arens regularity of the algebra.
For instance when $G$ is amenable, the number of topologically left invariant means
$\luc(G)^{*}$ or in $L^1(G)^{**}$ and the number of disjoint left ideals in
$G^{LUC}$ are the same. For the support of a (topologically) left invariant,
regarded as a Borel measure on $G^{LUC}$, is a left ideal in $G^{LUC}$.
Conversely, if $L$ is a left ideal in $G^{LUC}$, $x\in L$ and $\Psi$ is a
(topologically) left invariant mean, then  $\Psi x$ is also (topologically) left invariant and
regarded as a Borel measure on $G^{LUC}$, it has its support contained in $L$.

 In fact, already in 1951, Arens himself showed that the semigroup algebra $\ell^1$ with convolution
is non-Arens regular. For this, he produced two distinct invariant means $\Phi$ and $\Psi$ in ${\ell^1}^{**}$ (see \cite{Arens51}). This gave him immediately
 \[\Phi\Psi=\Psi\ne\Phi=\Psi\Phi\quad\text{and}\quad \Phi\diamond\Psi=\Phi\ne\Psi=\Psi\diamond\Phi,\] which means that \[T_1Z({\ell^1}^{**})\ne {\ell^1}^{**} \quad\text{ and }\quad T_2Z({\ell^1}^{**})\ne {\ell^1}^{**}.\]
In 1957, Day used the same argument  in \cite{day} to show that $L^1(G)$ is non-Arens regular for many infinite discrete groups, including all Abelian ones.
This same  approach  was followed by  Civin and Yood in their seminal paper \cite{CiYo} where $L^1(G)$ was shown to be  non-Arens
regular for any infinite locally compact Abelian group.
Their method relied
again on Day's result when the group is discrete.

Once an infinite F$\ell^1(\eta)$-base  is available in $\mA$, the theorems in this section  unify
all these results in $\mA^{**}$, $\luc(\mA)^*$ and $M_r(\mA)^{**}$ and many subsemigroups in these algebras.

We start with right cancellation.
We say that  a set $\Gamma$ in an algebra
 is right cancellable if for every nonzero element $a$ in the algebra , $ax\ne 0$ for some  $x\in \Gamma$. When  $\Gamma=\{x\}$ is right cancellable, we simply say that $x$ is right cancellabe.

\begin{theorem} \label{can} Let $\mA$ be a Banach algebra with a bai and an F$\ell^1(\eta)$-base $\mathscr L$  for some infinite cardinal $\eta$ and some integer $m\in\N.$
Let $X$ be as in Notation \ref{through} and $\mathscr S$ be any subsemigroup  of $\luc(\mA)^*$ with $ X\subseteq \mathscr S,$ or any subsemigroup $\mA^{**}$ with $ \widetilde X\subseteq \mathscr S.$
\begin{enumerate}
\item If $\mathscr L$ is of type 1 or 2, then there are at least $2^{2^\eta}$ many right cancellable elements in $\mathscr S$.
\item If  $\mathscr L$ is of type 3, then  there are at least $2^{2^\eta}$ many right cancellable sets in $\mathscr S$.
\end{enumerate}
\end{theorem}

\begin{proof}  Let $\mc{L}:=\bigcup\limits_{\alpha<\eta}A_\alpha\{ a_\alpha^k:k=1,...,m\}$ be an F$\ell^1(\eta)$-base in $\mA$.
Let $p$ and $q$ be two different cofinal ultrafilters in $\beta\eta$ together with their corresponding elements $(\Psi_{p k})_{k=1}^m$ and $(\Psi_{q k})_{k=1}^m$ in $\prod_{k=1}^mX_k.$
Pick two disjoint members $I$ of $p$ and $J$ of $q.$ For each  $k=1,...,m,$  $\Psi_{p k}$ is the weak$^*$-limit of a subnet of
$(a_\alpha^k)_{\alpha\in I}$, which we denote for simplicity again by $(a_\alpha^k)_{\alpha\in I}$. In the same way, $\Psi_{q k}$ is the weak$^*$-limit of a subnet of $(a_\alpha^k)_{\alpha\in J}$.
Then, for each $\alpha<\eta$ and $k=1,...,m,$ we may take by Remark \ref{luc}, an operator $P_k\in \luc(\mA)$ such that  \[\langle P_k,  a_\alpha^k\rangle=\begin{cases} 1,\quad\text{if}\quad  \alpha\in I\;\\0,\quad\text{elsewhere}.\end{cases}\]
Then, for each $k=1,...,m$,  we have \[\langle \Psi_{p k}, P_k\rangle=\lim_{\alpha\in I}\langle P_k,a_\alpha^k\rangle=1\quad\text{while}\quad \langle \Psi_{q k},P_k\rangle=\lim_{\alpha\in J}\langle P_k, a_\alpha^k\rangle=0.\]
In other words, for each $k=1,...,m,$ the elements $\Psi_{p k}$ and $\Psi_{q k}$ are different in $X_k$.
Thus the map $\mathscr C(\eta)\to X_k:$ $p\mapsto \Psi_{p k}$ is one-to-one for each $k=1,...,m$. Since the cardinality of $\mathscr{C}(\eta)$ is $2^{2^\eta},$ we obtain the same number of points $\Psi_{p k}$ in $X_k$ for each $k=1,...,m.$
This clearly implies that there are at least $2^{2^\eta}$ many elements $(\Psi_{p k})_{k=1}^m$ in
$\prod_{k=1}^mX_k.$
We claim first that each of the sets $\{\Psi_{p k}: k=1,...,m\}$ is right cancellable in $\mathscr S$ when $\mathscr S$ is a subsemigroup of $\luc(\mA)^{*}$.
So let $\Psi$ be any nonzero element in $\mathscr S$ and pick $P\in \luc(\mA)$ with $\langle \Psi, P\rangle\ne0.$ Then by Theorem \ref{lucfactori}, there are operators $T_k\in \luc(\mA)$, $k=1,...,m$, such that $\sum_{k=1}^m\Psi_{p k}\cdot T_k=P.$
Consequently, \[\langle \sum_{k=1}^m\Psi \Psi_{p k}, T_k\rangle=\langle \Psi, \sum_{k=1}^m\Psi_{p k}\cdot T_k\rangle=\langle \Psi, P\rangle\ne0,\] and so $\Psi \Psi_{p k}$ must be nonzero for at least one $k=1,...,m,$ as wanted.

In the case when $\mathscr S$ is a subsemigroup of  $\mA^{**}$, we simply take for each element $\Psi_{p k}\in X_k$ a fixed element  $\widetilde{\Psi_{p k}}\in \mA^{**}$
such that $\pi(\widetilde{\Psi_{p k}})=\Psi_{p k}.$ Then for a nonzero element $\Psi$ in  $\mathscr S$, pick $S\in\mA^*$
such that $\langle\Psi,S\rangle\ne0$ together with the $m$ operators $T_k\in\mA^*$ given by Theorem \ref{gfactori} such that
\[\sum_{k=1}^m\widetilde{\Psi_{p k}}\cdot T_k=\sum_{k=1}^m\Psi_{p k}\cdot T_k=S.\]
As before, this gives
 \[\langle \sum_{k=1}^m\Psi\widetilde{\Psi_{p k}},T_k\rangle=\langle \Psi,\sum_{k=1}^m\widetilde{\Psi_{p k}}\cdot T_k \rangle= \langle \Psi, S\rangle\ne0,\] and so again at least one of the terms $\Psi\widetilde{\Psi_{p k}}$ must be nonzero, showing that the set $\{\widetilde{\Psi_{p k}}:k=1,..,m\}$
is right cancellable in $\mathscr S$ for each $p\in \mathscr C(\eta)$.

This proves Statement (ii) when $m$ is arbitrary, and Statement (i) when taking  $m=1$.\end{proof}

\begin{theorem} \label{ride} Let $\mA$ be a Banach algebra with a bai and an F$\ell^1(\eta)$-base $\mathscr L$  for some infinite cardinal $\eta$ and some integer $m\in\N.$
Let $X$ be as in Notation \ref{through} and $\mathscr S$ be any non-trivial subalgebra  of $\luc(\mA)^*$ with $ X\subseteq \mathscr S,$ or any non-trivial subalgebra of $\mA^{**}$ with $ \widetilde X\subseteq \mathscr S.$
Then
the dimension of any principal right ideal,  and so of any nonzero right ideal, in $\mathscr S$
 is at least $2^{2^\eta}.$
\end{theorem}

\begin{proof} Let $\mathscr{L}$ be an F$\ell^1(\eta)$-base in $\mA$ as in Theorem \ref{can}.

We prove the theorem when $\mathscr S$ is a subalgebra  of $\luc(\mA)^*$, the proof is similar for $\mA^{**}.$
Pick any nonzero element $\Psi$ in a given nonzero right ideal of $\mathscr S$.
For each $p\in\mathscr C(\eta),$ let $(\Psi_{p 1},...,\Psi_{p m})$ be the corresponding elements in $\prod_{k=1}^mX_k$ as in Theorem \ref{can}.
Then we claim that the set \[\{(\Psi \Psi_{p 1},...,\Psi\Psi_{p m}):p\in\mathscr C(\eta) \}\] is linearly independent in the product space $(\Psi \mathscr S)^m$.
Suppose otherwise that for some scalars $c_n$ and distinct cofinal ultrafilters $p_n$, $n=1,\ldots , N$, we have
\[
\sum_{n=1}^N c_n (\Psi\Psi_{p_n 1}, ..., \Psi\Psi_{p_n m})=0.
\]
Then \begin{equation}\label{zero}
\sum_{n=1}^N c_n \Psi\Psi_{p_n k}=0\quad\text{ for every}\quad k=1,...,m.\end{equation}
Let $P\in\luc(\mA)$ such that $\langle\Psi, P\rangle=1.$
Pick  pairwise disjoint members $I_n$ of the ultrafilters $p_n$, $n=1,...,N.$
For each $i=1,...,N,$ pick the operators $T_{k}$ in $\luc(\mA)$, $k=1,...,m$, given by Theorem \ref{lucfactori} as in Remark \ref{ortho}; that is, \[\sum_{k=1}^m\Psi_{p_i k}\cdot T_{k}=P\quad\text{but}\quad\Psi_{p_n k}\cdot T_{k}=0\quad\text{for every}\quad k=1,...,m\quad\text{ whenever}\quad n\ne i.\]
(The operators $T_k$ depend on $i$, but this is not underlined since no confusion is caused.)

Then, for each $i=1,...,N$, we have by (\ref{zero})
\begin{align*}
0=\left\langle\sum_{k=1}^m\left(\sum_{n=1}^Nc_n \Psi\Psi_{p_n k}\right), T_{k}\right\rangle
&=\sum_{k=1}^m\sum_{n=1}^Nc_n\left\langle \Psi,\Psi_{p_n k}\cdot T_{k}\right\rangle
\\&=\sum_{n=1}^Nc_n \left\langle\Psi,\sum_{k=1}^m\Psi_{p_n k}\cdot T_{k}\right\rangle
 \\&=c_i \left\langle\Psi,P\right\rangle =c_i,\end{align*} showing the linear independence of the set
$\{(\Psi \Psi_{p 1},...,\Psi\Psi_{p m}):p\in\mathscr C(\eta) \}$, and so the dimension of the product space
$(\Psi\mathscr S)^m$ must be at least $2^{2^\eta}.$
Accordingly, the dimension of the coordinate space $\Psi\mathscr S$, as well as the right ideal containing $\Psi$,
must be also at least $2^{2^\eta}.$
This completes the proof.
\end{proof}

\medskip

We deal now with the number and dimension of left ideals.
As already observed in \cite{MMM2}, nonzero left ideals in $\mA^{**}$ and $\luc(\mA)^*$ may have finite dimensions. This is the case for example when $\mA$ is the Fourier algebra $A(G)$ of a locally compact group $G$, or the group algebra $L^1(G)$ of an amenable locally compact group $G$. Topologically left invariant means generate, in this case, one-dimensional left ideals in both $\mA^{**}$
and $\luc(\mA)^*$.
When $G$ is a nondiscrete, we find therefore  $2^{2^\eta}$ many such left ideals in the case of the Fourier algebra, where $\eta$ is the local weight of $G$ (see  \cite{chou82} and \cite{Hu95}).
When $G$ is amenable and noncompact,  we found $2^{2^\eta}$ many such left ideals in the case of the group algebra,
 where $\eta$ is the compact covering of $G$ (see \cite{LaPa} or \cite{FiPy03}).

In what follows, we shall  see that there is also an abundance of infinite-dimensional left ideals in the algebras  when an
F$\ell^1(\eta)$-base is available in $\mA.$ In particular, we show that there are  at least $2^{2^{\eta}}$ many left ideals each of dimension at least $2^{2^{\eta}}$ in both algebras $\mA^{**}$ and $\luc(\mA)^*$.
Under the  conditions imposed on $\mA$ in Theorem \ref{mfactori}, these theorems hold also  for the multiplier algebra $M_r(\mA).$

\begin{theorem}\label{lide} Let $\mA$ be a Banach algebra with a bai and an F$\ell^1(\eta)$-base  $\mathscr{L}$ for some infinite cardinal $\eta$ and some integer $m\in\N.$
Let $X$ be as in Notation \ref{through} and $\mathscr S$ be any subsemigroup or a subalgebra of $\luc(\mA)^*$ with $ X\subseteq \mathscr S,$ or any subsemigroup or a subalgebra of $\mA^{**}$ with $ \widetilde X\subseteq \mathscr S.$
\begin{enumerate}
\item If $\mathscr L$ is of type 1 or type 2 and  $\mathscr S$ a subsemigroup, then there are at least $2^{2^{\eta}}$ many disjoint principal left ideals in $\mathscr S$.  If $\mathscr S$ is a subalgebra,
these left ideals have zero pairwise intersections.
\item If $\mathscr L$ is of type 3, then there are at least $2^{2^{\eta}}$ many principal left ideals in $\mathscr S$.
\item If $\mathscr S$ is subalgebra, then the dimension of each of these left ideals is  at least $2^{2^{\eta}}$.
\end{enumerate}
\end{theorem}

\begin{proof} Let $\mathscr{L}:=\bigcup\limits_{\alpha<\eta}A_\alpha\{ a_\alpha^k:k=1,...,m\}$ be an F$\ell^1(\eta)$-base in $\mA$. As in Theorem \ref{can}, the case $m=1$ will correspond to both types of  F$\ell^1(\eta)$-base, 1 and 2.

We prove the theorem when $\mathscr S$ is a subsemigroup or a  subalgebra
of $\luc(\mA)^*$, the proof is similar
in the other case.
Let $p$ and $q$ be two distinct ultrafilters in $\mathscr C(\eta))$ together with their corresponding elements $\Psi_{p k}$  and $\Psi_{q k}$ in $X_k,$ $k=1,...,m.$  When $\mathscr L$ is of type 1 or 2, the corresponding elements in $X$ will be taken as $\Psi_{p}$  and $\Psi_{q}$.

When $\mathscr S$ is a subsemigroup, we consider the left ideals $\mathscr P$ and $\mathscr Q$ given, respectively,  by the sets
\[
 \mathscr S\{\Psi_{p k}\colon k=1,...,m\}\quad\text{and}\quad \mathscr S \{\Psi_{q k}\colon k=1,...,m\}.
\] When $\mathscr S$ is a subalgebra, the left ideals $\mathscr P$ and $\mathscr Q$ are generated by these two  sets.
In both cases of whether $\mathscr S$ is a subsemigroup or a  subalgebra, we prove that neither of these left ideals contains the other.

 Pick any nonzero element $\Psi$ in  $\mathscr S$ together with an operator $P$ in $\luc(\mA)$
 such that
$\langle \Psi , P \rangle \ne 0$.
By Theorem \ref{lucfactori} and Remark \ref{ortho}, let $T_{k}$, $k=1,...,m$, be the operators in $\luc(\mA)$
such that
\[\sum_{k=1}^m\Psi_{p k}\cdot T_{k}=P,\quad\text{but}\quad\Psi_{q k}\cdot T_{k}=0\quad\text{for every}\quad k=1,...,m.\]
So, for all $\Phi \in  \mathscr S$ and all $k=1,...,m,$ we have
\begin{equation*}\langle \Phi  \Psi_{q k}, T_{k}\rangle =\langle \Phi ,  \Psi_{q k}\cdot T_{k}\rangle
=\langle \Phi , 0\rangle =0,\end{equation*} and accordingly,
 \[\langle Q, T_{k}\rangle=0\quad\text{for every}\quad Q\in \mathscr Q\quad\text{ and}\quad k=1,...,m.\]
But since
 \begin{equation*}\sum_{k=1}^m\langle \Psi  \Psi_{p k}, T_{k}\rangle =
\langle \Psi , \sum_{k=1}^m \Psi_{p k}\cdot T_{k} \rangle = \langle \Psi , P \rangle \ne 0,
\end{equation*}
we must have, \[\langle \Psi  \Psi_{p k}, T_{k}\rangle \ne 0\quad\text{for some}\quad k=1, ...,m.\]
Therefore, $\Psi\Psi_{p k}\notin \mathscr Q$, showing that $\mathscr P$ is not contained in $ \mathscr Q$. In the same way, we show that
$\mathscr Q$ is not contained in $\mathscr P$. If $\mathscr L$ is of type 1 or 2, this shows that the principal left ideals $  \mathscr S\Psi_{p}$ and $ \mathscr S \Psi_{q}$ are disjoint when $\mathscr S$ is a subsemigroup, and have their intersection equal to zero when $\mathscr S$ is a subalgebra. If $\mathscr L$ is of type 3, this shows that
the  left ideals $\mathscr P$ and $\mathscr Q$ are different.
 Since there are $2^{2^{\eta}}$  ultrafilters in $\mathscr C(\eta)$, the first two statements follow.

As for Statement (iii), suppose that $\mathscr S$ is a subalgebra of $\luc(\mA)^*$.
We show that the dimension of each of the left ideals of $\mathscr S$ found in the previous statements is at least $2^{2^{\eta}}$. Recall first that by Theorem \ref{can}, the cardinality of the set $X$ is $2^{2^{\eta}}$. Moreover, arguing as in Theorem \ref{ride}, it is quick to see that $X$ is linearly independent. Therefore the dimension of $\mathscr S$ is at least equal to
 $2^{2^{\eta}}$ since  $X\subseteq\mathscr S$.
Now let $\mathscr P$ be any of the left  ideals $\mathscr S \{\Psi_{pk}\colon k=1,...,m\}$ of $\mathscr S,$
where  $p$ is a fixed cofinal ultrafilter in $\beta\eta$ and $\Psi_{pk}$ are the corresponding elements in $X_k$, $k=1,...,m$.
 By Theorem \ref{can}, the map
\[
\mathscr S\to \prod_{k=1}^m\mathscr S: \quad \Psi \mapsto \Psi (\Psi_{p 1},...,\Psi_{pm})
\]
is a linear isomorphism onto its image. Therefore,  the dimension of the linear space
$\prod_{k=1}^m\mathscr S \Psi_{pk}$
 is at least $2^{2^{\eta}}$.
 Clearly, this implies that  for at least one $k=1,...,m$, the linear space  $\mathscr S \Psi_{pk}$
is of dimension at least $2^{2^{\eta}}$. Thus, the dimension of the left ideal generated by
\[\mathscr S \{\Psi_{pk}\colon k=1,...,m\}
\] must be at least  $2^{2^{\eta}}$.
\end{proof}

 \begin{theorem} \label{mall} Let $\mA$ be a Banach algebra with a bai and an F$\ell^1(\eta)$-base $\mathscr{L}$ for some infinite cardinal $\eta$ and some integer $m\in\N.$
Suppose that $M_r(\mA)$ is a dual Banach algebra with predual algebra $\mC$ having a net in $\mA\cap \mC$  weak$^*$-converging to a left identity in $M_r(\mA)^*$.
Regard $X$ as a subset of $M_r(\mA)^{**}$ as agreed in Notation \ref{through} and let $\mathscr S$ be any subsemigroup or a subalgebra of  $M_r(\mA)^{**}$  with $ X\subseteq \mathscr S.$
 \begin{enumerate}
 \item  If $\mathscr L$ is of type 1 or type 2, then there are at least $2^{2^\eta}$ many right cancellable elements in $\mathscr S.$
\item  If  $\mathscr L$ is of type 3, then there are at least $2^{2^\eta}$ many right cancellable sets in $\mathscr S.$
\item If $\mathscr L$ is of type 1 or type 2 and   $\mathscr S$ a subsemigroup, then there are at least $2^{2^{\eta}}$ many disjoint  principal left ideals in $\mathscr S$.  If $\mathscr S$ is  a subalgebra,
these left ideals have zero pairwise intersections.
\item If $\mathscr L$ is of type 3, then there are at least $2^{2^{\eta}}$ many left ideals in $\mathscr S$.
\item If $\mathscr S$ is a non-trivial subalgebra of $M_r(\mA)^{**}$, then the dimension of each of these left ideals is  at least $2^{2^{\eta}}$.
\item If $\mathscr S$ is a  subalgebra of $M_r(\mA)^{**}$,  then the dimension of any principal right ideal,  and so of any nonzero right ideal, in $\mathscr S$ is at least $2^{2^\eta}.$
\end{enumerate}
\end{theorem}

\begin{proof} According to Theorem \ref{mfactori}, under the condition imposed on $\mA,$ an F$\ell^1(\eta)$-base in $\mA$ provides $M_r(\mA)^*$ with the $\eta$-factorization property. In particular, the $1$-factorization together with Remark \ref{ortho} (the   ingredients necessary to prove Theorems \ref{can}, \ref{ride} and \ref{lide})  are also available in this case. So we may proceed in the same way.
\end{proof}

\begin{remark} If $m\ge 2,$ the set $\{\Psi_{pk}: k=1,...m\}$ is right cancellable in the subsemigroups $\mathscr S$ considered in the theorems above, but
 we do not know if each (or any of) the elements $\Psi_{pk}$, $k=1,...m$, is  right cancellable
in $\mathscr S.$

The same problem concerns the left ideals generated by any of the elements $\Psi_{pk}: k=1,...m.$ If $m=1$ or $\mathscr L$ is of type 1 or 2, then these  principal left ideals have trivial intersections. But if $m\ge 2,$ we are only able to show that neither one contains the other.
\end{remark}

\section{Application I: Group algebra, semigroup algebra and measure algebra} \label{Examples}
A number of Banach algebras in harmonic analysis
were proved in the past to be strongly Arens irregular and to carry a rich algebraic structure.
In the rest of the paper,  we see that F$\ell^1(\eta)$-bases are indeed available in {\it almost} all these algebras, and so with the approach developed in the previous sections,  we shall unify almost all what is known so far on the subjects since the introduction by Arens of his two products in 1951.
Some further interesting new cases are also obtained.

\subsection{Some definitions}
Before we proceed with our first application, we summarize first   our  definitions and some properties.
For a locally compact group $G$ with a fixed left Haar measure, we will be considering two cardinal invariants of $G$.
The \emph{local weight }  of $G$ is  the  least cardinality of an open base  at the identity of $G$.
The \emph{compact covering number} of $G$ is the least cardinality of a compact covering of $G$.

 The group algebra $L^1(G)$ is  the Banach space of
all    equivalence classes of scalar-valued functions which are integrable with respect to the Haar measure under convolution, where as usual, two functions are  equivalent if and only if they differ only on a set of Haar measure zero.
   By $L^\infty (G)$ we understand the Banach dual of $L^1(G)$.
This
is the  C$^*$-algebra of all essentially bounded locally
measurable functions on $G$.

When $G$ is an infinite discrete semigroup, we shall also use $L^1(G)$ and $L^\infty(G)$ instead of the usual $\ell^1(G)$ and $\ell^\infty(G)$, respectively.

The C$^*$-algebra
 $\luc(G)$ is the space of
all bounded {\it right uniformly continuous} functions on $G$.
These are the bounded functions on $G$
such that for every $\epsilon>0$,
there exists a neighbourhood $U$ of $e$ such that
\[
|f(s)-f(t)|<\epsilon\quad\text{whenever}\quad st^{-1}\in U.\]
These are also the bounded functions on $G$ which are left norm continuous in the sense that
 \[s\mapsto {}_{s}{f}: G\to \CB(G)\] is continuous, where  $\CB(G)$ is the space of continuous bounded scalar valued functions equipped with its supremum norm and ${}_{s}f$ is the left translate of $f$ by $s\in G.$
  In the literature, this C$^*$-algebra  is denoted  by $C_{ru}(G)$,  $LC(G)$, $\lc(G)$ or  $\luc(G)$. We are using the latter notation.
As known $\luc(G)$  can be obtained as $L^1(G)\ast L^\infty(G)$, and so we have $\luc(L^1(G))=\luc(G)$.

The C$^*$-subalgebra of $\luc(G)$ consisting of functions vanishing at infinity will be  denoted by $C_0(G)$.

The C*-subalgebra of $L^\infty(G)$ consisting of functions   vanishing at infinity will be denoted by $L_0^\infty(G),$
that is, \[L_0^\infty(G)=\{f\in L^\infty(G):\|f\|_{G\setminus K}\to 0\; \text{as}\; K\to G\},\]
where $\|f\|_{G\setminus K}$ is the essential supremum norm of $f$ on $G\setminus K$.

The space
$M(G)$ is the usual measure algebra consisting of
bounded Radon measures on $G$.
The measure algebra $M(G)$ is the dual space of
$C_0(G)$,
and by the well-known Wendel's theorem,  $M(G)$ is the multiplier algebra of $L^1(G)$, that is $M(L^1(G))=M(G)$.

We shall consider the annihilator space of $L_0^\infty(G)$ in $L^1(G)^{**}$ given by
\[L_0^\infty(G)^\perp=\{\Psi \in L^1(G)^{**}: \langle\Psi,f\rangle= 0 \quad\text{for every}\quad f \in L_0^\infty(G)\}.\]
The space
$C_0(G)^{\perp_{lu}}$ is the
annihilator of $C_0(G)$ in $\luc(G)^*$ when $C_0(G)$ is regarded as a subalgebra of $\luc(G)$,
i.e.,
\[C_{0}(G)^{\perp_{lu}} = \{\Psi \in \luc(G)^*: \langle\Psi,f\rangle= 0 \quad\text{for every}\quad f \in C_0(G)\}.\]
The space $C_0(G)^{\perp_M}$  is the annihilator of $C_0(G)$ in $M(G)^{**}$ when $C_0(G)$ is regarded as a subalgebra of $M(G)^*,$
i.e., \[C_{0}(G)^{\perp_M} = \{\Psi \in M(G)^{**}: \langle\Psi,f\rangle= 0 \quad\text{for every}\quad f \in C_0(G)\}.\]
When $G$ is a discrete semigroup, we denote these annihilators simply by $c_0(G)^\perp.$

Recall that we have the following decompositions of $\luc(G)^*,$ $L^\infty(G)^*$ and $M(G)^{**}$ as Banach space direct sums
\begin{align*} \luc(G)^*&= M(G) \oplus C_{0}(G)^{\perp_{lu}},\\
 L^\infty((G)^*&= L_0^\infty(G)^*\oplus L_0^\infty(G)^\perp,\\
 M(G)^{**}&= M(G) \oplus C_{0}(G)^{\perp_M},\end{align*}
where $C_0(G)^{\perp_{lu}}$, $L_0^\infty(G)^\perp$ and $C_0(G)^{\perp_M}$ are weak*-closed ideals in the corresponding algebra.
See \cite{DaLa05, Fi, GhaLa, LauPym90}.

We shall also be concerned with the $\luc$-compactification $G^{LUC}$ of a non-compact locally compact group $G$ and the Stone-\v Cech compactification $\beta G$ of an infinite discrete semigroup $G$.
Both are compact right topological  subsemigroups of $\luc(G)^*$ and $\ell^\infty(G)^*$, respectively.
With the Gelfand topology, they can be produced as the spectra (non-zero multiplicative linear functionals)
of the C$^*$-algebras $\luc(G)$ and $\ell^\infty(G),$ respectively.
They are the largest semigroup compactifications of $G$ with the joint-continuity property.

The canonical morphism  $\epsilon \colon G\to G^{LUC}$  given by evaluations
\[\epsilon(s)(f)=f(s)\; \text{for every}\;f\in \luc(G)\; \text{and}\; s\in G\]
identifies $G$ as a subgroup of $G^{LUC}$.
The same morphism on  $\beta G$, defined with $\ell^\infty(G)$ instead of $\luc(G)$, identifies $G$ as a subsemigroup of $\beta G$ when $G$ is a discrete semigroup.
In both cases,  $G^*$ will denote the remainders $G^{LUC}\setminus G$ and $\beta G\setminus G$.

The reader is directed to \cite{HS} for the algebra in $\beta G$, and to our most recent
investigation in \cite{FG17} on $G^{LUC}$ and other semigroup compactifications.
It may be noted that the spectrum of $L^\infty(G)$
is not a semigroup unless $G$ is compact or discrete (see \cite{LauPym93}), so this is going to be our concern only when $G$ is discrete, i.e., the spectrum is $\beta G. $

\subsection{F$\ell^1(\eta)$-bases}

We start with the group algebra of a non-compact locally compact group.
Variants of the construction below can be found in many articles dealing with the topological centre and the algebraic structure of $\luc(G)^*$ and $L^1(G)^{**}$.

\begin{theorem} \label{ex1} Let $G$ be a non-compact locally compact group with a compact covering $\eta$. Then the group algebra $L^1(G)$
%and the measure algebra $M(G)$ have
has an F$\ell^1(\eta)$-base of type 2.
\end{theorem}

\begin{proof}
Let  $\lambda$ be a fixed left Haar measure and $\Delta$ be the modular function on $G.$
Let
$\{K_\alpha \}_{\alpha<\eta}$ be a compact  covering of $G$ made of symmetric subsets of $G$ with nonempty interiors and with $\lambda(K_\alpha)\le 1$ for each $\alpha$.
Then put \[V_\alpha=\bigcup\limits_{\beta<\alpha} K_\beta\quad\text{for each}\quad\alpha<\eta,\] and observe that $(V_\alpha)_{\alpha<\eta}$ is an increasing family of symmetric subsets of $G$
with nonempty interiors and $\kappa(V_\alpha)\le\alpha$ for each $\alpha<\eta.$
Now choose inductively elements $x_\alpha\in G$ such that
\begin{equation}\label{thin1}
({V_\alpha} x_\alpha) \cap   (V_\beta x_\beta ) = \emptyset\quad\text{ when }\quad\alpha\ne\beta.\end{equation}
This is possible because, for every $\alpha<\eta,$
\[\kappa\left(V_\alpha (\bigcup_{\beta<\alpha} {V_\beta}
x_\beta)\right) \le\alpha.\]

Let now, for each $\alpha<\eta,$ $A_\alpha$ be the set of all continuous positive functions $\varphi$ with
compact supports contained in $V_\alpha$
and with $\|\varphi\|_1=1.$

We check that the conditions in Definition \ref{Fell} are satisfied so that \begin{equation}\label{lonefell}\mathscr L=\bigcup_{\alpha<\eta}(A_\alpha\ast\delta_{x_\alpha})\end{equation} is indeed an F$\ell^1(\eta)$-base of type 2  for $L^1(G)$.

First note that the closed linear span of the set $A_\eta$ is $L^1(G).$ Since the functions in $A_\alpha\ast \delta_{x_\alpha}$
are supported in $V_\alpha x_\alpha,$ it is also clear that the intersection of the linear spans of each pair of  sets $A_\alpha\ast\delta_{x_\alpha}$ is zero.
Let then $\varphi$ be any function in the linear span of $ \mathscr L$
of the form \[\varphi=\sum_{n=1}^p \varphi_n\ast \delta_{x_{\alpha_n}},\]
 where each $\varphi_n$ ($n=1,...,p$) is in the  linear span of $ A_{\alpha_n}$.

 Then, for each $n=1,...,p$, \[\varphi_n\ast\delta_{x_{\alpha_n}}=\Delta(x_{\alpha_n}^{-1})(\varphi_n)_{x_{\alpha_n}^{-1}}\]
 (the right translate of $\varphi_n$ by ${x_{\alpha_n}^{-1}}),$ and so
 \begin{equation}\label{L1}\begin{split}\left\|\sum_{n=1}^p \varphi_n\ast\delta_{x_{\alpha_n}}\right\|_1 &=
\int_G\left|\sum_{n=1}^p\varphi_n\ast\delta_{x_{\alpha_n}} \right| d\lambda =
\int_G\sum_{n=1}^p\left|\varphi_n\ast\delta_{x_{\alpha_n}}\right|d\lambda\\&=
\sum_{n=1}^p\|\varphi_n\ast\delta_{x_{\alpha_n}}\|_1=
\sum_{n=1}^p\|\varphi_n\|_1,\end{split}
\end{equation} where  the second equality is due to the fact that the supports of the functions $\varphi_n\ast \delta_{x_{\alpha_n}}$ are disjoint. The third equality shows that $\mathscr L$ is selective (with $K_1=1$) while the last equality shows the validity of Condition (\ref{iv}) in Definition \ref{Fell} with $K_2=1$. For,  when $p=1$,  the identity (\ref{L1}) simply says that  \[\left\| \varphi\ast\delta_{x_{\alpha}}\right\|_1 =\|\varphi\|_1\quad\text{ for every}\quad\varphi\in \langle A_\alpha\rangle.\]

Note that Condition (\ref{vi}) (and so Condition (\ref{v}) in Definition \ref{Felll}) is also valid here with $m=1$, $a_\alpha$ and $b_\alpha$ may be  simply taken as $\delta_{x_\alpha}$ and $\delta_{x_\alpha^{-1}}$, respectively, for every $\alpha<\eta$.

This ends the proof for $L^1(G)$.
\end{proof}

The same construction gives F$\ell^1(\eta)$-bases when $G$ is a discrete semigroup which respects certain cancellation properties, where $\eta$ is the cardinality of $G.$
 Recall first that a semigroup $G$ is called \textit{weakly cancellative} if the sets
 \[
 s^{-1}t=\{u\in G:su=t\}\quad \text{and}\quad ts^{-1}=\{u\in G:us=t\}
 \]
 are finite for every $s,t\in G$. The semigroup $G$ is  \textit{left (right) cancellative}
 when $s^{-1}t$ ($ts^{-1}$, respectively ) has at most one element.
 We shall use the notations
 \[
 s^{-1}B=\{u\in G:su\in B\}\quad\text{ and}\quad A^{-1}B =
 \bigcup_{s\in A}  s^{-1}B,
 \]
 where $s\in G$ and $A,B\subseteq G$. The sets $Bs^{-1}$ and $BA^{-1}$
 are defined similarly.

 \begin{theorem} \label{ex2} Let $G$ be an infinite discrete weakly cancellative, right cancellative semigroup
  with cardinality
 $\eta$. Then the semigroup algebra $\ell^1(G)$
 has an F$\ell^1(\eta)$-base of type 1.
 \end{theorem}

 \begin{proof}
 Enumerate $G$ as $G=\{ s_\alpha : \alpha<\eta\}$, and let for each $\alpha<\eta,$ \[G_\alpha=\{s_\beta:\beta<\alpha\}.\] Then
$\{G_\alpha\}_{\alpha<\eta}$ is an increasing cover of $G$  with
 $|G_\alpha|= \alpha$ for every $\alpha<\eta.$ Collect then  by induction  a set $\{x_\alpha:\alpha<\eta\}$ such that \begin{equation}\label{thin2}G_\alpha x_\alpha \cap G_\beta x_\beta =\emptyset\quad\text{for every}\quad\alpha\ne\beta .\end{equation} This is possible since $G$
 is weakly left cancellative, and so \[|G_\alpha^{-1}G_\beta|\le\max\{\alpha, \beta\}<\eta.\] Then, let for each $\alpha<\eta,$
$A_\alpha$ consist of all point-mass measures $\delta_x$ with $x\in G_\alpha,$ i.e.,
$A_\alpha=\{\delta_{s_\beta}:\beta<\alpha\}.$
Now by construction,
 the  linear spans of each pair of sets $A_\alpha\ast\delta_{x_\alpha}$ have zero intersection.
Moreover, since $G$ is right cancellative, $xx_\alpha\ne yx_\alpha$ for every $x\ne y$ in $G$ and $\alpha<\eta,$ and so
the supports of the point measures $\delta_x\ast\delta_{x_\alpha}$ (that is $\{xx_\alpha\}$) are disjoint for every $\alpha<\eta$
and $x\in A_\alpha$.
So the elements $\delta_x\ast \delta_{x_\alpha}$ in $\mathcal L$ are all distinct for distinct
pairs $(x,\alpha)\in G_\alpha\times \eta.$
It is then easy to see that \begin{equation}\label{lonefell2}\mathscr L=\bigcup_{\alpha<\eta}(A_\alpha\ast\delta_{x_\alpha})\end{equation} is in fact an $\ell^1(\eta)$-base
in $\ell^1(G)$ (with $K=1$), and so it is an F$\ell^1(\eta)$-base of type 1 in $\ell^1(G)$.
 \end{proof}

\subsection{Algebraic structure}\label{gr} Recall first the following facts.
\begin{enumerate}
\item In $\luc(G)^*$,
non-trivial  left  ideals  of finite  dimensions exist  if and only if $G$ is amenable,
they are generated by elements defined by mixing topological left invariant means (TIMs) together with coefficients of finite dimensional representations of $G$ (see \cite{F99} and \cite{FiMe}). A TIM is a positive linear functional $\Psi$ in $\luc(G)^*$ or $L^1(G)^{**}$ with $\langle \Psi, 1\rangle=1$
and \begin{equation}\label{TIMM}a\Psi=\langle 1,a\rangle\Psi=\left(\int_G a(x)d\lambda\right)\Psi\quad\text{for every}\quad a\in L^1(G).\end{equation}
Since there is one TIM  (the left Haar measure) when $G$ is compact, and there are  $2^{2^\eta}$ TIMs in $\luc(G)^*$  when $G$ is non-compact and amenable,  at least $\max\{|\widehat G_F|,2^{2^\eta} \}$ many distinct left ideals of finite dimension exit in $\luc(G)^*$ when $G$ is infinite and amenable.
Here $\eta$ is the compact covering of $G$ and $\widehat G_F$ is the set of all finite-dimensional
representations of $G$.

\item In $L^1(G)^{**}$, there are two types of left ideals of finite dimension (and the mixture of both). The first type exists whenever $G$ is not discrete. They are generated  by elements $\Psi$ in  $\luc(G)^\perp$. Since $\luc(G)=L^\infty(G)\cdot L^1(G)=L^1(G)\ast L^\infty(G)$, these elements are easily seen to satisfy $L^1(G)^{**}\Psi=\{0\}$.
Since these elements may be seen as the difference of two distinct right identities in $L^1(G)^{**}$, and since there are at least $2^{2^\omega}$ right identities in $L^1(G)^{**}$ (\cite{IPU87}, see also \cite[Remak 4.6.2(3)]{filali-singh}),
we see that the number of left ideals of finite dimension in $L^1(G)^{**}$ is
 at least $2^{2^\omega}$ for any non-discrete group.
The second type of finite-dimensional left ideals (which are not generated by annihilators of $\luc(G)$) exists  if and only if $G$ is amenable; they are generated as in $\luc(G)^*$
by elements defined by mixing TIMs together with coefficients of finite dimensional representations of $G$ (see \cite{F99} and \cite{FiMe}). Since the number of TIMs in $L^1(G)^{**}$is $1$ when $G$ is compact and   $2^{2^\eta}$  when $G$ is non-compact and amenable, we see that the number of finite-dimensional left ideals in  $L^1(G)^{**}$  is  at least equal to $\max\{2^{2^w},|\widehat G_F|,2^{2^\eta} \}$ for any infinite locally compact group.
 For more details, see \cite{F99} and \cite{FiMe}.

\item In $M(G)^{**}$,  A TIM is a positive linear functional $\Psi\in M(G)^{**}$ satisfying the identity (\ref{TIMM}) for every $a\in M(G)$.
TIMs exist in $M(G)^{**}$ if and only if $G$ is amenable, see for example \cite[Corollary 4.3]{Lau83},  \cite[Propositions 1.21 and 5.10]{DLS},
or Proposition 1.4 in our recent article \cite{rideal}.
So again, arguing as for $\luc(G)^*$ in \cite{F99} or applying the general theorem proved in \cite[Theorem 2.8]{FiMe}, non-trivial  left  ideals  of finite  dimensions exist in $M(G)^{**}$ if and only if $G$ is amenable,
they are generated by elements defined by mixing topological left invariant means (TIMs) together with coefficients of finite dimensional representations of $G$.
Since the number of TIMs in $M(G)^{**}$ is $1$ when $G$ is compact and at least $2^{2^\eta}$
 when $G$ is non-compact and amenable, we deduce that at least $\max\{|\widehat G_F|,2^{2^\eta} \}$ many distinct left ideals of finite dimension exit in $M(G)^{**}$ when $G$ is infinite and amenable.
\end{enumerate}

\begin{remark}\label{number}
The number of TIMs in $\luc(G)^*$ and $L^1(G)^{**}$ being $2^{2^\eta}$  when $G$ is non-compact and  amenable  was proved in \cite{LaPa} (see also \cite{FiPy03}), but can also be deduced as follows form the proof of Statements (ii) of Theorem \ref{lalgebra} below. The number of TIMs in $M(G)^{**}$ is also included in our simple argument. This latter result seems to be new.

\begin{proof}
Let $G$ be a non-compact locally compact group with compact covering $\eta$ or an infinite discrete weakly cancellative, right cancellative semigroup
  with cardinality
 $\eta$.  Let $\{x_\alpha:\alpha<\eta\}$ be the set constructed in Theorems \ref{ex1} and \ref{ex2}
and $X$  be as in Notation \ref{through}.
If $G$ is amenable and $\Phi$ is a nonzero TIM in $\luc(G)^*$, $L^1(G)^{**}$
or $M(G)^{**}$, then $\Phi\Psi_1$ and $\Phi\Psi_2$ are easily checked to be TIMs, and Theorem \ref{lalgebra} shows that $\Phi\Psi_1$ and $\Phi\Psi_2$ are distinct whenever $\Psi_1$ and $\Psi_2$ are distinct in $X.$ So the cardinality of $X$, which is $2^{2^\eta}$, equals that of the TIMs.
\end{proof}
\end{remark}

Having F$\ell^1(\eta)$-bases of type 2 and type 1 in the group algebra and the semigroup algebra, respectively, an application of the theorems in Section 3 yield immediately the following theorems
on the algebra in the related semigroup compactifications and the Banach algebras with an Arens product.
This include many results obtained in earlier papers, see Notes \ref{notes} below.
The left ideals provided in the theorems  have infinite dimension and are of different type than those we have just discussed at the beginning of this subsection.

\begin{theorem}\label{lalgebra} Let $G$ be a non-compact locally compact group with a compact covering $\eta$ or an infinite discrete weakly cancellative right cancellative  semigroup with cardinality $\eta$. Let $\mathscr S$ be any of the semigroups $G^{LUC},$ $\beta G$ when $G$ is an infinite weakly cancellative right cancellative discrete semigroup, $G^*,$ or any of the algebras $\luc(G)^*$, $L^1(G)^{**}$, $M(G)^{**}$, $L_0^\infty(G)^\perp$,  $C_{0}(G)^{\perp_{lu}}$
or  $C_{0}(G)^{\perp_M}$. Then
\begin{enumerate}
\item  there are at least $2^{2^\eta}$ many right cancellable elements in $\mathscr S$.
\item   there are at least $2^{2^{\eta}}$ many  principal left ideals in $\mathscr S$.
 These left ideals are pairwise disjoint in the cases of $G^{LUC},$ $\beta G$ and $G^*$, and have  pairwise intersections equal to zero in the other cases.
 \end{enumerate}
\end{theorem}

\begin{proof}
Let  $\{x_\alpha:\alpha<\eta\}$ be the subsets of $G$, needed for the F$\ell^1(\eta)$-bases in $L^1(G)$ of a non-compact locally compact group
	and in $\ell^1(G)$  of an infinite discrete weakly cancellative right cancellative semigroup, constructed in Theorems \ref{ex1} and \ref{ex2}, and let
	$X$ and $\widetilde X$ be as in Notation \ref{through}.
	Then the sets $X$ and   $\widetilde X$ are contained in $G^*$ and $L_0^\infty(G)^\perp,$ respectively, and so are the elements $\Psi_p$ and $\widetilde\Psi_p$,  used in the proofs
	of Statement (i)  of Theorems \ref{can} and  \ref{lide}.
	
	Therefore, a direct application  of Theorems \ref{can} (i)  and \ref{lide} (i) gives the claimed properties when $\mathscr S$ is	$G^{LUC},$ $\beta G$ when $G$ is discrete, $G^*,$ $\luc(G)^*$ or  $C_{0}(G)^{\perp_{lu}}$ since $G^*$ is contained in $\mathscr S$ in each case.

These theorems yield also the claimed properties when $\mathscr S$ is  $L^1(G)^{**}$ or
 $L_0^\infty(G)^\perp$ since  $\widetilde X$ is  contained in  $\mathscr S$ in each case.

The measure algebra $M(G)$ is a dual algebra with predual $C_0(G)$. If $(\varphi_i)\subset C_{00}(G)$ is a bai for $C_0(G)$, then its  weak$^*$-limit in $M(G)^*$ is the identity $1$ (the constant function $1$) in $M(G)^*$ (remember that $C_0(G)$ is Arens regular).
Since again the set $X$ when regarded as a subset of $M(G)^{**}$ is clearly contained in $C_{0}(G)^{\perp_M}$, all the conditions are satisfied for Theorem \ref{mall} (i) and (iii) to hold in $M(G)^{**}$ as well.
\end{proof}

In our next theorem, the dimension of left ideals and right ideals is provided when $\mathscr S$
is an algebra.

\begin{theorem}\label{dlalgebra} Let $G$ be a non-compact locally compact group with a compact covering $\eta$ or an infinite discrete weakly cancellative right cancellative  semigroup with cardinality $\eta$. Let $\mathscr S$ be any of the algebras $\luc(G)^*$, $L^1(G)^{**}$, $M(G)^{**}$, $L_0^\infty(G)^\perp$,  $C_{0}(G)^{\perp_{lu}}$
or  $C_{0}(G)^{\perp_M}$. Then
\begin{enumerate}
\item  the dimension of any principal right ideal,  and so of any nonzero right ideal, in $\mathscr S$
 is at least $2^{2^\eta}.$
 \item  the dimension of each of the $2^{2^{\eta}}$ principal left ideals found in Theorem \ref{lalgebra} in $\mathscr S$   is  at least $2^{2^{\eta}}$.
 \end{enumerate}
\end{theorem}

\begin{proof} Let again $\{x_\alpha:\alpha<\eta\}$, $X$ and $\widetilde X$ be as in the previous theorem. Since $X\subseteq G^*$ and $\widetilde X\subseteq L_0^\infty(G)^\perp$, a direct application  of  Theorem  \ref{ride}  and Statement (iii) in Theorem \ref{lide}  gives the claimed properties when $\mathscr S$  is any of the algebras $\luc(G)^*$, $L^1(G)^{**}$, $L_0^\infty(G)^\perp$  or $C_{0}(G)^{\perp_{lu}}$.

For $M(G)^{**}$ and $C_{0}(G)^{\perp_M}$, we apply  Theorem \ref{mall}
since $X\subseteq C_{0}(G)^{\perp_M}.$
\end{proof}

\begin{remark}
The number of right ideals is more difficult. Our approach  in this paper does not apply when counting the number of right ideals in the semigroups and algebras.
But at this point, we may direct the reader to \cite{BM}, \cite{HS} and \cite{Z}, where the number of right ideals in $\beta G$ and $G^{LUC}$ was computed in certain cases.
We shall return for further developments on this problem in our next paper \cite{rideal}.
\end{remark}

\begin{notes}\label{notes}
For the right cancellation property in $\beta G$ when $G$ is an infinite discrete  weakly cancellative right cancellative  semigroup, the reader is directed to   \cite{Fi1999},
%\cite{Fi},
\cite{F} and \cite[Chapter 8]{HS}.
For  the right cancellation property in $G^{LUC}$ when $G$ is a non-compact locally compact group,
the reader is directed to \cite{FCamb}, \cite{FiPy03}
 and  \cite{FiSa07}.
% and \cite{FiSa07a}.
	For the number of left ideals in $\beta G$ when $G$ is an infinite discrete  weakly cancellative right cancellative  semigroup, the reader is directed to
 \cite{chou69},    \cite{Fi}, \cite{Fi1999},  \cite{F},  \cite{FiSa07a} and \cite[Chapter 6]{HS}.
For the number of left ideals in $G^{LUC}$ when $G$ is a non-compact locally compact group,
the reader is directed to \cite{F}, \cite{FiPy03},
  \cite{FiSa07},  \cite{FiSa07a} and
 \cite{LMP}.

 For the right cancellation property  in $\luc(G)^*$ and $L^1(G)^{**}$,   the reader is directed to \cite{FiSa07}.
For  the number of left ideals in the algebras $\luc(G)^*$ and $L^1(G)^{**}$ the reader is directed to \cite{FiPy03} and
  \cite{FiSa07}.
	\end{notes}

\subsection{Topological centres}
We look now at the topological centres of $L^1(G)^{**}$, $\luc(G)^*$ and $M(G)^{**}$ when $G$ is a  locally compact group, of $\ell^1(G)^{**}$ when $G$ is weakly cancellative right cancellative semigroup, and many of their subsemigroups.

Recall  that the group algebra of  any locally compact group with a compact covering $\eta$ has Mazur property of level $\eta.w$ (see  \cite[Corollary 3.3.(i)]{HuNe06}) or \cite[Theorem 3.4]{Neufang08}).

The measure algebra of any locally compact group $G$ has Mazur property of level $|G|.w$ (see \cite[Corollary 5.6(i)]{HuNe06}), and $|G|=\eta.2^{w(G)}$ (see \cite[Theorem 3.12(iii)]{Comfort} or \cite[Corollary 5.5]{HuNe06}). Recall also from \cite[Corolary 5.6]{HuNe06} that $M(G)$ has Mazur property of level $w$ (i.e., the classical Mazur property) if and only if $|G|$ is non-measurable.  For more information and references related to measurable cardinals, see \cite[page 219]{Neufang05}.

With the F$\ell^1(\eta)$-bases found in Theorems \ref{ex1} and \ref{ex2},
Theorems \ref{Lcentre} and \ref{lucentre} yield
the following theorems, which summarizes some of the main results proved on topological centres in harmonic analysis in the last thirty years, and includes some new results as well.

\begin{theorem} \label{centres} Let $G$ be a locally compact group with compact covering $\eta$ and local weight $w(G)$.
Then \begin{enumerate}
\item the first topological centre of $L^1(G)^{**}$ is $L^1(G)$.
\item the first topological centre of $\luc(G)^*$ is $M(G)$.
\item the first topological centre of ${L_0^\infty}(G)^\perp$ is $\{0\}$.
\item the first topological centre of $C_{0}(G)^{\perp_{lu}}$ is $\{0\}$.
\item the first topological centre of $G^{LUC}$ is $G$.
 \item the first  topological centre of $G^*$ is empty.
\item the first topological centre of $M(G)^{**}$  is $M(G)$ whenever $\eta\ge 2^{w(G)}$ or $|G|$ is non-measurable.
\item the first topological centre of $C_{0}(G)^{\perp_M}$ is $\{0\}$ whenever $\eta\ge 2^{w(G)}$ or $|G|$ is non-measurable.
\end{enumerate}
\end{theorem}

\begin{proof} When $G$ is compact, the first statement was proved in \cite{IPU87}, and it may also be deduced from
\cite[Theorem 2.2]{balapy}. The five statements which follow are trivial when $G$ is compact.
So let us assume that $G$ is non-compact, and let $\{x_\alpha:\alpha<\eta\}\subset G$ be as in Theorem \ref{ex1}, 	$X$ and $\widetilde X$ be as in Notation \ref{through}.
Having an F$\ell^1(\eta)$-base of type 2 in the group algebra,
Theorem \ref{gfactori} and Theorem \ref{mfactori} provide $L^\infty(G)$ and $M(G)^*$  with the $\eta$-factorization property whenever $G$ is a non-compact locally compact group.

Since the $\eta$-factorization property and the $\eta$-Mazur property are both available in $L^\infty(G)$
and $L^1(G)$, respectively, the first two statements follow directly from Theorems \ref{Lcentre} (ii)  and \ref{lucentre} (iii).

Now since $X\subseteq G^*\subseteq C_{0}(G)^{\perp_{lu}}$
	in the case of $\luc(G)^*$ and
 $\widetilde X\subseteq {L_0^\infty}(G)^\perp$ in the case of $L^1(G)^{**}$,
Statements (iii)--(vi) follow also directly from  \ref{Lcentre}
 and \ref{lucentre}  and the facts that
 \[L_0^\infty(G)^\perp\cap L^1(G)=C_{0}(G)^{\perp_{lu}}\cap M(G)=\{0\},\]
\[G^{LUC}\cap M(G)=G\quad\text{ and}\quad G^*\cap M(G)=\emptyset.\]

Finally, note that under the conditions imposed in Statements (vii) and (viii),
 $M(G)$ has  the required Mazur property, which is
of level $\eta$ when $\eta\ge 2^{w(G)}$ and of level $w$ when $|G|$ is non-measurable.
Note also that in the latter case, $M(G)^*$ has in particular by Theorems \ref{mfactori} and \ref{ex1} the $w$-factorization property (since up to $\eta$-many functionals in $M(G)^*$ may be factorized).
Therefore Statement (vii) follows from Theorem \ref{Lcentre}.

Now since  $C_{0}(G)^{\perp_M}\cap M(G)=\{0\}$ and $X\subseteq C_{0}(G)^{\perp_M}$ (when regarded as a subset of $M(G)^{**}$),
the last statement is also a corollary of Theorem  \ref{Lcentre}.
\end{proof}

Having F$\ell^1(\eta)$-bases of type 1 in the semigroup algebra,  we also obtain in the same way the first topological centres when $G$ is an infinite discrete weakly cancellative right cancellative  semigroup.

\begin{theorem} \label{naughtcentres} Let $G$ be an infinite discrete weakly cancellative right cancellative  semigroup.
Then \begin{enumerate}
\item the first topological centre of $\ell^1(G)^{**}$ is $\ell^1(G)$.
\item the first topological centre of $c_0(G)^\perp$ is $\{0\}$.
\item the first topological centre of $\beta G$ is $G$.
\item the first topological centre of $G^*$ is empty.
\end{enumerate}
\end{theorem}

\begin{remark}\label{mirrors} The mirror theorems of Theorem \ref{centres} and \ref{naughtcentres} are also
true; that is, if "first" is replaced by "second" and "right"  replaced by "left" in these theorems, then the conclusions stay unchanged. To see this when $G$ is not compact,
one  needs to return to Theorems \ref{ex1} and \ref{ex2} and define inductively the
sets $\{x_\alpha:\alpha<\eta\}\subset G$ with the property
\begin{equation*} x_\alpha G_\alpha\cap  x_\beta G_\beta  =\emptyset\quad\text{for every}\quad\alpha\ne\beta \end{equation*}
instead of (\ref{thin1}) and (\ref{thin2}).
Then consider the family of functions given by
\begin{equation*}\mathscr L=\bigcup_{\alpha<\eta}(\delta_{x_\alpha}\ast A_\alpha)\end{equation*}
instead of (\ref{lonefell}) and (\ref{lonefell2}).
Precisely as in Theorems \ref{ex1} and \ref{ex2}, this gives a right F$\ell^1(\eta)$-base
in $L^1(G)$. Following the steps of the proof of Theorem \ref{gfactori}, we see then that
$L^\infty(G)$ has the right $\eta$-factorization according to Remark \ref{mirror}.
Then we are only left to apply Theorems \ref{diamondcentre} and \ref{ducentre}.

When $G$ is compact, we may use  \cite[Theorem 2.2(ii)]{balapy} where  Baker, Lau and Pym proved that $\mathcal Z_1(\mA^{**})=\mA$ if $\mA$ is weakly sequentially complete, contains a bounded  approximate identity and is an ideal in its second dual.
Their proofs can also be applied to show that $\mathcal Z_2(\mA^{**})=\mA$. So  these algebras,
in particular $L^1(G)$ when $G$ is compact,  are sAir by \cite[Theorem 2.2(ii)]{balapy}.
\end{remark}

\begin{corollary}
 Let $G$ be a locally compact group with compact covering $\eta$ and local weight $w(G)$,
or an infinite discrete weakly cancellative right cancellative  semigroup.
Then
\begin{enumerate}
\item $L^1(G)$ is strongly Arens irregular.
\item $M(G)$ is strongly Arens irregular whenever $\eta\ge 2^{w(G)}$ or $|G|$ is non-measurable.
\end{enumerate}
\end{corollary}

\begin{notes}\label{smalleta}~\begin{enumerate}
\item The compact case in Statement (i) of Theorem \ref{centres} is not a consequence of our approach (see Final remarks \ref{finalremark}(ii) below). This statement is due to Isik,  Pym and \"Ulger (\cite{IPU87}). The general Statement in (i) is due to Lau and Losert (\cite{LaLo88}).
 Statement (ii)  is due to Lau (\cite{Lau86}).
For non-compact groups, these results were also reached in \cite{dales-lau-strauss},  \cite{FiSa07a} and \cite{Neufang04B}.
Statement (iii) and (iv) were proved in \cite{FiSa07a}.
Statement (v) was proved originally by Lau and Pym in
 \cite{LauPym95}, then by Protasov and Pym in  \cite{PyPo} and by Filali and Salmi in \cite{FiSa07a}.
Statement (vi) was proved by Protasov and Pym in \cite {PyPo} and Filali and by Salmi in \cite{FiSa07a}.

Statement (vii) is due to Neufang (\cite{Neufang05}.)
The general statement that the topological centre of $M(G)^{**}$  is $M(G)$ without any condition has been recently solved by Losert et al in
\cite{Loetal}. It is worthwhile to note that a careful look at the constructions made in \cite{Loetal}
may help to construct an F$\ell^1(|G|.w)$-base in $M(G)$. With Theorem  \ref{mfactori} and Theorem \ref{Lcentre}, this would give another proof
to the latter result. We hope to get back to this problem in the future. Statement (viii) is new.

\item
In Theorem \ref{naughtcentres},  Statement (i) was proved by Lau in \cite{Lau86}
under a slightly weaker cancellation property of the semigroup $G$,   and reached also in \cite{FiSa07a}.
 Statement (ii) was proved in \cite{FiSa07a}.
Concerning Statements (iii) and (iv),  Davenport and Hindman  proved  in \cite{DHi}, that the algebraic centre of $\beta G$ coincides with that of $G$, and the algebraic centre of $G^*$ is empty, when $G$
is an infinite discrete cancellative semigroup. Under  a weaker cancellation property of $G,$
these two statements were proved again by Hindman and Strauss in \cite[Theorem 6.54]{HS}.
Statements (iii) and (iv) were proved in \cite{FiSa07a}.
\end{enumerate}
\end{notes}

\section{Application II: Weighted group, semigroup and measure algebras}\label{wexamples}
 Our approach applies also to the weighted (semi)group algebra and the weighted measure algebra when the weight is diagonally bounded.

\subsection{Some definitions}
Recall that a \emph{weight} on  a semigroup $G$ (with a topology) is a continuous function
$w: G\to (0,\infty)$ which is submultiplicative, that is,
\[
w(st)\le w(s)w(t)\quad\text{for every}\quad s,t\in G.
\]
If $G$ is a group with identity $e$, we shall assume that $w(e)=1.$
 The weight function $w$ is called
\textit{diagonally bounded} on $G$ if there exists $K>0$ such that
\[
w(s)w(t)\leq Kw(st)\quad\text{for every}\quad s,t\in G.\]

For the analogues of the theorems proved in the previous section,
we recall the following weighted spaces:
\begin{itemize}
\item
$L^\infty(G, w^{-1})$ consists of all Borel measurable functions $S$ on $G$ such that
$  w^{-1}S \in L^\infty(G)$.
\item $L_0^\infty(G, w^{-1})$ is the subspace of those functions $S\in L^\infty (G, w^{-1})$ such that
$w^{-1}S\in L_0^\infty(G)$.
\item $\luc(G, w^{-1})$ is the subspace of those functions $S\in L^\infty (G, w^{-1})$ such that
$w^{-1}S\in \luc(G)$.
\item $C_0(G,w^{-1})$ is the subspace of  those functions $S\in \luc(G,w^{-1})$  such that
$w^{-1}S\in C_0(G)$.

The norm of $S$ in these weighted spaces  is given by $\|w^{-1}S\|$ so that
the space and its weighted space are isometric.
\medskip

\item The weighted group algebra $L^1(G,w)$
consists of all Borel measurable functions $a$ on $G$ such that $wa\in L^1(G).$
\item
The weighted semigroup algebra $\ell^1(G,w)$ is defined in the same way when $G$ is a discrete semigroup.
\item
The weighted measure algebra   $M(G, w)$ consists of all regular compacted-Borel measures $a$ on $G$
such that $wa\in M(G).$

The norm of $a$ in these weighted spaces  is given by $\|wa\|$, so that
the space and its weighted space are isometric. The definition of $M(G,w)$ has been rectified  recently by Stokke in \cite{ross} by adding the adjective "compacted".

As it is the case when $w=1,$
both $L^1(G, w)$ and $M(G, w)$ are
Banach algebras under convolution since $w$ is submultiplicative,
 and so
 the Arens products may be defined on  $L^1(G,w)^{**}$ and   $M(G, w)^{**}.$
Moreover, as known  $L^1(G,w)$ is a closed two-sided ideal
of $M(G,w)$. The weighted Wendel's theorem is also true, so that $M(G,w)$ is the multiplier algebra of $L^1(G,w)$, see \cite[Lemma 2.3]{Ghah} and \cite[Theorem 3.2]{safoura}.

The spaces $L^\infty(G, w^{-1})$  and  $L^1(G,w)^*$ are identified by the usual dual pairing. And since \[\luc(G, w^{-1})=L^\infty(G, w^{-1})\cdot L^1(G, w)\] (see \cite{Gr} for the thorem and \cite[Proposition 7.15]{DaLa05} for the corrected proof), the actions of $\luc(G, w^{-1})^*$ on $L^\infty(G,w^{-1})$
and $L^1(G,w)^{**}$ may be defined as previously.
\end{itemize}

The Banach space decompositions used in the previous section  are true
in the weighted situation, the proofs are as in  \cite{DaLa05, Fi, GhaLa, LauPym90}.
\begin{equation}\label{decompo}\begin{split} \luc(G,w^{-1})^*&= M(G,w) \oplus C_{0}(G,w^{-1})^{\perp_{lu}},\\
 L^\infty((G, w^{-1})^*&= L_0^\infty(G,w^{-1})^*\oplus L_0^\infty(G,w^{-1})^\perp,\\
 M(G,w)^{**}&= M(G,w) \oplus C_{0}(G,w^{-1})^{\perp_M},\end{split}\end{equation}
where $C_0(G,w^{-1})^{\perp_{lu}}$, $L_0^\infty(G,w^{-1})^\perp$ and $C_0(G,w^{-1})^{\perp_M}$ are  weak*-closed ideals in the corresponding algebras.
The proofs are easily adapted from the unweighted cases.

We shall not be concerned as in the previous section with the spectrum of $\luc(G,w^{-1})$.
This is not a subsemigroup of $\luc(G,w^{-1})$ unless $w$ is a homomorphism,
in which case this spectrum is isomorphic to $G^{LUC}$ as a semigroup compactification.
For the same reason, the spectrum of $\ell^\infty(G,w^{-1})$ when $G$ is an infinite discrete semigroup is not included in our theorem.
%\end{document}

\subsection{F$\ell^1(\eta)$-bases}
\begin{theorem} \label{ex3} Let $G$ be a non-compact locally compact group with compact covering $\eta$ and let $w$ be a weight on $G$
that is diagonally bounded. Then the weighted group algebra $L^1(G,w)$
 has an F$\ell^1(\eta)$-base of type 2.
\end{theorem}

\begin{proof}
Let $\{V_\alpha\}_{\alpha<\eta}$, $\{x_\alpha\}_{\alpha<\eta}$ and $\{A_\alpha\}_{\alpha<\eta}$
be as in Example \ref{ex1} and let  \[\mathscr L=\bigcup_{\alpha<\eta} \left(\frac{A_\alpha}{w}\right)\ast \frac{\delta_{x_\alpha}}{w(x_\alpha)}.\] We claim that  $\mathscr L$
is an F$\ell^1(\eta)$-base in $L^1(G,w)$. To see this, we only need to check that the analogue of (\ref{L1}) in Example \ref{ex1} is true, the rest is clear and follows as in Example \ref{ex1}.

Note first that for every $\varphi\in L^1(G),$ we have
\[\frac{\varphi}{w}\ast \frac{\delta_{x_\alpha}}{w(x_\alpha)}=
\frac1{w(x_\alpha)}{\Delta({x_\alpha^{-1}})\left(\frac{\varphi}{w}\right)_{x_\alpha^{-1}}}\quad(\text{the right  translate of}\; \frac{\varphi}{w}\;\text{by}\;x_\alpha^{-1}).\]
Let now  $\varphi$ be a function in the closed linear span of $ \mathscr L$. Then $\varphi$ has
 the form \[\varphi=\sum_{n=1}^p \frac{\varphi_n}{w}\ast \frac{\delta_{x_{\alpha_n}}}{w(x_{\alpha_n})},\]
 where each $\varphi_n$ ($n=1,...,p$) is in the closed linear span of $ A_{\alpha_n}$.
 Accordingly,
\begin{align*}
\left\|\sum_{n=1}^p \left(\frac{\varphi_n}{w}\ast\frac{\delta_{x_{\alpha_n}}}{w(x_{\alpha_n})}\right)\right\|_{1,w}=&
\left\|\sum_{n=1}^p  w\left(\frac{\varphi_n}{w}\ast\frac{\delta_{x_{\alpha_n}}}{w(x_{\alpha_n})}\right)\right\|_1 \\&=
\sum_{n=1}^p
\int_G w\left|\left(\frac{\varphi_n}{w}\ast\frac{\delta_{x_{\alpha_n}}}{w(x_{\alpha_n})}\right)\right|d\lambda\\&
=
\sum_{n=1}^p \left\|\frac{\varphi_n}{w}\ast\frac{\delta_{x_{\alpha_n}}}{w(x_{\alpha_n})})\right\|_{1,w}\\&=
\sum_{n=1}^p \int_G\left|w \frac1{w(x_{\alpha_n})}{\Delta({x_{\alpha_n}^{-1}})
\left(\frac{\varphi_n}{w}\right)_{x_{\alpha_n}^{-1}}}\right|d\lambda\\&=
\sum_{n=1}^p\int_G\left| \frac{w(yx_{\alpha_n})}{w(y)w(x_{\alpha_n})}\varphi_n(y)\right|d\lambda(y)\\&\ge
\sum_{n=1}^p \int_G \frac1K  |\varphi_n(y)|d\lambda(y)
=\frac1K\sum_{n=1}^p\left\|\frac{\varphi_n}{w}\right\|_{1,w},\end{align*}
where as in Theorem \ref{ex1}, the second equality follows from the fact that the functions under the sum have disjoint supports. The third equality proves then that $\mathscr L$ is selective. The  inequality is due to the diagonal boundedness of $w,$ and checks the validity of Condition
(\ref{iv}) in Definition \ref{Fell}. For, this inequality simply becomes
\[K\left\|\frac{\varphi}{w}\ast\frac{\delta_{x_{\alpha}}}{w(x_{\alpha})}\right\|_{1,w}
\ge\left\|\frac{\varphi}{w}\right\|_{1,w}\] when  $p=1$.
As in the previous example, note that Condition (\ref{vi}) (and so Condition (\ref{v}) in Definition \ref{Felll}) is also valid here with $m=1$, $a_\alpha$ and $b_\alpha$ may be taken as $\frac{\delta_{x_\alpha}}{w(x_\alpha)}$ and $w(x_\alpha)\delta_{x_\alpha^{-1}}$, respectively, where $\|a_\alpha\|_{w}=1$ and $\|b_\alpha\|_w=w(x_\alpha)w(x_\alpha^{-1})\le K$ for every $\alpha<\eta$ since $w$ is diagonally bounded by $K.$
\end{proof}

\begin{theorem} \label{ex4} Let $G$ be an infinite discrete weakly cancellative right cancellative
semigroup with cardinality
$\eta$ and let $w$ be a weight on $G$ that is diagonally bounded. Then the weighted semigroup algebra $\ell^1(G, w)$
has an F$\ell^1(\eta)$-base of type 1.
\end{theorem}

\begin{proof} Proceed as in Theorems \ref{ex2} and \ref{ex3}.
\end{proof}

\subsection{Algebraic structure}

As in the previous section, having F$\ell^1(\eta)$-bases, the theorems in Section 3 yield immediately the weighted analogues to
Theorems \ref{lalgebra} and  \ref{dlalgebra}. The proofs are precisely the same.
Both analogues are summarized in the theorem below.
The parts on the number of left ideals and the dimension of right ideals in
$L^1(G,w)^{**}$ and $\luc(G,w^{-1})^*$ appeared in \cite{FiSa07}. The rest is new.

\begin{theorem}\label{wlalgebra} Let $G$ be a non-compact locally compact group with  compact covering $\eta$ or an infinite discrete weakly cancellative right cancellative  semigroup with cardinality $\eta$, and let $w$ be a diagonally bounded weight on $G$. Let $\mathscr S$ be any of the algebras  $L^1(G,w)^{**}$, $\luc(G,w^{-1})^*$, $M(G,w)^{**}$, $L_0^\infty(G,w^{-1})^\perp$, $C_{0}(G,w^{-1})^{\perp_{lu}}$ or $C_{0}(G,w^{-1})^{\perp_M}$.
 Then
\begin{enumerate}
\item  there are at least $2^{2^\eta}$ many right cancellable elements in $\mathscr S$.
\item   there are at least $2^{2^{\eta}}$ many  principal left ideals in $\mathscr S$ with
 pairwise intersections equal to zero, the  dimension of each  is  at least $2^{2^{\eta}}$.
\item  the dimension of any principal right ideal,  and so of any nonzero right ideal, in $\mathscr S$
 is at least $2^{2^\eta}.$
 \end{enumerate}
\end{theorem}

\subsection{Topological centres}
We turn again to topological centres. Now that we have F$\ell^1(\eta)$-bases in $L^1(G,w)$ when $G$ is a non-compact locally compact group, and of $\ell^1(G,w)$ when $G$ is an infinite discret weakly cancellative, right cancellative semigroup, Theorems \ref{gfactori}  and \ref{mfactori} provide $L^\infty(G,w^{-1})$ and $M(G,w)^*$  with the $\eta$-factorization property.

Since the Mazur property and its level are preserved under isomorphisms  (as already observed by Neufang in \cite[Remark 3.4]{Neufang08B}
and Hu and Neufang in  \cite[Remark 1.5]{HuNe06}), the weighted group algebra and the weighted measure algebra have each Mazur property of level $\eta.w$ and $|G|.w$, respectively. Theorems \ref{gfactori}
and \ref{mfactori} together with Theorem \ref{Lcentre}
 and \ref{lucentre}  yield therefore, in our next theorem, the weighted analogues of
  Theorems \ref{centres} and \ref{naughtcentres}.

\begin{theorem}\label{wcentres}  Let $G$ be a locally compact group or a weakly cancellative right cancellative
semigroup,  and let $w$ be a weight on $G$ that is diagonally bounded.
Then \begin{enumerate}
\item the first topological centre of $L^1(G,w)^{**}$ is $L^1(G,w)$.
 \item the first topological centre of $\luc(G,w^{-1})^*$ is $M(G,w)$.
\item the first topological centre of $M(G,w)^{**}$  is $M(G,w)$ whenever $\eta\ge 2^{w(G)}$ or $|G|$ is non-measurable.
\item the first topological centre of ${L_0^\infty}(G,w^{-1})^\perp$ is $\{0\}$.
\item the first topological centre of $C_{0}(G, w^{-1})^{\perp_{lu}}$ is $\{0\}$.
\item the first topological centre of $C_{0}(G, w^{-1})^{\perp_M}$ is $\{0\}$ whenever $\eta\ge 2^{w(G)}$ or $|G|$ is non-measurable.
\end{enumerate}
\end{theorem}

\begin{proof} When $G$ is compact, Statement (i) follows from  \cite[Theorem 2.2]{balapy} since $L^1(G,w)$ is weakly sequentially complete, contains a bounded  approximate identity and is an ideal in its second dual.
Statements (ii), (iv) and (iv) are trivial when $G$ is compact.
So let us assume that $G$ is non-compact.
Then again the proofs of the first three statements are exactly the same as those given for Theorem \ref{centres}.

Since  $X\subseteq C_{0}(G,w^{-1})^{\perp_{lu}}$ in the  case of $\luc(G,w^{-1})^*$,
  $\widetilde X\subseteq {L_0^\infty}(G,w^{-1})^\perp$ in the case of $L^1(G,w)^{**}$
	and
$X\subseteq C_{0}(G,w^{-1})^{\perp_M}$ in the case of $M(G,w)^{**}$, and since \[L_0^\infty(G,w^{-1})^\perp\cap L^1(G,w)=C_{0}(G,w^{-1})^{\perp_{lu}}\cap M(G)=C_{0}(G,w^{-1})^{\perp_{M}}\cap M(G)=\{0\},\]
 the last three statements of the theorem follow  as in Theorem \ref{centres} as well.
\end{proof}

As in Remark \ref{mirrors},  the mirror theorem of Theorem \ref{wcentres} holds.
In particular, we obtain the following corollary.

\begin{corollary} \label{wsAir}
 Let $G$ be a locally compact group with compact covering $\eta$ and local weight $w(G)$,
or an infinite discrete weakly cancellative right cancellative  semigroup with cardinality $\eta$,
and let $G$ have a diagonally bounded weight $w$.
Then
\begin{enumerate}
\item $L^1(G,w)$ is strongly Arens irregular.
\item $M(G,w)$ is strongly Arens irregular whenever $\eta\ge 2^{w(G)}$ or $|G|$ is non-measurable.
\end{enumerate}
\end{corollary}

\begin{notes}
Statements (i) in Theorem \ref{wcentres} and Corollary \ref{wsAir} are due to Neufang in \cite{Neufang08B} for non-compact groups, Dales-Lau in \cite{DaLa05}
for locally compact compact groups with non-measurable compact covering,    and Filali-Salmi in \cite{FiSa-b} for non-compact SIN-groups.
Statement (ii) in Theorem \ref{wcentres} was also proved for locally compact compact groups
in  \cite{Neufang08B}, \cite{DaLa05} and \cite{FiSa-b}. Statement (iv) was proved in \cite{FiSa-b} when $G$ is a  non-compact SIN-group.
Statement (v) in Theorem \ref{wcentres}  was proved in \cite{FiSa-b}.
The rest of the statements are new. Again, as noted earlier, a careful study of \cite{Loetal}
may help to find an F$\ell^1(|G|.w)$-base in $M(G,w)$, which
would yield Statements (iii) and (vi) and Corollary \ref{wsAir}(ii) without any condition.
\end{notes}

\section{Application III: Fourier algebra}\label{Aexamples}
 This section deals with the Fourier algebra of compact non-metrizable groups, infinite products
of compact groups and compact connected groups with infinite dual rank. The dual rank will be defined in time.

\subsection{Some definitions and lemmas}
Recall that the Fourier algebra is the collection of all functions $h$ on $G$ of the form
$\bar f\ast \check g$ with $f,g\in L^2(G)$ and norm \[||h||=\inf\left\{||f||_2||g||_2:h=\bar f\ast\check g,\; f,g\in L^2(G)\right\}.\] The Banach dual of $A(G)$ is
 the group von Neumann algebra $VN(G)$ which is the closure in the weak operator topology of the linear span of $\{\lambda(x) : x \in G\}$
in $B(L^2(G))$, where $\lambda$ is the left regular representation of $G$ on $L^2(G)$ (see \cite{Eymard64}).

A subspace of $VN(G)$ which shall be of our interest is $\uc_2(G)$.
This is defined by
\[\uc_ 2(G)=\overline{VN(G)\cdot A(G)}^{||.||} = \overline{\{T\in VN(G): \mbox{supp}(T)\mbox{ is compact}\}}^{||.||},\]
where the support $\mbox{supp} ( T)$ of $T$ is in the sense of \cite[page 227]{Eymard64}, \[\supp(T )=\{x\in G : u\in A(G)\;\text{and}\; u\cdot T =0\implies u(x)=0\}.\]
 Recall that ${VN(G)\cdot A(G)}$ is automatically closed when  $G$ is amenable.
This is  usually called the  space  of  uniformly  continuous  functionals  on $G$, and is often denoted in the literature by $UC(\widehat G)$.
We  are using  the  notation $\uc_2(G)$ to  avoid  possible  confusion  arising  from the usage of the dual group $\widehat G$ of $G$.
The space $\uc_2(G)$ is in fact a C*-subalgebra of $VN(G)$,  proved by Granirer in \cite[Proposition 2(a)]{Ed}.
This is obtained by regarding $\uc_2(G)$ as the space generated by operators in $VN(G)$ with compact support and using \cite[Proposotion 4.8 (5$^o$)]{Eymard64}.
In the subsequent  paper \cite[Theorem 3]{Ed2}, Granirer proved that $\uc_2(G) = VN(G)$ if and only if $G$ is compact.

Following \cite{Eymard64}, let $P(G)$ be the space of continuous positive definite
functions on $G$ and $B(G)$ be its linear span,  equivalently,     the space of all coefficient functions of continuous
 unitary representations of $G$.
The space $B(G)$ is a Banach algebra,  called the  Fourier-Stieltjes algebra, and if $C^*(G)$ is the group
$C^*\text{-algebra}$ of $G$, then $B(G)$ is its Banach dual.
The Fourier algebra $A(G)$   is a closed ideal of $B(G)$ and can be seen as the norm-closed linear span of $P(G)\cap C_{00}(G)$ in $B(G)$,
where $C_{00}(G)$ be the space of continuous functions  with compact support.

As known $B(G)$ is the multiplier algebra of $A(G)$ if and only if $G$ is amenable, see \cite[page 209]{pier} for the sufficiency,  \cite{losert} for the necessity, and \cite{nebbia} when $G$ is discrete.

When $G$ is a compact group, as it is the case in this section,  $B(G)=A(G)$ and so $\uc_2(G)=VN(G)=B(G)^*$.

For more details on the Fourier algebra and the Fourier-Stieltjes algebra on locally compact groups, we direct the reader to the classical  article \cite{Eymard64} by Eymard
and to the recent book \cite{KaLa04} by Kaniuth and Lau.

When $G$ is abelian and $\widehat G$ is its dual group, then  via the Fourier transform, $A(G)$ identifies with $L^1(\widehat G)$ ,
$VN(G)$ with $L^\infty(\widehat G),$ $\uc_2(G)$ with $\luc(\widehat G)$, $P(G)$ with the set of Fourier transforms of all probability
measures  on $\widehat G$ and $B(G)$ with $M(\widehat G)$.
So all the results stated for the group algebra and the measure algebra hold for the Fourier algebra
and   the Fourier-Stieltjes algebra, respectively,  when $G$ is abelian.
However, unlike the case of $L^1(G),$ non-Arens regularity of $A(G)$ has turned out to be more resistant in general,
and the solution is still not complete.

We denote by $C^*(G)^{\perp}$ the annihilator space
 of $C^*(G)$ in $\uc_2(G)^*$ when $C^*(G)$ is regarded as a subalgebra of $\uc_2(G)$,
i.e.,
\[C^*(G)^{\perp} = \{\Psi \in \uc_2(G)^*: \langle\Psi,f\rangle= 0 \quad\text{for every}\quad f \in C^*(G)\}.\]

Recall also that we have the following decomposition of $\uc_2(G)^*$
as Banach space direct sum
\begin{align*} \uc_2(G)^*= B(G) \oplus C^*(G)^{\perp}
,\end{align*}
where $C^*(G)^{\perp}$
is a weak*-closed ideal in $\uc_2(G)^*$
(see for example \cite[Corollary  3.5]{dmm}).
\medskip

 For a locally compact group $G$,  let $\U(G)$ be the set of all unitary representations of $G$ on some Hilbert space $\hh$, $\widehat G$ be the dual object of $G$, that is
 the set of all equivalence classes of irreducible unitary representations of $G$ on some  Hilbert space $\hh,$
 and
$\U(\hh)$ be the space of all unitary operators on $\hh$.
  If $\pi\in\U(G)$,   $\hh_\pi$ will denote the representation space, i.e., $\pi \colon G\to \U( \hh_\pi)$. If $v\in \hh_\pi$, we define the matrix coefficients  $\pi^{uv}\colon G\to \C$ as usual by $\pi^{uv}(s)=\langle \pi(s)u,v\rangle$.
The diagonal matrix coefficients will be denoted simply by $\pi^{v}.$

	When the representation $\pi$ has finite dimension $n$ with a fixed base $\{u_1,...,u_n\}$ say,  we simply write $\pi^{kl}$, $1\le k,l\le n$ for the matrix coefficients of $\pi.$

\medskip

{\bf The $\mathcal{Z}$-property }
We shall prove first that  F$\ell^1(\eta)$-bases exist in the Fourier algebra when the group satisfies a certain property we shall call the {\it $\mathcal{Z}$-property}. We will see then
that a large class of locally compact groups has this property.

  \begin{definition}
\label{def:z}    Let $G$ be a locally compact group. We say
 that  $G$ has a \emph{$\mathcal{Z}$-property} of level $\eta$ when its dual $ \widehat{G}$ has an         upward directed  cover  $\{ K_\alpha\}_{\alpha<\eta}$\footnote{The meaning of \emph{upward directed} is: given $\alpha_1,\,\alpha_2 $ there is $\alpha_3>\alpha_1,\alpha_2$ such that $K_{\alpha_3}\supset K_{\alpha_1}\cup K_{\alpha_2}$.} together with
 a set of irreducible  representations
    $\{\pi_\alpha\colon \alpha<\eta\}$ such that
		\[{\displaystyle K_\alpha \otimes  \pi_\alpha\cap K_\beta\otimes\pi_\beta=\emptyset \mbox{ for every } \alpha< \beta<\eta}.\]
             \end{definition}

\medskip

The following lemma  was proved by  Arsac for any locally compact group (\cite[3\`eme partie, B-C]{arsac76},  or see \cite[Section 2.8]{KaLa04}). The lemma will be needed in each of our constructions
of F$\ell^1(\eta)$-bases in this section.
When the group is compact (as it is in most of the cases dealt with in this section), the interested reader can easily reproduce the proof using
Schur's orthogonality relations.

  \begin{lemma}[Arsac]\label{arsac}   Let $G$ be a locally compact group, $\{\pi_\alpha\colon \alpha<\eta\}$ be a set of disjoint unitary representations on $G$ and let $\pi=\oplus_{\alpha<\eta}\pi_\alpha.$
				Then \begin{enumerate}
				\item the closed linear spans  of the coefficients of $\pi_\alpha$ have zero pairwise intersections.
				\item each $\varphi$ is the closed linear span of the coefficients of $\pi$
			can be written uniquely as $\sum_{\alpha<\eta} \varphi_\alpha$,
			where each $\varphi_\alpha$ is in the closed linear span of the coefficients of $\pi_\alpha$
			and
						\[\|\varphi\|_{A(G)}=\sum_{\alpha<\eta}\|\varphi_\alpha\|_{A(G)}.\]
	\end{enumerate}		
	      \end{lemma}

\begin{lemma}\label{zcoef}
  Let $G$ be a   compact group with
a $\mathcal{Z}$-property of level $\eta$, and let
$\{\pi_\alpha\colon \alpha<\eta\}$ be the family of representations required for  the $\mathcal{Z}$-property for the   cover  $\{K_\alpha\}_{\alpha<\eta}$ of $\widehat{G}.$ For each $\alpha<\eta,$ let \[A_\alpha=\{\sigma^{uv}: \sigma\in K_\alpha, u,v\in \hh_\sigma\}.\]
\begin{enumerate}
\item
If for some constant $K_2>0$,  for every $ \alpha<\eta$, there exist
$1\le k, l\le \dim \pi_\alpha$
such that
\begin{equation}\|\varphi\|_{A(G)}\le K_2\|\varphi \pi_\alpha^{kl}\|_{A(G)}\quad\text{for every}\quad\varphi\in \langle A_\alpha\rangle,\label{needed}\end{equation}  then $A(G)$ has an F$\ell^1(\eta)$-base of type 2 with constants $K_1=1$ and $K_2=K.$
\item If there are $\eta$-many representations $\pi_\alpha$ with a fixed dimension $m,$ then $A(G)$ has an F$\ell^1(\eta)$-base of type 3
with constants $K_1=1$ and $0<K_2\le m$.
\end{enumerate}
  \end{lemma}

\begin{proof}
 For each $\alpha<\eta,$ for simplicity and for keeping the notation used in Section 3, let for each $\alpha<\eta$, $a_\alpha$ be any coefficient of the representation $\pi_\alpha$
 and let \[\mathscr L= \bigcup\limits_{\alpha<\eta} A_\alpha a_\alpha.\]
We claim first that $\mathscr L$ is a zooming $\ell^1(\eta)$-selective base in $A(G)$.
For this, we need to check that the spans of the sets $A_\alpha a_\alpha$ have zero pairwise intersections
 as Conditions (i) and (ii)
in Definition \ref{Fil} are clearly valid.

If $\sigma_{1}\in K_{\alpha}$ and $\sigma_{2}\in K_{\beta}$ for $\alpha<\beta<\eta$, then \[\sigma_{1}\otimes\pi_{\alpha} \in K_{\alpha}\otimes \pi_{\alpha}\quad\text{and}\quad  \sigma_{2}\otimes\pi_{\beta} \in K_{\beta}\otimes \pi_{\beta},\] and so  Condition (ii) of the $\mathcal{Z}$-property assures that
  \[\sigma_{1}\otimes\pi_{\alpha}\cap
\sigma_{2}\otimes\pi_{\beta}=\emptyset.\]

Therefore, by Lemma  \ref{arsac},
 the spans of the sets $A_\alpha a_\alpha$, $\alpha<\eta,$ have zero pairwise intersections.

We still need check that Condition (\ref{Csel}) in Definition \ref{sel} holds.
 Let $b\in\langle \mathscr L\rangle$.
Then we may write $b$ as \[b=\sum_{n=1}^p \varphi_n a_{\alpha_n}\quad\text{ with}\quad\varphi_n\in \langle A_{\alpha_n}\rangle, \quad\text{for}\quad n=1,..., p,\]
with  $\alpha_1< \alpha_2<...<\alpha_p$.

	Since the sets of unitary representations $K_\alpha\otimes\pi_\alpha$ are disjoint,
			 Lemma \ref{arsac} implies  that
	\begin{align*}\|b\|=
         \left\| \sum_{n=1 }^p  \varphi_{n}a_{\alpha_n}\right\|&=
						         \sum_{n=1}^p\left\| \varphi_{n}a_{\alpha_n}\right\|.
              \end{align*}
					We conclude that $\mathscr L$ is a zooming F$\ell^1(\eta)$-selective base with constant $K_1=1.$
	In particular, this holds when $a_\alpha=\pi_\alpha^{kl}$,  for each $\alpha<\eta$, where
	 $\pi_\alpha^{kl}$ are the coefficients for which Condition (\ref{needed}) in the lemma holds.
		
	The additional condition (\ref{needed}) makes $\bigcup_{\alpha<\eta} A_\alpha\pi_\alpha^{kl}$ an F$\ell^1(\eta)$-base of type 2 with constant $K_2$, proving the first statement.

	For the second statement, suppose that each representation in $\{\pi_\alpha:\alpha<\eta\}$
	has dimension $m$ and denote $\pi_\alpha^{1k}$ and $\overline{\pi_\alpha^{k1}}$ by $a_\alpha^k$ and
	$b_\alpha^k$, respectively, where $k=1,...,m.$  Then as in Statement (i),  \[\mathscr L=\bigcup_{\alpha<\eta}A_\alpha\{a_\alpha^k:k=1,...,m\}\] is  a zooming $\ell^1(\eta)$-selective base with constant $K_1=1$.
	Clearly, \[\sum_{k=1}^ma_\alpha^{k}b_\alpha^{k}=1\;\text{for each}\;\alpha<\eta,\]  giving Condition (\ref{vi}) in Remark \ref{fellrem}(ii), and so by Remark \ref{fellrem}(ii), we have an F$\ell^1(\eta)$-base of type 3 with constants \[K_1=1\quad\text{ and}\quad
	K_2=\sup_\alpha\|(\overline{\pi_\alpha^{k1}})_{k=1}^m\|_p=\sup_\alpha\|(\pi_\alpha^{1k})_{k=1}^m\|_p\le m.\]
	\end{proof}

\subsection{F$\ell^1(\eta)$-bases for compact non-metrizable group}\label{nonmetrizable}
We start with a lemma, which is needed for the construction of F$\ell^1(\eta)$-bases of type 3 in $A(G)$
when $G$ is compact, non-metrizable with a local weight $\eta$ having
 uncountable cofinality.
The lemma is due to Hu \cite[Proposition 4.3]{Hu95}, it is a stronger version of an earlier result of Lau--Losert \cite[Lemma 4.8]{LaLo93} whose proof is
based on Kakutani-Kodeira Theorem (see for example \cite[8.7]{HR}).
 With the F$\ell^1(\eta)$-bases, we shall   simplify and improve in Subsections \ref{alst},
  and \ref{topcen2} the main results on strong Arens irregularity of and algebraic structure proved in \cite{MMM1} and
\cite{MMM2}.

\begin{lemma}[Hu]\label{le:130508A}
Let $G$ be a $\sigma$-compact non-metrizable locally compact group with unit element $e$
and local weight $\eta,$
and let $\{O_\alpha\colon 0\le \alpha <\eta\}$ be an open base at $e$.
Then there exists a decreasing family $(N_\alpha)_{0\le \alpha\le\eta}$
of normal subgroups of $G$ (that is, $N_\beta\subset N_\alpha$ whenever $\alpha\le\beta$) such that
\begin{enumerate}
\item $N_0=G$ and $N_\eta=\{e\}$,
\item $N_{\alpha +1}\subset N_\alpha \cap O_\alpha$  for all  $0\le \alpha <\eta$,
\item $N_\alpha$ is compact for each $\alpha>0$,
\item $N_\alpha /N_{\alpha +1}$ is metrizable but $N_{\alpha +1}\ne N_\alpha$ for all $\alpha <\eta$,
\item $N_\gamma =\cap_{\alpha <\gamma} N_\alpha$ for every limit ordinal $\gamma <\eta$,
\item the local weight of $N_\alpha$ is $\eta$ for each $\alpha<\eta$.
\end{enumerate}
\end{lemma}

 Note that the condition in the following theorem on the cofinality of $\eta$ being uncountable is equivalent to
$\widehat G$ containing $\eta$ many representations of a fixed dimension $m$. For,
since each $\pi_\alpha$ has finite dimension, we may define
$E_m=\{\alpha \colon \dim \pi_\alpha =m\}$. Then $\cup_{m=1}^\infty E_m=\{\alpha \colon \alpha <\eta\}=\eta$.
So the condition that the cofinality of $\eta$ being uncountable implies that the cardinality of $E_m$  is $\eta$ for at least one $m$
   (\cite[Lemma 3.6]{Jech78}). The converse is clear.

\begin{theorem}\label{ex5} Let $G$ be a compact non-metrizable group with a local weight $\eta$ and suppose that $\eta$ has uncountable cofinality.
Then $A(G)$ has an F$\ell^1(\eta)$-base of type 3 with constants $K_1=1$ and $0<K_2\le m$ for some $m\in \N$.\end{theorem}

\begin{proof}
Let $\{N_\alpha\}$ be the family of normal subgroups of $G$ given by Hu's Lemma.
For each $\alpha<\eta,$ put \[K_\alpha=\{\sigma\in \widehat{G}:N_\alpha\subseteq \ker\sigma\},\] and note that $K_\alpha\subseteq K_\beta$ if $\alpha\le\beta<\eta.$
We claim first that $\widehat G=\bigcup_{\alpha<\eta} K_\alpha$.
Let $\sigma : G\to U(\hh_\sigma)$ be any representation in $\widehat{G}.$
Let $U$ be any neighbourhood of the identity in $U(\hh_\sigma).$ We can assume that $U$ contains no nontrivial subgroups (see, e.g., \cite[Corollary 2.40]{hoffmorr}).
Then $\sigma^{-1}(U)$ is a neighbourhood of the identity in $G$ and so it contains some basic
neighbourood $O_\alpha$ of the identity in $G$. Thus there exists some $\alpha<\eta$ such that
\[N_{\alpha+1}\subseteq \sigma^{-1}(U).\]  Therefore, $\sigma(N_{\alpha+1})$ is  a subgroup of $U(\hh_\sigma)$ contained in $U.$ Our choice of $U$ ensures then that  $N_{\alpha+1}\subseteq \ker\sigma$, i.e., $\sigma\in K_{\alpha+1},$ as required.

Next we define representations $\pi_\alpha\in \widehat G$ such that \[\sigma_1\otimes\pi_\alpha\cap \sigma_2\otimes\pi_\beta=\emptyset\quad\text{for every}\quad\sigma_1\in K_\alpha,\;\sigma_2\in K_\beta,\; \beta<\alpha<\eta\] so that
$G$ has the required $\mathcal Z$-property of level $\eta.$
To do so, start for each $\alpha<\eta$  with any non-trivial irreducible representation $\widetilde\pi_\alpha$ on $N_\alpha/N_{\alpha+1}$ and extend it to an irreducible representation   $\pi_\alpha^*$ on $N_{\alpha}$, where
$\pi_\alpha^*=\widetilde\pi_\alpha\circ q_\alpha$ and $q_\alpha \colon N_\alpha \to N_\alpha/N_{\alpha+1}$ is the quotient mapping. This latter representation is then extended, as in \cite[Theorem 27.46 ]{HR}, to an irreducible  representation $\pi_\alpha$ on $G$.

 Let $\alpha>\beta$ and suppose otherwise that for some $\sigma_1\in K_\alpha$ and $\sigma_2\in K_\beta,$ there is $\rho\in \widehat G$
such that \[\rho\in (\sigma_1\otimes \pi_\alpha)\cap(\sigma_2\otimes \pi_\beta).\]
Then \[\rho_{|N_\alpha}\in ({\sigma_1}_{|N_\alpha}\otimes {\pi_\alpha}_{|N_\alpha})\cap({\sigma_2}_{|N_\alpha}\otimes {\pi_\beta}_{|N_\alpha}).\]
 Now $N_\alpha\subseteq N_{\beta+1}\subset N_\beta$ (since  $\alpha>\beta$), so  ${\sigma_1}_{|N_\alpha}=\mathrm{Id}$ and ${\sigma_2}_{|N_\alpha}=\mathrm{Id}$ by the definition of $K_\alpha$.
We also note that ${\pi_\beta}_{|N_\alpha}=\mathrm{Id}$ since  by the definition of $\pi_\beta$,
${\pi_\beta}_{|N_{\beta+1}}=\mathrm{Id}$. Thus, $\rho_{|N_\alpha}=\mathrm{Id}$, and consequently,
  \[\mathrm{Id}\in {\sigma_1}_{|N_\alpha}\otimes {\pi_\alpha}_{|N_\alpha}=\mathrm{Id}\otimes {\pi_\alpha}_{|N_\alpha}= {\pi_\alpha}_{|N_\alpha}\oplus...\oplus{\pi_\alpha}_{|N_\alpha}.\]
 Now ${\pi_\alpha}_{|N_\alpha}$ is not necessarily irreducible but it decomposes into a finite direct
 sum of irreducible ones; and the proof of Theorem 27.48 in \cite{HR} confirms that the original representation $\pi_\alpha^*$ of $N_\alpha$ appears in the decomposition. Thus $\mathrm{Id}= \pi_\alpha^*,$
 which goes against our choice of $\pi_\alpha^*$. Therefore, $\sigma_1\otimes\pi_\alpha$ and
$\sigma_2\otimes\pi_\beta$ must be disjoint for every $\sigma_1\in K_\alpha,\;\sigma_2\in K_\beta,\; \beta<\alpha<\eta.$
Therefore, $G$ has a $\mathcal Z$-property of level $\eta.$

Now we go after Condition (\ref{v}) in Definition \ref{Felll}.
Since each $\pi_\alpha$ has finite dimension, we may define
$E_m=\{\alpha \colon \dim \pi_\alpha =m\}$ so that \[\cup_{m=1}^\infty E_m=\{\alpha \colon \alpha <\eta\}=\eta
.\]
As observed before the statement of the theorem, by the condition that the cofinality of $\eta$ being uncountable, we may take for some $m\in\N,$
a subset $E_m$ of $\eta$ of cardinality  $\eta$ such that $\dim \pi_\alpha =m$ for each $\alpha\in E_m$.
Then by replacing the family of subgroups $\{N_\alpha:\alpha <\eta\}$ and its associated family of representations
$\{\pi_\alpha^*:\alpha<\eta\}$ with those that are indexed by the set $E_m$, we obtain a family of irreducible unitary representations, which
for simplicity we still denote by $\{\pi_\alpha :\alpha <\eta\}$,
such that  $\dim \pi_\alpha =m$ for all $\alpha <\eta$. Since this new family of representations still provides $G$ with a  $\mathcal Z$-property of level $\eta,$ we may proceed as in Lemma \ref{zcoef} (ii) to complete the proof. The F$\ell^1(\eta)$-base so obtained has constants $K_1=1$ and $K_2\le m.$
\end{proof}

\subsection{F$\ell^1(\eta)$-bases for infinite products of compact groups}
Our second task will be first with infinite products of compact groups. The case of countably infinite products of metrizable compact groups with
 a second countable locally compact amenable group
was considered by Lau and Losert in \cite{LaLo05}.
 We prove that  infinite products of $\eta$ many compact groups have the $\mathcal Z$-property of level $\eta$ and satisfy Condition (\ref{needed}) in Lemma \ref{zcoef} (i), and so by this lemma the Fourier algebra of such groups
	contains an F$\ell^1(\eta)$-base of type 2.
 With the F$\ell^1(\eta)$-bases, we shall   simplify and improve in Subsection \ref{topcen2} the main result on strong Arens irregularity proved in \cite{LaLo05}. In addition to this, we shall
also deduce in Subsection \ref{alst}  some rich algebraic structures in the algebras.

 We need first the following simple lemma.

\begin{lemma} \label{complemented} Let $G$ be  a locally compact group and suppose that $G=LM$ for some subgroups $L$ and $M$  of $G.$ Let $\pi:G\to U(\hh_\pi)$
be  a representation of $G$ with $\pi_{|M}=\mathrm{Id}$ (or $\pi_{|L}=\mathrm{Id})$.  Then $\pi$ is irreducible if and only if $\pi_{|L}$ (
$\pi_{|M}$, respectively) is irreducible.
\end{lemma}

\begin{proof} Let $X$ be  a closed subspace of $\hh_\pi.$ If $X$ is invariant under $\pi,$ then it is clearly so under  $\pi_{|L}$
and $\pi_{|M}$. For the converse, suppose that $\pi_{|M}=\mathrm{Id}$ and $X$ is invariant under $\pi_{|L}$.
Then, writing each $s\in G$ as $s=lm$ for some $l\in L$ and $m\in M,$ we see that  \[\pi(s)(X)=\pi(lm)(X)=\pi(l)(X)\quad\text{for every}\quad s\in G,\] which shows that
$X$ is invariant under $\pi$. The proof is the same in the other situation.
\end{proof}

\begin{theorem}\label{prods}
Let $\{G_\alpha\}_{\alpha<\eta}$  be an infinite family of non-trivial compact groups and let $G=\prod\limits_{\alpha<\eta}G_\alpha$. Then
$A(G)$
 has an F$\ell^1(\eta)$-base of type 2 with constants $K_1=K_2=1.$
\end{theorem}

\begin{proof} First we  prove first that
 $\widehat G$ has a family of representations
$\{\pi_\alpha\}_{\alpha<\eta}$ with the $\mathcal Z$-property.
 For each $\alpha<\eta,$ let $e_\alpha$ be the identity in $G_\alpha$.
Let $L_\alpha$ and $M_\alpha$
be the subgroups given by  \[L_\alpha=\{e_\gamma\}_{\gamma<\alpha}\times\prod\limits_{\gamma\ge \alpha}G_\gamma\quad\text{and}\quad M_\alpha=
\prod\limits_{\gamma< \alpha}G_\gamma\times \{e_\gamma\}_{\gamma\ge\alpha}.\]
Consider, for each $\alpha<\eta,$ the set \[K_\alpha=\{\sigma\in \widehat{G}:L_\alpha
\subseteq \ker\sigma\}.\] Note that $K_\alpha\subseteq K_\beta$ if $\alpha\le\beta<\eta.$
We claim first that $\widehat G=\bigcup_{\alpha<\eta} K_\alpha$.
Let \[\sigma:\prod\limits_{\alpha<\eta}G_\alpha\to U(\hh_\sigma)\] be any representation in $\widehat{G}.$
As in Theorem \ref{ex5},  let $U$ be any neighbourhood of the identity in $U(\hh_\sigma)$ that contains no nontrivial subgroups.
Then $\sigma^{-1}(U)$ is a neighbourhood of the identity in $\prod\limits_{\alpha<\eta}G_\alpha$ and so it contains some basic
neighbourhood of the identity in $\prod\limits_{\alpha<\eta}G_\alpha.$ Thus there exists some $\alpha<\eta$ such that
\[\{e_\gamma\}_{\gamma<\alpha}\times\prod\limits_{\gamma\ge \alpha}G_\gamma\subseteq \sigma^{-1}(U),\] that is, $L_\alpha\subseteq \sigma^{-1}(U).$ Therefore, $\sigma(L_\alpha)$ is  a subgroup of $U(\hh_\sigma)$ contained in $U.$
Since $U$ contains no nontrivial subgroups,
our claim that $L_\alpha\subseteq \ker\sigma$ follows.

Denote now  the projections of $G$ onto $G_\alpha$ by $pr_\alpha$, and take for each $\alpha<\eta$,
 a non-trivial representation $\pi^*_\alpha\in \widehat{G_\alpha}$.

Then, define for each $\alpha<\eta,$ a representation
$\pi_\alpha$ on $G$ simply by $\pi^*_\alpha\circ pr_\alpha,$
and consider the family $\{\pi_\alpha:\alpha<\eta\}$. To check tha this family of cardinality $\eta$
provides $G$ with the $\mathcal Z$-property,
suppose otherwise that for some $\alpha< \beta<\eta$, $\sigma_1\in K_\alpha$ and $\sigma_2\in K_\beta,$
we have
\[ \sigma_1 \otimes \pi_\alpha\cap\sigma_2 \otimes \pi_\beta\ne\emptyset,\]
and let $\rho\in \widehat{G}$ such that
\[\rho\leq \sigma_1\otimes \pi_{\alpha}\quad\text{and}\quad\rho\leq  \sigma_2\otimes\pi_{\beta}.\]
 Then, since $\alpha<\beta,$ we see that $L_{\beta}\subseteq L_\alpha$, and so $L_\beta\subseteq \ker\sigma_1\cap\ker\sigma_2.$
Thus,
  ${\sigma_1}_{|L_{\beta}}$ and  ${\sigma_2}_{|L_{\beta}}$ are identities. We also see by the definitions of $L_\beta$ and $\pi_\alpha$
	that ${\pi_{\alpha}}_{|L_{\beta}}$ is the identity.
	Thus,
\begin{equation}\label{Id}\begin{split}
\restr{\rho}{L_{\beta}}\leq \restr{\left(\sigma_1\otimes\pi_{\alpha}\right)}{L_{\beta}}&=\mathrm{Id}\oplus\cdots \oplus \mathrm{Id}\\
\restr{\rho}{L_{\beta}}\leq \restr{\left(\sigma_2\otimes\pi_{\beta}\right)}{L_{\beta}}&=\restr{\pi_{\beta}}{L_{\beta}}\oplus\ldots\oplus \restr{\pi_{\beta}}{L_{\beta}}.\end{split}\end{equation}
 Now it is clear, that for each $\beta<\eta,$ we have $G=L_{\beta}M_{\beta}$ and ${\pi_{\beta}}_{|M_{\beta}}=\mathrm{Id}.$ Therefore,   by Lemma \ref{complemented},  each $\restr{\pi_{\beta}}{L_{\beta}}$ is an irreducible representation of $L_{\beta}$.
It follows from the equalities in (\ref{Id}) that $\restr{\pi_{\beta}}{L_{\beta}}$ must be equal to the trivial one-dimensional representation $1$. But this implies in turn that $\pi_{\beta}=1$ which goes against our choice of $\pi_\beta$.
 Therefore, \[\sigma_1 \otimes \pi_\alpha\cap\sigma_2 \otimes \pi_\beta=\emptyset\quad\text{for every}\quad
\alpha< \beta<\eta,\; \sigma_1\in K_\alpha,\; \sigma_2\in K_\beta,\]
as required for the $\mathcal Z$-property.

Next we check that Condition (\ref{needed}) in Lemma \ref{zcoef} holds.
 Recall that  \[A_\alpha=\{\sigma^{uv}: \sigma\in K_\alpha, u,v\in \hh_\sigma\}.\]
For each $\alpha<\eta,$ let then $\varphi\in \langle A_\alpha\rangle$ and $\pi_\alpha^{k}$ be any diagonal coefficient of the representation $\pi_\alpha$
($k$  may change with $\alpha$).

Then $\varphi=\sum_{n=1}^p z_n\sigma_n^{u_nv_n},$ where  for $n=1,...,p$, $z_n$ are scalars and  $\sigma_n^{u_nv_n}$ are coefficients of the representations  $\sigma_n\in K_\alpha$ so that $L_\alpha\subseteq \ker\sigma_n$.

Since $G=L_\alpha M_\alpha,$ writing each $s\in G$ as $s=hm$ with $h\in L_\alpha$ and $m\in M_\alpha,$
we see that \[\varphi(s)\pi_\alpha^{k}(s)=\varphi(hm)\pi_\alpha^{k}(hm)=\varphi(m)\pi_\alpha^{k}(h)
=\varphi\otimes\pi_\alpha^{k}(m,h).\]
So if $P$ is any operator in $VN(M_\alpha)$ of norm $1$ and $I$ is the identity operator in $VN(L_\alpha)$, then \[\|\varphi\pi_\alpha^{k}\|_{A(G)}=\|\varphi\otimes \pi_\alpha^{k}\|_{A(L_\alpha\times M_\alpha)}
 \ge |\langle \varphi\otimes \pi_\alpha^{k}, P\otimes I\rangle|=|\langle \varphi, P\rangle||\langle  \pi_\alpha^{k}, I\rangle|
=|\langle \varphi, P\rangle|.\]
Thus, $\|\varphi\pi_\alpha^{k}\|\ge \|\varphi\|,$ and so by Lemma \ref{zcoef}, $A(G)$ has an F$\ell^1(\eta)$-base of type 2 with constants $K_1=K_2=1$.
\end{proof}

\begin{remark} \label{abitmore} We can actually prove a bit more in Theorem \ref{prods}:
 For any $\alpha<\eta$ and $\sigma_1\neq \sigma_2\in K_\alpha$, the representations $\sigma_1\otimes\pi_\alpha$ and $\sigma_2\otimes\pi_\alpha$  are also disjoint.
To prove this, suppose that  for some $\alpha<\eta,$ and some $\sigma_1,$
$\sigma_2\in K_\alpha$,
the representations $\sigma_1\otimes\pi_\alpha$ and $\sigma_2\otimes\pi_\alpha$ contains a common irreducible representation $\rho.$
Then since ${\pi_\alpha}_{|M_\alpha}=\mathrm{Id},$  we obtain \[\begin{split}\rho_{|M_\alpha}&\le (\sigma_1\otimes\pi_\alpha)_{|M_\alpha}\cap (\sigma_2\otimes\pi_\alpha)_{|M_\alpha}
\\&=({\sigma_1}_{|M_\alpha}\oplus{\sigma_1}_{|M_\alpha}\oplus...\oplus{\sigma_1}_{|M_\alpha})\cap({\sigma_2}_{|M_\alpha}\oplus{\sigma_2}_{|M_\alpha}\oplus...\oplus{\sigma_2}_{|M_\alpha}).\end{split}\]
Therefore $({\sigma_1}_{|M_\alpha}\oplus{\sigma_1}_{|M_\alpha}\oplus...\oplus{\sigma_1}_{|M_\alpha})\cap({\sigma_2}_{|M_\alpha}\oplus{\sigma_2}_{|M_\alpha}\oplus...\oplus{\sigma_2}_{|M_\alpha})\ne\emptyset.$
Now since $G=L_\alpha M_\alpha $ and  ${\sigma_1}_{|L_\alpha}$ and  ${\sigma_2}_{|L_\alpha}$ are identities,
 Lemma \ref{complemented} implies that ${\sigma_1}_{|M_\alpha}$ and ${\sigma_2}_{|M_\alpha}$ are irreducible, and so
${\sigma_1}_{|M_\alpha}$ and ${\sigma_2}_{|M_\alpha}$ must be equal. Again, since $G=L_\alpha M_\alpha,$ we must have $\sigma_1=\sigma_2$, as required.
\end{remark}

\subsection{F$\ell^1(\eta)$-bases for compact connected groups}
The results on infinite products of compact groups seen in the previous subsection can be transferred to compact connected groups $G$ having  a certain infinite "dimension". This will be our aim in the  present subsection. The arguments are based on some structure theorems,  where $G$ is sandwiched between two groups having the form of the groups studied in the previous subsection  (see \cite[Theorems 9.25 2.26]{hoffmorr}).

\subsubsection{\bf The $\mathcal{Z}^\prime$-property.}
We shall develop in this subsection the arguments used in the proof of Theorem \ref{prods}
to include other compact groups having the $\mathcal Z$-property.
For this we introduce another property of a family of representations of $G$ called the $\mathcal{Z}^\prime$-property of level $\eta$
for some infinite cardinal $\eta$, as well as a new definition of a d-rank of a group.
The $\mathcal{Z}^\prime$-property is easier to handle than the $\mathcal Z$-property, while the d-rank of a group
 is a kind of a dimension of a group  based on which the required representations with the $\mathcal{Z}^\prime$-property  are constructed.
We shall see then that all compact connected groups with infinite d-rank have the $\mathcal Z^\prime$-property.

Since  $\mathcal Z^\prime$-property of level $\eta$ yields  $\mathcal Z$-property of level $\eta$, and this type of groups
satisfies Condition (\ref{needed}) in Lemma \ref{zcoef} (i), their Fourier algebras will
	contain an F$\ell^1(\eta)$-base of type 2.
	
	As with the previous cases, this will enable us to obtain in the last subsection \ref{alst} the strong Arens irregularity not only of $A(G)$
	 when $G$ belongs to this class of groups, but 	of $A(G\times H)$ where $H$ is any locally compact amenable group. The centre of $\uc_2(G\times H)^*$ will be seen to be
the Fourier-Stieltjes $B(G\times H).$
We shall also deduce 	some rich algebraic structures in the algebras related to these groups.

  \begin{definition}\label{def:zp} Let $G$ be a locally compact group with its dual enumerated as $ \widehat{G}=\left\{\sigma_\alpha\colon \alpha<\eta\right\}$.  We say
 that  $G$ has the \emph{$\mathcal{Z'}$-property} of level $\eta$ when $\widehat{G}$ contains a subset     $\{\pi_\alpha\colon \alpha<\eta\}$ such that if
\[L_\alpha=\bigcap_{\beta<\alpha} (\ker \sigma_\beta\cap \ker \pi_\beta)\quad\text{for every}\quad \alpha< \eta,\] then      \[G=L_\alpha \ker \pi_\alpha\quad\text{ for every}\quad \alpha< \eta.\]
        \end{definition}

\begin{notation}
For a locally compact group $G$ with a subgroup $L$,
let
\begin{enumerate}
\item $A_L(G)=\{\varphi\in A(G): \varphi(nx)=\varphi(x)\quad\text{for every}\quad x\in G, n\in L\}.$
\item $R_L: A(G)\to A(L)$ be the restriction map.
\item $Q_L: A(G/L)\to A_L(G)$ be given by $Q_L(u)=u\circ q_L$, when $L$ is normal and  $q_L:G\to G/L$ is the quotient homomorphism.
\end{enumerate}
\end{notation}

\begin{lemma}\label{LM}
Let $G$ be a compact group such that $G=LM$, where $L$ and $M$ are normal subgroups of $G$
with $L$ compact and $M$ closed. Then, for every $\varphi\in A_L(G),$ we have
\[\|R_M(\varphi)\|_{A(M)}=\|\varphi\|_{A(G)}.\]
\end{lemma}

\begin{proof} Consider the following diagram
\[\begin{CD}
A(G/L)@>Q_L >>A_L(G)\\
@VVjV@VVR_MV\\A(M/(L\cap M))@>Q_{L\cap M}>>A_{L\cap M}(M).
\end{CD}\]
Then $Q_L$ and $Q_{L\cap M}$ are isometric isomorphisms (see \cite[Theorem 2.4.2]{KaLa04},
while the map $j$ is the isometric isomorphism induced by the topological isomorphism
$LM/L\to M/(L\cap M).$
Since the diagram commutes, $R_M$ has to be an isometric isomorphism as well.
\end{proof}

\begin{lemma}\label{BL} Let $G$ be a compact group with the $\mathcal{ Z}^\prime$-property
of level $\eta$. Then $A(G)$ contains an F$\ell^1(\eta)$-base of type 2 with constants $K_1=K_2=1.$
\end{lemma}

\begin{proof}
Let
 $\widehat{G}=\left\{\sigma_\alpha\colon \alpha<\eta\right\}$ with $\sigma_0=1$.
 Let $\left\{\pi_\alpha\colon \alpha<\eta\right\}$ be a family of irreducible representations having  the $\mathcal{ Z}^\prime$-property for this enumeration of $\widehat{G}$,
and let $\{L_\alpha:\alpha<\eta\}$ be, as in Definition \ref{def:zp}, the family of closed subgroups of $G$ directed downward,
\[L_\alpha=\bigcap_{\beta<\alpha} (\ker \sigma_\beta\cap \ker \pi_\beta)\quad\text{for every}\quad \alpha< \eta.\]

 We see first that the family $\{\pi_\alpha\colon \alpha<\eta\}$ has the $\mathcal{Z}$- property.

To begin with, define for each $\alpha<\eta$,
 \[K_\alpha=\{\sigma \in \widehat{G}\colon L_\alpha\subseteq \ker \sigma\}.\]

 Since $\sigma_\alpha \in K_{\alpha^+}$ for every $\alpha<\eta$, we see that $\widehat{G}=\bigcup_{\alpha<\eta}K_\alpha$
so that $\{ K_\alpha\}_{\alpha<\eta}$ is an         upward directed  cover of  $\widehat{G}$.

Next, let $\alpha_1<\alpha_2<\eta$  and   $\sigma_i\in K_{\alpha_i}$, $i=1,2$, be given.
Suppose that for some $\rho\in \widehat{G}$, \[\rho\leq \sigma_1\otimes\pi_{\alpha_1}\cap \sigma_2\otimes\pi_{\alpha_2}.\]

 Then, since $L_{\alpha_2}\subseteq \ker \sigma_1\cap\ker \sigma_2$ and $\pi_{\alpha_1}(L_{\alpha_2})=\mathrm{Id}$ (as $L_{\alpha_2}\subseteq \ker \pi_{\alpha_1}$ by definition), we find
\begin{align*}
\restr{\rho}{L_{\alpha_2}}\leq \restr{\left(\sigma_1\otimes\pi_{\alpha_1}\right)}{L_{\alpha_2}}&=\mathrm{Id}\oplus\cdots \oplus \mathrm{Id}\\
\restr{\rho}{L_{\alpha_2}}\leq \restr{\left(\sigma_2\otimes\pi_{\alpha_2}\right)}{L_{\alpha_2}}&=\restr{\pi_{\alpha_2}}{L_{\alpha_2}}\oplus\ldots\oplus \restr{\pi_{\alpha_2}}{L_{\alpha_2}} .\end{align*}

  Now as before,   Lemma \ref{complemented} implies that $\restr{\pi_{\alpha_2}}{L_{\alpha_2}}$ is an irreducible representation of $L_{\alpha_2}$ since $G=L_{\alpha_2}\ker\pi_{\alpha_2}$.  We deduce from the equalities above that $\restr{\pi_{\alpha_2}}{L_{\alpha_2}}=1$, the trivial one-dimensional representation. But this implies that $\pi_{\alpha_2}=1$ which goes against our construction. We thus deduce that Definition \ref{def:z} holds, i.e., $G$
	has the $\mathcal Z$-property of level $\eta.$

	Next we check that condition (\ref{needed}) in Lemma \ref{zcoef} is valid so that  $A(G)$ has
	an F$\ell^1(\eta)$-base of type 2.
		 Recall that  for each $\alpha<\eta,$ \[A_\alpha=\{\sigma^{uv}: \sigma\in K_\alpha, u,v\in \hh_\sigma\},\] and as in the proof of Theorem \ref{prods},
	let $\varphi\in \langle A_\alpha\rangle$ and $\pi_\alpha^{k}$ be any diagonal coefficient of the representation $\pi_\alpha$ where $k$ may change with $\alpha$.
Since $\varphi$ is a linear combination of coefficients $\sigma^{uv}$ where $\sigma\in K_\alpha$ (so that $L_\alpha\subseteq \ker\sigma$), we see that
	 $\varphi\in A_{L_\alpha}(G)$.
	Now if $M_\alpha$ denotes $\ker(\pi_\alpha)$, then $G=L_\alpha M_\alpha$ by the $\mathcal{ Z}^\prime$-property. We may apply therefore Lemma \ref{LM} to see that $\|\varphi\|_{A(G)}=\|R_{M_\alpha}(\varphi)\|_{A(M_\alpha)}$.
	Choose then for every $\epsilon>0,$ $S_\epsilon\in VN(M_\alpha)$ with $\|S_\epsilon\|\le 1$
	and \[\|\varphi\|_{A(G)}=\|R_{M_\alpha}(\varphi)\|_{A(M_\alpha)}\le|\langle S_\epsilon, R_{M_\alpha}(\varphi)\rangle|+\epsilon.\]
	
	Since $M_\alpha=\ker{\pi_\alpha}$, we have $R_{M_\alpha}(\varphi\pi_\alpha^{k})=R_{M_\alpha}(\varphi)$, whence
	\[\|\varphi\|_{A(G)}\le|\langle S_\epsilon, R_{M_\alpha}(\varphi\pi_\alpha^{k})\rangle|+\epsilon=
	|\langle  (R_{M_\alpha})^*(S_\epsilon),\varphi\pi_\alpha^{k}\rangle|+\epsilon.
	\]
	Since
	$\|R_{M_\alpha}^*(S_\epsilon)\|\le\|S_\epsilon\|\le 1$  (see for instance \cite[Theorem 2.4.1]{KaLa04}), we conclude
	as required that
	$\|\varphi\|_{A(G)}\le \|\varphi\pi_\alpha^{k}\|_{A(G)}.$ Therefore, $A(G)$ has an F$\ell^1(\eta)$-base of type 2 with constants $K_1=K_2=1.$
	  \end{proof}

  \begin{remarks}\label{prodd}~\normalfont

	\begin{enumerate}
\item Again, as in Remark \ref{abitmore}, we can prove a bit more when $G$ is as in Lemma \ref{BL}.
In fact, for  every $\alpha<\eta$ and $\sigma_1,\sigma_2\in K_\alpha$,   \[ \sigma_1\otimes \pi_{\alpha}\cap \sigma_2\otimes\pi_{\alpha}\ne\emptyset\quad\text{if and only if}\quad  \sigma_1\;\text{ and}\;\sigma_2\;\text{ are equivalent}.\] To prove this, suppose that
	
	$\rho\leq\sigma_1\otimes \pi_{\alpha}\cap \sigma_2\otimes\pi_{\alpha}$ for some $\rho\in \widehat{G}$. Then,
\begin{align*}
\restr{\rho}{\ker\pi_{\alpha}}\leq \restr{\left(\sigma_1\otimes \pi_{\alpha}\right)}{\ker\pi_{\alpha}}&=
\restr{\sigma_1}{\ker\pi_{\alpha}}\oplus\ldots\oplus \restr{\sigma_1}{\ker\pi_{\alpha}} \\
\restr{\rho}{\ker\pi_{\alpha}}\leq \restr{\left( \sigma_2\otimes\pi_{\alpha}\right)}{\ker\pi_{\alpha}}&=
\restr{\sigma_2}{\ker\pi_{\alpha}}\oplus\ldots\oplus \restr{\sigma_2}{\ker\pi_{\alpha}} .\end{align*}

Now, for $i=1,2$, $\sigma_i(L_\alpha)=\mathrm{Id}$ and $G=L_\alpha \ker\pi_{\alpha}$.  By Lemma \ref{complemented}, it follows that $\restr{\sigma_i}{\ker\pi_{\alpha}}$ is an irreducible representation of $\ker\pi_{\alpha}$. The above equalities then imply that  $\restr{\sigma_1}{\ker\pi_{\alpha}}$ and $\restr{\sigma_2}{\ker\pi_{\alpha}}$ are equivalent. Using again that $G=L_\alpha \ker\pi_{\alpha}$, we conclude that $\sigma_1$ and $\sigma_2$ are equivalent.
\item If  $(G_\alpha)_{\alpha<\eta}$ is an infinite family of  non-trivial compact groups,
then it is not difficult to verify that $G=\prod\limits_{\alpha<\eta} G_\alpha$ has the $\mathcal{Z}^\prime$-property of level $\eta$,
and so  $A(G)$ contains an F$\ell^1(\eta)$-base of type 2 by  Lemma \ref{BL} (as already shown in the previous subsection).
\end{enumerate}
\end{remarks}

\subsubsection{\bf d-rank of a compact group}
 We proceed now to prove that the $\mathcal Z^\prime$-property is available for compact connected groups with an infinite dimension in the sense of Definition \ref{drank} below.
Since our task concerns elements picked from the dual of $G$, we shall call our dimension of $G$ the dual rank of $G$.
When the dual rank is infinite, we shall see that it coincides with the local weight $\eta$ of the group.
So by Lemma \ref{zcoef}, for these groups there are  F$\ell^1(\eta)$-bases of type $2$ in  $A(G)$. With these bases, as before, we deduce the strong Arens irregularity as well as the algebraic structure in the algebras.
Recall that when the group is abelian, $A(G)$ is  strongly Arens irregular independently the dual rank of the group since we know from Section \ref{Examples} that $L^1(\widehat G)$ is strongly Arens irregular for any locally compact group $G.$

Before we give the precise definition of a dual rank of a compact group, we recall the following facts.

The torsion-free rank $r_0(H)$ of an abelian discrete group $H$ is the cardinality of a maximal independent subset of elements in $H$ whose order is infinite.  So if $t(H)$ denotes the torsion subgroup of $H$, then $r_0(H)$ is the maximal number of independent elements in $H/t(H)$ and $r_0(H)=r_0(H/t(H)),$ see
\cite[A13]{HR}. Note that if the quotient  $H/t(H)$ is countable, then $r_0(H)$ may be either finite or countable.
If the  quotient $H/t(H)$ is uncountable, then $r_0(H)$ is uncountable and in fact $r_0(H)=|H/t(H)|$.

There are several notions  of  dimension of a topological space. The main ones  coincide in some
instances such as separable metrizable spaces or  compact abelian groups, see the discussion at \cite[8.25]{hoffmorr} or  \cite[Chapter 7]{EngelTop}.
 We take that given in \cite[Definition 3.11]{HR}.
With this definition,
the dimension $\dim G$ of a compact  abelian group $G$ is the same as the torsion-free rank of its dual group $\widehat G$, see \cite[Theorem 24.28]{HR}. So we may define the dual rank of a
compact  abelian group
$G$  as \[\mathrm{d-rank}(G)=\dim G=r_0(\widehat G).\]

  When $G$ is a connected compact abelian group, its dual group $\widehat G$ is torsion free, see for example \cite[Theorem 24.25]{HR}, and
	so $\widehat{G}$ contains a subgroup isomorphic to $\bigoplus_{\alpha<r_0(\widehat{G})} \Z$ and it is contained in its divisible hull
			$\bigoplus_{\alpha<r_0(\widehat{G})} \Q$ (the divisible hull being also torsion free, has the stated form,  see for example the proof of \cite[Theorem A. 14]{HR}).
	Hence,  the following observation follows.
		\begin{equation}\label{obse}\begin{split} &\widehat G
\;\text{ is countable  (i.e.,}\; G \text{is metrizable and}\; w(G)=w)\;\text{if and only if}\\&
r_0(\widehat G)\le w.\;\text{  Otherwise,}\;	r_0(\widehat G)=|\widehat G|=w(G). \end{split}\end{equation}
    Hence for any connected compact abelian group, we have the three following cases
		\begin{equation}\label{rank}\begin{split}
\mathrm{d-rank}(G)&=r_0(\widehat{G})\quad\text{ when}\; w(G)=w\;\text{and}\;r_0(\widehat{G})\quad\text{is  finite}, \\
\mathrm{d-rank}(G)&=r_0(\widehat{G})=|\widehat G|=w(G)\quad\text{ when}\;r_0(\widehat{G})=w= w(G),\\
		\mathrm{d-rank}(G)&=r_0(\widehat G)=|\widehat G|=w(G),\; \text{when}\; w(G)> w, \end{split}\end{equation} where $w(G)$ is the local weight of $G$.

	We need also to recall the commutator $G^\prime$  of a topological group $G.$
	This is the subgroup of $G$ generated by the commutators $sts^{-1}t^{-1}$ of all elements $s,t\in G.$
The subgroup $G^\prime$ is normal but it is not in general closed in $G$.
The quotient $G/N$ of $G$ by any normal subgroup $N$ is abelian if and only if $G'\subseteq N;$ in particular $G/G'$ is abelian.
 A compact connected group is said to be semisimple when $G=G'.$

When $G$ is a compact connected group,  the commutator $G'$ of $G$ is closed and connected, \cite[Theorem 9.2]{hoffmorr}. There is actually a family of simple simply connected compact Lie groups $\{S_j:j\in J\}$
 such that  $G'$ is a quotient of the product  $\prod_{j\in J}S_j$,  \cite[Theorem 9.19]{hoffmorr}. The cardinal card $J$ is then an isomorphy invariant of $G$, denoted as $\aleph(G)$, see \cite[Page 496 or 469]{hoffmorr}.

Now every compact connected group $G$ is a semidirect product of $G^\prime$ and the compact connected abelian group $G/G^\prime$ (see \cite[Theorem 9.39]{hoffmorr}), and so $G$ is homeomorphic to $G^\prime \times G/G^\prime$.
 So we may define the dual rank of $G$ of a compact connected group as \[ \drr(G)=\aleph(G)+r_0(\widehat{G/G^\prime}).\]
		
Finally, recall from  with \cite[Theorem 3.12]{HR} that the dimension $\dim G$ of a totally disconnected space is $0$.
Accordingly, since	
every compact group $G$ is homeomorphic to $G_0\times G/G_0$  (\cite[Theorem 10.38]{hoffmorr}), where $G_0$ is the connected component of $G$,
we may define the dual rank of a compact group $G$ as the dual rank of its connected component. Thus, we have

\begin{definition}\label{drank} The \emph{dual rank} of a compact group $G$ is given by
\[ \drr(G)=\drr(G_0)=\aleph(G_0)+r_0(\widehat{G_0/G_0^\prime}),\]
where $G_0$ denotes the connected component of $G$ and $\aleph(G_0)$ is the isomorphy invariant as defined above.
\end{definition}

\begin{remarks}\label{remark5}
We remark here some  immediate consequences of the definition of the dual rank of a compact group $G$.
 \begin{enumerate}

\item\label{rem5:ii}   When $G$ is abelian, $G^\prime$ is trivial, and so  $\drr(G)= r_0(\widehat{G_0}).$
\item \label{rem5:iii}
When $G$ is connected,   it follows from    \cite[Theorem 9.55 and Theorem 9.52]{hoffmorr}
that $\drr(G)$ is finite if and only if $\dim G$ is finite. Indeed, these two theorems yield the
following equivalent statements
\begin{enumerate}
\item $\drr(G)=\aleph(G)+r_0(\widehat{G/G^\prime})<\infty.$
\item
$G^\prime$ s a compact Lie group and  $G/G^\prime$ is a finite dimensional compact abelian
group.
\item$\dim G<\infty.$
\end{enumerate}
\item\label{rem5:iv} When $G$ is connected  the quotient group $G/G^\prime$
is a compact connected abelian group. Therefore, according to  whether $r_0(\widehat{G/G^\prime})$ is finite and so $w(G/G^\prime)=w$,  or $r_0(\widehat{G/G^\prime})$ is infinite, our previous observations  in (\ref{rank}) imply, respectively,  that
\begin{equation*}\begin{split}\drr(G)&=\aleph(G)+r_0(\widehat{G/G^\prime}) \quad\text{or}\\
\drr(G)&=\aleph(G)+w{(G/G^\prime}).\end{split}\end{equation*}
\end{enumerate}
\end{remarks}

For the following theorem,  see  \cite[9.24--9.26]{hoffmorr}.

\begin{theorem}\label{HM92526} Let $G$  be a connected compact group,
$Z_0(G)$ be the identity component of the centre $Z(G)$ of $G$,
 and let $\Delta=Z_0(G)\cap G^\prime$.
\begin{enumerate}
  \item The subgroup $\Delta$ is totally disconnected and central, and $G=Z_0(G)G^\prime$.
  \item There exist three continuous
 surjective morphisms \begin{align*}\pi_\g&:\prod_{\gamma<\aleph(G)}S_\gamma\to G^\prime,
\\\mu_\g&:Z_0(G)\times\prod_{\gamma<\aleph(G)}S_\gamma\to G\quad  \mbox{ and }\\ \nu_\g&:G\to \frac{Z_0(G)}{\Delta}\times\prod_{\gamma<\aleph(G)}\frac{S_\gamma}{Z_0(S_\gamma)},\end{align*} where
 each $S_\gamma$, $\gamma<\aleph(G)$, is a simple, simply connected compact Lie group such that
\begin{enumerate}
\item $\mu_\g(z,\mathbf{s})=z\pi_\g(\mathbf{s})$, for each $z\in Z_0(G)$ and $\mathbf{s}\in\prod_{\gamma<\aleph(G)}S_\gamma$,
\item $\nu_\g(g)=\bigl(z\Delta, (s_\gamma Z(S_\gamma))_{\gamma<\aleph(G)}\bigr)$ for every $g \in G$, where  $g=z \pi_\g(\mathbf{s})$, $\mathbf{s}=(s_\gamma)_{\gamma}\in $, is the decomposition of $g$ seen as an element of $Z_0(G)G^\prime$,
\item $\mu_\g$ and $\nu_\g$ have totally disconnected central kernels,
and
\item $\nu_G\circ\mu_G=q$, where
\[q:Z_0(G)\times\prod_{\gamma<\aleph(G)}S_\gamma\to\frac{Z_0(G)}{\Delta}\times\prod_{\gamma<\aleph(G)}\frac{S_\gamma}{Z_0(S_\gamma)}\] is given by $q(z,(s_\gamma)_{\gamma<\aleph(G)})=(z\Delta, (s_\gamma Z_0(S_\gamma))_{\gamma<\aleph(G)}).$
\end{enumerate}
\end{enumerate}
\end{theorem}

\begin{lemma}\label{dim}
Let $G$ be an infinite compact group.
\begin{enumerate}
  \item  $G$ has a finite dual rank if and only if its connected component  $G_0$ has a finite dimension.
  \item If $G$ has an infinite dual rank, then $\drr(G)=w(G_0)$.
  \end{enumerate}
\end{lemma}

\begin{proof} Since $\drr (G)=\drr (G_0)$ by definition, Statement (i) is just a repetition of Remark \ref{remark5}(ii).

We prove the second assertion first for a connected compact group $G;$ we shall be repeatedly looking  at the equalities \begin{equation}
\begin{split}
  w(G) & =w(G^\prime)+w(G/G^\prime)\label{wdrr1}\\
   \drr(G)= & =\aleph(G)+r_0\bigl(\widehat{G/G^\prime}\bigr).\end{split}
\end{equation}

We note  first that,  the existence of the surjective homomorphisms $\pi_\g$ and $\nu_{\g'}$ in Lemma \ref{HM92526}  shows that
\begin{equation}\label{waleph}\aleph(G)\leq w(G^\prime)\leq  \max\{\aleph(G),\omega\}.\end{equation}
It follows that if  $\aleph(G)$ is finite, then
\begin{equation}\label{finite}\aleph(G)\leq w(G^\prime)\leq  \omega;\;
\text{otherwise,}\; \aleph(G)=w(G^\prime)\ge\omega.\end{equation}
Assume now  that $G$ has infinite $\drr$. Then $\aleph(G)$ and  $r_0(\widehat{G/G^\prime})$ cannot be both finite,
and so to prove that $\drr(G)=w(G)$, we look at two separate cases.
\emph{Case 1:}  If $r_0(\widehat{G/G^\prime})$  is infinite, then by  Remark \ref{remark5} (iii),
\[\drr(G)=\aleph(G)+w(G/G^\prime).\] If $\aleph(G)<w,$ then by (\ref{finite}), $w(G^\prime)\le\omega$, and so it is clear that
\[\drr(G)=w(G/G^\prime)+w(G')=w(G).\]
If $\aleph(G)\ge w,$ then  $\aleph(G)=w(G^\prime)$, and so \[\drr(G)=w(G/G^\prime)+w(G')=w(G).\]

\emph{Case 2:}  If $r_0(\widehat{G/G^\prime})$  is finite, then by the observation (\ref{obse}),  $w(G/G^\prime)= \omega$.
Since $\aleph(G)$ must be infinite, we have  by (\ref{finite}), $\aleph(G)=w(G^\prime)$ and a further application of Remark \ref{remark5} (iii) shows that
 \[\drr(G)=\aleph(G)+r_0(\widehat{G/G^\prime})=\aleph(G)=
w(G^\prime)+w(G/G^\prime)=w(G),\]
as wanted.

Since $\drr (G)=\drr (G_0)$, the lemma is proved when $G$ is an infinite  compact group.
\end{proof}

  \begin{lemma}\label{compconn}
    Let $G$ be a connected compact group and let $L$ be a proper, closed, normal subgroup of $G$
    that is not totally disconnected. Then $L$ has a normal supplement $M$, i.e., a proper closed normal subgroup $M$ of $G$ with $G=LM $.
  \end{lemma}

  \begin{proof}
    We apply the structure theorem, Theorem \ref{HM92526}.

		Let	$L_0$ be  the connected component of $L$,		and observe first that  $\nu_\g(L_0)\neq \{1\}$  since $\ker \nu_\g$
		is totally disconnected.
		     Let   \begin{align*}
       p&\colon \frac{Z_0(G)}{\Delta}\times\prod_{j\in J}\frac{S_j}{Z(S_j)} \to \prod_{j\in J} \frac{S_j}{Z(S_j)},\\
       q&\colon \frac{Z_0(G)}{\Delta}\times\prod_{j\in J}\frac{S_j}{Z(S_j)}\to
     \frac{Z_0(G)}{\Delta}\quad \text{ and}\\
      p_j&\colon \frac{Z_0(G)}{\Delta}\times\prod_{j\in J}\frac{S_j}{Z(S_j)} \to \frac{S_j}{Z(S_j)}
    \quad\text{for}\; j\in J, \end{align*} denote the  natural projections. We consider two cases.
		(Remember $\aleph(G)=\card J$.)

\textbf{Case I.  $q\bigl(\nu_\g(L_0)\bigr)$ is trivial.}
Then $p\bigl(\nu_\g(L_0)\bigr)$ is not trivial. Note first that $L_0$ is also a normal subgroup of $G$
(being the connected component of $L$). So $p\bigl(\nu_\g(L_0)\bigr)$ is a non-trivial normal subgroup of
$\prod_{j\in J} \frac{S_j}{Z(S_j)}$.
Since the factors $\frac{S_j}{Z(S_j)} $ are simple Lie groups and simple Lie groups contain no nontrivial connected normal subgroups, we conclude that there is  some $j_0\in J$ such that $p_{j_0}(\nu_\g(L_0))=\frac{S_{j_0}}{Z(S_{j_0})}$. Consider now the kernel of $p_{j_0}$.
If $1_{j_0}$ is the identity in $\frac{S_{j_0}}{Z(S_{j_0})}$,  then
\[\ker (p_{j_0})=\frac{Z_0(G)}{\Delta}\times\prod_{j\in J, j\ne j_0}\frac{S_j}{Z(S_j)}\times \{1_{j_0}\}.\]
Then it is clear that $\ker(p_{j_0})$ is a normal supplement of $\nu_\g(L_0)$ in
the group $\frac{Z_0(G)}{\Delta}\times\prod_{j\in J}\frac{S_j}{Z(S_j)}$.
 So for every $s\in G$ there exist $u\in \ker(p_{j_0})$ and $t\in L_0$ such that $\nu_\g(s)=\nu_\g(t)u$. Thus $u=\nu_\g(t^{-1}s)$ and \[s=t(t^{-1}s)\in L_0(\nu_\g)^{-1}\bigl(\ker(p_{j_0})\bigr).\] It follows that
\[M:=(\nu_\g)^{-1}\bigl(\ker(p_{j_0})\bigr)=\ker\bigl(p_{j_0}\circ\nu_\g\bigr)\]
is a normal supplement of $L_0$, and so of $L$, in $G.$

  \textbf{Case II.  $q\bigl(\nu_\g(L_0)\bigr)$ is nontrivial.}
  Now $q\bigl(\nu_\g(
	L_0)\bigr)$ is a nontrivial subgroup of the compact
  abelian group $\frac{Z_0(G)}{\Delta}$. So let $\chi$ be any   character of $\frac{Z_0(G)}{\Delta}$ that is not trivial on $q\bigl(\nu_\g(L_0)\bigr)$. We claim that $\frac{Z_0(G)}{\Delta}=q(\nu_\g(L_0))\ker \chi$, i.e., $\ker \chi$
  is a normal supplement of $q(\nu_\g(L_0))$ in $\frac{Z_0(G)}{\Delta}$.
    Since the circle group $\T$ has no nontrivial connected subgroups and $q\bigl(\nu_\g(L_0)\bigr)$ is  connected, we must have \[\chi\left(q\bigl(\nu_\g(L_0)\bigr)\right)=\chi\left(\frac{Z_0(G)}{\Delta}\right)=\T.\] We can therefore find for each element $x\in \frac{Z_0(G)}{\Delta}$, an element
    $s\in L_0$
     with $\chi(q(\nu_\g(s)))=\chi(x)$. In other words, \[x=q(\nu_\g(s))\bigl((q(\nu_\g(s)))^{-1}x\bigr)\;\text{ with}\;(q(\nu_\g(s)))^{-1}x\in \ker \chi,\] which is what is required.

    Once we have \[\frac{Z_0(G)}{\Delta}=q(\nu_\g(L_0))\ker \chi\]
		and observe that
\[(q(\nu_\g(L_0))\ker \chi)\times    \prod_{j\in J} \frac{S_j}{Z(S_j)}\subseteq
(\nu_\g(L_0)\bigl(\ker \chi\times     \prod_{j\in J} \frac{S_j}{Z(S_j)}\bigr),\]
it is easy to see that \[(\ker \chi)\times
    \prod_{j\in J} \frac{S_j}{Z(S_j)}\] is a normal supplement of $\nu_\g(L_0)$ in $\frac{Z_0(G)}{\Delta}\times\prod_{j\in J}\frac{S_j}{Z(S_j)}.$ Accordingly, as argued in the previous case,  \[M=(\nu_\g)^{-1}\bigl((\ker \chi)\times
    \prod_{j\in J} \frac{S_j}{Z(S_j)}\bigr)\] is a normal supplement of $L_0$, and so of $L$ in $G$, as wanted.
  \end{proof}

  \begin{lemma}  \label{abelian}
  Let $G$ be a  compact abelian group and let $H$ be a totally disconnected subgroup of $G$.
  If $G/H$ has finite dual rank, then so has $G$. In other words, if $r_0(\widehat{G/H})$ is finite, then so is $r_0(\widehat{G})$.
     \end{lemma}

     \begin{proof}
       If $r_0(\widehat{G})$ is infinite, we can find an infinite set $\{\chi_n \colon n\in \N\}\subset \widehat{G}$ consisting of independent characters with $\chi_n^k\neq 1$ for any $n\in \N$ and $k\in \Z\setminus\{0\}$. Since $\widehat{G/H}$ is isomorphic to $H^\perp= \{\chi\in \widehat{G}\colon \chi(H)=\{1\}\}$ and $r_0(\widehat{G/H})$ is finite,  we see that $\langle \chi_n\rangle \cap H^\perp=\{1\} $ for all but finitely many $n\in \N$. Pick $n_0\in \N$ with  $\langle \chi_{n_0}\rangle \cap H^\perp=\{1\} $.
 Then  $\chi_{n_0}^k\notin H^\perp$  for any $k \in \Z\setminus \{0\}$, hence $\restr{\chi_{n_0}}{H}^k\neq 1$ for every $k \in \Z\setminus\{0\}$.  Having a  character of infinite order, $H$ cannot be totally disconnected (see \cite[Theorem 24.26 or Corollary 24.18]{HR}). We conclude that $r_0(\widehat{G})$ must be  finite.
 \end{proof}

\begin{lemma}
  \label{connker}
  Let $G$ be a compact connected group and let $H$ be a totally disconnected normal subgroup of $G$.
  Then \begin{enumerate}
  \item $w(G/H)=w(G)$.
   \item If $G/H$ has a finite dual rank, then so has $G$.
    \end{enumerate}
          \end{lemma}

  \begin{proof}
  The first assertion follows from \cite[Proposition 12.26]{hoffmorr}.

  We divide the proof of the second assertion into two cases:

  \textbf{Case I: $G$ is semisimple, i.e., $G^\prime=G$}. In this case, the dual rank of $G$ is simply
  $\drr(G)=\aleph(G)$.
 Let $\widetilde{G}= \prod_{j\in J}S_j$  be the simply connected semisimple compact group	whose quotient is
	$G$  by the map $\pi_\g:\widetilde{G}\to G^\prime=G$.
	Then
	$G/H$  is also a  quotient of $\prod_{j\in J}S_j$ by the map $q\circ\pi_\g$
	where $q:G\to G/H$ is the quotient map.
	Since $H$ is totally disconnected, we see that $\ker(q\circ\pi_\g)$ is totally disconnected.
	So, by the uniqueness theorem proved in \cite[Theorem 9.19]{hoffmorr} (i), we see the family of simple Lie groups $\{S_j:j\in J\}$ defining $\aleph(G)$ is the same as the one defining $\aleph(G/H)$. Therefore $\aleph(G)=\aleph(G/H),$ and so the claim is clear.

  \textbf{Case II: $G$ is an arbitrary compact connected group}. In this case \[\drr(G)=\aleph(G)+r_0(\widehat{G/G^\prime}).\]
  Then  $G=Z(G) G^\prime$, since $G=Z_0(G) G^\prime$ by  \cite[Theorem 9.24]{hoffmorr}.
  Since every totally disconnected normal subgroup of a connected group is central (see for instance \cite[Proposition A4.27 or Lemma 6.13]{hoffmorr}), we see that $H$ is central, i.e., $H\subset Z(G)$ .
  Therefore we can write $G/H$ as \[G/H=(Z(G) G^\prime)/H=(Z(G)/H) (G^\prime H/H).\]
	It follows that each of the groups $(Z(G)/H)$ and $(G^\prime H/H)$ has a finite dual rank, since
	the group $G/H$ has a finite dual rank by assumption.
 On one hand, using Lemma \ref{abelian}, we see that $Z(G)$ has a finite dual rank, i.e., $r_0(\widehat{Z(G)})$ is finite. And because $G/G^\prime $ is a quotient of $Z(G) $ since it is topologically isomorphic to $Z(G)/(G^\prime\cap Z(G))$ by the second isomorphism theorem,  we deduce that   $r_0(\widehat{G/G^\prime})\leq r_0(\widehat{Z(G)})$, and so $r_0(\widehat{G/G^\prime})$ is finite.

   On the other hand, as $G^\prime H/H$  is topologically isomorphic to $G^\prime/(H\cap G^\prime)$, it follows
    that $G^\prime/(H\cap G^\prime)$ has a finite dual rank, and so Case I implies that $G^\prime$ has a finite dual rank, i.e., $\aleph(G)$ is finite.
   Accordingly, $\drr(G)=\aleph(G)+r_0(\widehat{G/G^\prime})$ is finite, as wanted.
	  \end{proof}

\begin{lemma}\label{rankZ}
      Every compact connected group with an infinite dual rank $\eta$ has the $\mathcal{Z}^\prime$-property of level $\eta$.
			\end{lemma}

      \begin{proof} By Lemma \ref{dim}. $\mathrm{d-rank}(G)=w(G)$.

 Let $\widehat{G}=\left\{\sigma_\alpha\colon \alpha<w(G)\right\}$ be any enumeration of $\widehat{G} $ with $\sigma_0=1$, the trivial representation. We construct our sequence $\{\pi_\alpha \colon \alpha<w(G)\}$ inductively.
 Choose $\pi_0$ to be any nontrivial unitary representation of $G$ and let $L_1=\ker \pi_0.$

Since the quotient group  $G/\ker \pi_0$ (i.e., the image $\im(\pi_0)$ of $\pi_0$ by the first isomorphism theorem) is a subgroup of a unitary group, and unitary groups are,   topologically,  subspaces  of $\C^{n}$, $n\in \N$,  Lemma  \ref{dim} shows that  $G/\ker \pi_0$ must have  a  finite dual rank.

 The second assertion of Lemma \ref{connker} implies then that $\ker\pi_0$
 cannot be totally disconnected, otherwise $G$ will have a finite dual rank.
  Lemma \ref{compconn} gives then a proper normal subgroup $M_1$ of $G$ such that  $G=\ker \pi_0 M_1.$  We then choose  $\pi_1\in \widehat{G}$ such that $M_1\subset \ker \pi_1$. Thus, $G=L_1\ker \pi_1$.

 Assume now that, for some $\beta<w(G)$, $\{\pi_\alpha\colon \alpha<\beta\}$ have already been chosen with the $\mathcal{Z}^\prime$-property    and define  $L_\beta=\bigcap_{\alpha<\beta}(\ker\sigma_\alpha\cap\ker\pi_\alpha)$.

If we denote $(\ker\sigma_\alpha\cap\ker\pi_\alpha)$ by $H_\alpha$, i.e., $L_\beta=\bigcap_{\alpha<\beta}H_\alpha,$
then the natural map \[gL_\beta\mapsto (gH_\alpha)_{\alpha<\beta}\] shows there is a copy of $G/L_\beta$ in $\prod_{\alpha<\beta} G/H_\alpha,$
and so, by the first isomorphism theorem, we may see the quotient group
$G/L_\beta$  isomorphic to
 a subgroup of \[\prod_{\alpha <\beta} \bigl(G/(\ker\sigma_\alpha)\times G/(\ker \pi_\alpha)\bigr)=\prod_{\alpha <\beta} \bigl(\im(\sigma_\alpha)\times \im(\pi_\alpha)\bigr).\]
 If $\beta <\omega$, then as above Lemma  \ref{dim} shows that $G/L_\beta$   has finite dual rank, and  so  the second assertion of Lemma \ref{connker} shows that $L_\beta$ cannot be totally disconnected.
 If  $\beta\geq\omega$,  we have   $w(G/L_\beta)\leq \beta< w(G)$
since each factor $\im(\sigma_\alpha)\times \im(\pi_\alpha)$ under the product space is a subspace of some $\C^n$. Then, the first assertion of Lemma \ref{connker} shows again that $L_\beta$ cannot be totally disconnected.

So in both cases, applying again Lemma  \ref{compconn}, we find a  proper closed normal subgroup $M_\beta $ of  $G$  with
$G=L_\beta M_\beta$. The representation $\pi_\beta$  is then chosen such that $M_\beta \subset \ker \pi_\beta$. Obviously, $G=L_\beta \ker \pi_\beta$ and the inductive process is concluded.
        \end{proof}

\begin{theorem}\label{connect} Let $G$ be a compact connected group with an infinite dual rank
and a local weight $\eta.$ Then   $A(G)$ has an F$\ell^1(\eta)$-base of type 2 with constants $K_1=K_2=1.$
\end{theorem}

      \begin{proof}
This follows directly from Lemmas
\ref{BL} and \ref{rankZ}.
\end{proof}
	 	
\bigskip

Having F$\ell^1(\eta)$-bases in $A(G)$ for the class of compact groups studied in this section, we may apply as previously  Sections \ref{algebra} and \ref{FC}
 to obtain the centres of the Banach algebras $A(G)^{**}$,  $\uc_2(G)^*$ and $C^*(G)^{\perp}$ and their algebraic structure.
We shall however postpone the statements of these results for next section.
The factorization of $VN(G\times H)$, where $H$ is any amenable locally compact group, proved  in Thereorem \ref{tensors},  enables us to  extend the results and state them for a larger class of groups.

	\section{Factorizations of $VN(G\times H)$ and applications}
Throughout the section, we assume that $G$  is as earlier, i.e., a compact group whose local weight $\eta$ has
uncountable cofinality, an infinite product $\prod_{\alpha<\eta}G_\alpha$  of non-trivial compact groups, or a compact connected group with infinite d-rank $\eta$, and let H be an amenable locally compact group. The Fourier algebra $A(G\times H)$ may not have an F$\ell^1(\eta)$-base, but we shall prove factorization theorems for $VN(G\times H)$ and $\uc_2(G\times H)$ using the factorization theorems available for $VN(G)$ and $\uc_2(G)$.
This yields information on the algebraic structure   of $A(G\times H)^{**}$ and  $\uc_2(G\times H)^*$ as well as the strong Arens irregularity of $A(G\times H)$ and  the centre of $\uc_2(G\times H)^*$.

 If $a \in B(G),$ we define $u \in  B(G\times H)$  by $u(s, t) = a(s);$ that is, $u = a\otimes 1_H,$ where $1_H$ is the function
with the constant value $1$ on $H.$ Also if $T \in  V N(G),$ we consider $T \otimes I\in VN(G\times H)$ where $I$ is
the identity operator on $L^2(H).$   So  \[ T\otimes I( f \otimes h) = (T f) \otimes h\quad\text{ for}\quad
f \in L^2(G)\;\text{ and}\;h \in  L^2(H).\]

It is then very quick to verify the following identities, for any locally compact groups $G$ and $H,$
%\begin{enumerate}\item
 \begin{equation}\label{tensor}(u\otimes v) \cdot (S \otimes P) = (u \cdot S) \otimes (v \cdot P)\end{equation}
for every $u\in B(G), v\in B(H), S\in VN(G)$ and $P\in VN(H)$.
So
 \begin{equation}\label{tensor2}(\Psi\otimes v) \cdot (S \otimes P) = (\Psi \cdot S) \otimes (v \cdot P)\end{equation}
for every $\Psi\in \uc_2(G)^*$, $v\in B(H), S\in VN(G)$ and $P\in VN(H)$.

Let $\mathscr L=\bigcup_{\alpha<\eta}A_\alpha\{a_\alpha^{k}:k=1,...,m\}$  be the F$\ell^1(\eta)$-base in $A(G)$ obtained in Theorem \ref{ex5}, \ref{prods} or
\ref{connect}. So in the first case,
$\mathscr L$ is of type 3, $m\in \N$ is the dimension of the representaions $\pi_\alpha$ picked for the $\mathcal Z$-property
and $a_\alpha^{k}=\pi_\alpha^{1k}, k=1,...,m$, is the first row of the matrix coefficients of the representation $\pi_\alpha$ for each $\alpha<\eta.$
In the second and third cases, $\mathscr L$ is of type 2, $m=1$  and $a_\alpha=\pi_\alpha^{k}$ is any diagonal coefficient of the representation $\pi_\alpha$ for each $\alpha<\eta$
($k$ may depend on $\alpha$).

Recall also from Section \ref{FF}
the sets
\[X_k=\{\Psi_{p k}:p\in \mathscr{C}(\eta)\}\quad\text{and let}\quad X=\bigcup_{k=1}^mX_k,\]
where $\Psi_{p k}$  are the cluster points of $\{a_\alpha^{k}:\alpha<\eta,\;k=1,...,m\}$
which correspond to the cofinal ultrafilters $p$ in $\eta,$ they are in the closed unit ball of
$\uc_2(G)^{*}$.

For each  $p\in \mathscr{C}(\eta)$ and  $ k=1,...,m$, denote  by $\Psi_{pk}\otimes 1_H$ the corresponding weak$^*$-cluster point of  $\{a_\alpha^{k}\otimes 1_H:\alpha<\eta\}$. Let, for each $ k=1,...,m$, $X_k\otimes 1_H$ be the set formed by these cluster points in the unit ball of  $\uc_2(G\otimes H)^{*}$.
For an arbitrary element in $\prod_{k=1}^mX_k\otimes 1_H$, we will not refer to the ultrafilters
to which it corresponds, and write it simply as $\Psi\otimes 1_H$.

\begin{theorem} \label{tensors} Let $G$  be  a compact group whose local weight $\eta$ has
uncountable cofinality, a compact connected group with infinite d-rank $\eta$, or  an infinite product $\prod_{\alpha<\eta}G_\alpha$  of non-trivial compact groups $G_\alpha$,  and let $H$ be an amenable locally compact group. Then
\begin{enumerate}
\item $VN(G\times H)$ has the uniform $\eta$-factorization through  $\prod_{k=1}^mX_k\otimes 1_H$.
\item $\uc_2(G\times H)$ has the $1$-factorization through  $\prod_{k=1}^mX_k\otimes 1_H$.
\end{enumerate}
\end{theorem}

\begin{proof}
(i) We prove the $1$-factorization  of $VN(G\times H)$ through  $\prod_{k=1}^mX_k\otimes 1_H$ when $G$ has the cofinality of its local weight
uncountable.  The other two cases follow in the same way with $m=1.$
The case when $Q\in VN(G\times H)$ has the form $S\otimes P$ with $S\in VN(G)$ and $P\in VN(H)$ is straightforward.
Pick by Theorem \ref{gfactori}, any $\Psi=(\Psi_k)_{k=1}^m\in\prod_{k=1}^mX_k$ and $T=(T_k)_{k=1}^m\in (VN(G))^m$
 such that \begin{equation} \label{zero1} S=\sum_{k=1}^m\Psi_k\cdot T_k=\Psi\cdot T.\end{equation}
Then, using (\ref{tensor2}), we find \[(\Psi\otimes 1_H)\cdot (T\otimes P)=
\sum_{k=1}^m(\Psi_k\otimes 1_H)\cdot (T_k\otimes P)=\sum_{k=1}^m(\Psi_k\cdot T_k)\otimes P=S\otimes P=Q,\]
as required.

Let now $Q\in VN(G\times H)$  be arbitrary and see it by \cite{EfRu90},
as
a weak$^*$-limit of a bounded net $(Q_F)_{F\in \N^{<w}}$, where $\N^{<w}$ is the set of all finite subsets of $\N$ and  \[Q_F=\sum_{i\in F}S_i\otimes P_i,\quad\text{with}\quad S_i\in VN(G),\; P_i\in VN(H)\;\text{ for each}\; i\in F.\]

Let $\mA=A(G)$ and $A_\alpha$ be as defined in Lemma \ref{zcoef}.
For each $i\in F$,    define  $S^i$ on $\mathcal L=\bigcup_{\alpha<\eta} A_\alpha{\bf a}_\alpha$ by
\begin{equation}\label{lambda}\langle S^i, a{\bf a}_\alpha \rangle=\langle S_i, a\rangle,\quad\text{for every}\quad a\in A_\alpha\quad\text{and}\quad\alpha<\eta,\end{equation}
where ${\bf a}_\alpha=(a_\alpha^k)_{k=1}^m=(\pi_\alpha^{1k})_{k=1}^m$.
Let for each $F\in \N^{<w}$, \[R_F=\sum_{i\in F} S^{i}\otimes P_i\] be defined on $\mathcal L\otimes A(H).$
Let $x\in \langle\mathcal L\rangle\otimes A(H)$ and write it as  $x=(\sum_{n=1}^p a_n{\bf a}_{\alpha_n})\otimes b$, where $a_n\in\langle A_{\alpha_n}\rangle$ and $\alpha_1<\alpha_2<...<\alpha_p$.
Accordingly,
\begin{equation}\label{KK}  \begin{split}
|\langle R_F, x\rangle|&= |\langle \sum_{i\in F}\sum_{n=1}^p \langle S^{i}, a_n{\bf a}_{\alpha_n}\rangle\langle P_{i}, b\rangle|=
 |\langle \sum_{i\in F}\sum_{n=1}^p  \langle S_{i}, a_n\rangle\langle P_{i}, b\rangle|\\&=
 %|\langle \sum_{n=1}^p\sum_{i\in F}  \langle S_{i}, a_n\rangle\langle P_{i}, b\rangle|=
 |\langle \sum_{i\in F}   S_{i}\otimes P_i, \sum_{n=1}^pa_n\otimes b\rangle|\\&\le
\|\sum_{i\in F}   S_{i}\otimes P_i\|\| \sum_{n=1}^pa_n\otimes b\|=
\|Q_F\|\| \sum_{n=1}^pa_n\|\| b\|\\&\le \|Q_F\|\sum_{n=1}^p\|a_n\|\| b\|
\le K_2\|Q_F\|\sum_{n=1}^p\|a_n{\bf a}_{\alpha_n}\|\| b\|\\&
\le K_1  K_2\|Q_F\|\sum_{n=1}^p a_n{\bf a}_{\alpha_n}\|\| b\|
= K_1  K_2\|Q_F\|\sum_{n=1}^p a_n{\bf a}_{\alpha_n}\otimes b\|\\&
= K_1  K_2\|Q_F\|x\|,
\end{split}\end{equation}
where  by Lemma \ref{zcoef} since $A(G)$ has the $\mathcal{Z}$-property of level $\eta$, $K_1=1$ is the selective constant and $K_2=
=\sup_\alpha\|(\pi_\alpha^{1k})_{k=1}^m\|\le m$ is the F$\ell^1(\eta)$-base constant.
(In the two other cases, $K_1=K_2=1$.)
In other words, $R_F$ is a bounded operator (by $\| Q_F\|$) on the subspace $\langle\mathcal L\rangle\otimes A(H)$ of $A(G\times H)$.
We may therefore extend $R_F$ to an operator in $VN(G\times H)$ with  $\|R_F\|\le \|Q_F\|.$ We denote this extension again by $R_F$ and consider the net $(R_F)_{F\in \N^{<w}}$. This is a bounded net since $(Q_F)_{F\in \N^{<w}}$ is bounded, and so it weak$^*$-clusters
at some $R\in VN(G\times H).$  We claim that \[(\Psi\otimes 1_H)\cdot R=Q.\]
%\sum_{i\in F}S_i\otimes P_i=Q_F.\]
If $a\in A(G)$ and $b\in A(H),$ Then, $a\in A_\beta$ for some $\beta$ and so for every $\alpha\ge\beta$, we find
\begin{align*}\langle R_F, (a {\bf a}_\alpha)\otimes b\rangle&= \langle  \sum_{i\in F} S^{i}\otimes P_i,(a {\bf a}_\alpha)\otimes b\rangle\\&=
\langle  \sum_{i\in F} S^{i},(a {\bf a}_\alpha)\rangle\langle P_i, b\rangle\\&=
\langle  \sum_{i\in F} S_{i},a\rangle\langle P_i, b\rangle=
\langle   \sum_{i\in F} S_{i}\otimes P_i,a\otimes b\rangle\\&=\langle Q_F, a\otimes b\rangle.\end{align*}

Therefore, \begin{align*}\langle R, (a {\bf a}_\alpha)\otimes b\rangle= \lim_F\langle R_F, (a {\bf a}_\alpha)\otimes b\rangle
=\lim_F\langle Q_F, a\otimes b\rangle=\langle Q, a\otimes b\rangle.\end{align*}
Accordingly,

\begin{align*} \langle (\Psi\otimes 1_H)\cdot R,  a\otimes b\rangle&=
\lim_\gamma\langle ({\bf a}_{h(\gamma)}\otimes 1_H)\cdot R, a\otimes b\rangle\\&=\lim_\gamma\langle  R,
 (a\otimes b) ({\bf a}_{h(\gamma)}\otimes 1_H)\rangle\\&=\lim_\gamma\langle  R,
 (a{\bf a}_{h(\gamma)})\otimes b)\rangle
=\langle Q, a\otimes b\rangle, \end{align*}
as wanted.

To prove that $VN(G\times H)$ has the uniform $\eta$-factorization, we may argue in two ways either as in Theorem \ref{gfactori} or as in Remark \ref{ortho}(ii).
We choose the shorter route of the remark.
Let $\{Q_\lambda\}_{\lambda<\eta}$ be any bounded family of operators in $VN(G\times H)$,
and regard each $Q_\lambda$ as a weak$^*$-limit of a net $(Q_{\lambda F})_{F\in \N^{<w}}$
made as earlier of simple operators \[Q_{\lambda F}=\sum_{i\in F}S_{\lambda i}\times P_{\lambda i},\quad F\in \N^{<w}.\]
Partition $\eta$ into $\eta$ many subsets $I_\lambda$ such that each $I_\lambda$ is cofinal in $\eta$.
Take
for each  $\lambda<\eta$ and each $ k=1,...,m$,  an element $\Psi_{k\lambda}\otimes 1_H$  in $X_k\otimes 1_H$ corresponding to some ultrafilter  $p\in \mathscr{C}(I_\lambda)$,
i.e, an element which comes as a weak$^*$-cluster point of  $\{a_\alpha^k\otimes 1_H:\alpha\in I_\lambda\}$, where $a_\alpha^k=\pi_\alpha^{1k}$ are picked as in Theorem \ref{ex5} or
$a_\alpha^k=\pi_\alpha^{k}$ are picked as in Theorems \ref{prods} and \ref{connect} depending on the type of our base.
The elements  $\Psi_{k\lambda}\otimes 1_H$ are in the unit ball of  $\uc_2(G\otimes H)^{*}$ and belong to $C^*(G\times H)^{\perp}$. Put $\Psi_\lambda\otimes 1_H=(\Psi_{k\lambda}\otimes 1_H)_{k=1}^m$.
($\Psi_{k\lambda}$ and $\Psi_{pk}$ are not to be confused, one corresponds to the cofinal ultrafilter $p$ on $\eta$
while the other corresponds to the cofinal subset $I_\lambda$ of $\eta$.)
For each $\lambda<\eta$,  $F\in \N^{w}$ and $i\in F$,  instead of (\ref{lambda}) define  $S^{\lambda i}$ on $\mathcal L=\bigcup_{\alpha<\eta} A_\alpha{\bf a}_\alpha$ by
\begin{equation}\label{idstr}\langle S^{\lambda i}, a{\bf a}_\alpha \rangle=\begin{cases}\langle S_{\lambda i}, a\rangle,\quad\text{if}\quad \alpha\in I_\lambda\\
0\quad\quad\quad\;\text{if}\quad \alpha\notin I_\lambda,\end{cases}\end{equation}
and let \[R_{\lambda F}=\sum_{i\in F} S^{\lambda i}\otimes P_{\lambda i}.\]
Then the argument in (\ref{KK}) shows that $R_{\lambda F}$ is a bounded operator on $\langle\mathcal L\rangle\otimes A(H)$
and so may be extended to an operator in $VN(G\times H)$. The bounded net $(R_{\lambda F})_{F\in \N^{<w}}$ so obtained, weak$^*$-clusters
at some $R_{\lambda}\in VN(G\times H).$  The operator $R_\lambda$
satisfies \[\langle R_\lambda, (aa_\alpha)\otimes b\rangle=\begin{cases} \langle Q_\lambda, a\otimes b\rangle\quad\text{if}\;\alpha\in
I_\lambda\\
0\quad\quad\text{if}\; \alpha\notin I_\lambda.\end{cases}\]
Now let $R=\sum_{\lambda<\eta} R_\lambda$ be first defined on $\mathcal L\otimes A(H)$ and continue as in Remark \ref{ortho}(ii).

(ii) Since $G\times H$ is amenable, $A(G\times H)$ has a bai.  So we may proceed precisely as in Theorem \ref{lucfactori} for the 1-factorization of $\uc_2(G\times H)$ since now that we have the 1-factorization of $VN(G\times H)$.
\end{proof}

\begin{remark} \label{notquit} Due to the availability of the F$\ell^1(\eta)$-bases, together with the factorization theorems,  Remark \ref{ortho} was also necessary for the algebraic structure reached in the previous sections. Next we check that this remark is also valid here.

So let our group have the form $G\times H$ as in Theorem \ref{tensors}, and $Q$ be in $VN(G\times H)$ or in $\uc_2(G\times H)$ and $I$ be any cofinal subset of $\eta$. Then there exists $R\in VN(G\times H)$ or $\uc_2(G\times H),$ respectively
such that
\begin{equation}\label{idstru} \Psi_p\cdot R=\begin{cases} Q,\quad\text{if}\quad I\in p\\0,\quad\;\;\text{otherwise}.\end{cases}\end{equation}
To see this, suppose that $Q$ be in $VN(G\times H)$ and let $p$ be any ultrafilter having $I$ as a member and $\Psi_p$ be any element in   $\prod_{k=1}^mX_k\otimes 1$.
Let $(\sum_{i\in F}S_i\otimes P_i)_{ F\in \N^w}$ be the net in $VN(G\times H)$ with $Q$ as a weak$^*$-limit.
As in (\ref{idstr}), define for each ${ F\in \N^w}$ and $i\in F$, $S^i$ on $\mathcal L$ by
\begin{equation*}\langle S^{i}, a{\bf a}_\alpha \rangle=\begin{cases}\langle S_{i}, a\rangle,\quad\text{if}\quad \alpha\in I\\
0\quad\quad\quad\;\text{if}\quad \alpha\notin I,\end{cases}\end{equation*}
and let \[R_{F}=\sum_{i\in F} S^{i}\otimes P_{i}.\]

Proceed then as in the proof of Theorem \ref{tensors} to see that $(R_F)_F$ is bounded, and so has a weak$^*$-cluster operator $R\in VN(G\times H)$  having the property required in (\ref{idstru}).

If $Q=S\cdot u\in \uc_2(G\times H)$ with $S\in VN(G\times H)$ and $u\in A(G\times H,)$ then the  factorization of $S$ by $\Psi_p$ and $R$ as in (\ref{idstru}) yields the same type of factorization of $Q$ by $\Psi_p$ and $R\cdot u\in \uc_2(G\times H)$.
\end{remark}

\subsection{Algebraic structure}\label{alst}
As it is with the group algebra in Section \ref{gr}, nontrivial  left  ideals  of finite  dimensions exist in $A(G)^{**}$ and $\uc_2(G)^*$ for any locally compact group (not necessarily amenable in this case), they are generated by the topological invariant means (TIMs),  by the annihilators of $\uc_2(G)$ when $G$ is not compact, or  the mixture of both.
Since there are  $2^{2^\eta}$ TIMs in $A(G)^{**}$ and $\uc_2(G)^*$  when $G$ is non-discrete
(see \cite[Theorem 3.3]{chou82} for metrizable $G$, and \cite[Theorem 5.9]{Hu95} or \cite[Corollary 5.4]{enar2} for nonmetrizable $G$),
at least this many distinct left ideals of finite dimension exists in $A(G)^{**}$
and $\uc_2(G)^*$ when $G$ is non-discrete.
For more details, see \cite{FiMe}. But as we shall soon see, there are also plenty of distinct left ideals of infinite dimension
when the group is as in Theorem \ref{tensors}.

The part on  $A(G\times H)^{**}$ and $\uc_2(G\times H)^*$ in Theorem \ref{above13} made up the main results in \cite{MMM1}. However, our proofs here are much more simplified.

\begin{theorem}\label{above13}  Let $G$ be a compact group such that $\eta=w(G)$ has uncountable cofinality and $H$ be an amenable locally compact group. Let $\mathscr S$ be any of the algebras
$A(G\times H)^{**}$, $\uc_2(G\times H)^*$ or $C^*(G\times H)^{\perp}$.
Then
\begin{enumerate}
\item  there are at least $2^{2^\eta}$ many right cancellable sets in $\mathscr S$.
\item  there are at least $2^{2^{\eta}}$ many  left ideals in $\mathscr S$ with dimension at least  $2^{2^{\eta}}$.
\item  the dimension of any principal right ideal,  and so of any nonzero right ideal, in $\mathscr S$
 is at least $2^{2^\eta}.$
\end{enumerate}
\end{theorem}

\begin{proof} Since $X$ is contained in   $C^*(G\times H)^{\perp}$ and since we have  the $1$-factorization of $VN(G\times H)$ and
$\uc_2(G\times H)$ by Theorem \ref{tensors}, using the proofs of Theorems \ref{can}(ii), \ref{ride}  and
\ref{lide}(ii) with Remarks \ref{luc}  and \ref{notquit} in mind gives the claimed properties.
\end{proof}

\begin{theorem}\label{above22}  Let $G$ be a compact group which is either connected with infinite d-rank $\eta$ or is an infinite product $\prod_{\alpha<\eta}G_\alpha$  of non-trivial compact groups, and $H$ be an amenable locally compact group. Let $\mathscr S$ be any of the algebras
$A(G\times H)^{**}$, $\uc_2(G\times H)^*$ or  $C^*(G\times H)^{\perp}$.
Then
\begin{enumerate}
\item  there are at least $2^{2^\eta}$ many right cancellable elements in $\mathscr S$.
\item  there are at least $2^{2^{\eta}}$ many  principal left ideals in $\mathscr S$ with
trivial pairwise intersections and with  dimension at least  $2^{2^{\eta}}$.
\item  the dimension of any principal right ideal,  and so of any nonzero right ideal, in $\mathscr S$
 is at least $2^{2^\eta}.$
\end{enumerate}
\end{theorem}

\begin{proof} This follows with the same proof as above using the proofs of Theorems  \ref{can}(i), \ref{ride} and  \ref{lide}(i) and Remarks \ref{luc} and \ref{notquit}.
\end{proof}

\subsection{Topological centres}\label{topcen2}
Statements (i) and (ii) in the following theorem were proved in \cite[Corollary 4.7]{MMM1}
for the case when $G$ is compact non-metrizable with the cofinality of $w(G)$ uncountable.
The condition that $\omega(H)\le w(G),$ when $H$ is not necessarily compact, was missing in  \cite[Corollary 4.7]{MMM1}. This condition guarantees that $A(G\times H)$ has  Mazur property of level $w(G)$, required for Theorems \ref{Lcentre} and \ref{lucentre} to  apply.
This is now corrected in the following theorem.
In Statements (i) and (ii), we also simplify and improve the main result proved by Lau and Losert in \cite{LaLo05}, which was  proved  with $\eta=w$, each $G_\alpha$ is a non-trivial metrizable compact group and $H$ is a second countable locally compact amenable group.

The rest of the results are new.

\begin{theorem} \label{GtimesH} Let $G$ be a compact group with local weight $\eta=w(G)$ and $H$ be  an amenable locally compact group with $w(H)\le \eta.$ Then in the following cases
\begin{enumerate}
\item[(i)] $G$ is  non-metrizable such that $\eta$ has uncountable cofinality,
\item[(ii)]  $G$ is an infinite product $\prod_{\alpha<\eta}G_\alpha$  of non-trivial compact groups such that  $\sup \omega(G_\alpha)\le \eta$,
\item[(iii)]   $G$ is connected with infinite dual rank $\eta$,
\end{enumerate}
\begin{enumerate}
\item  the topological  centre of $A(G\times H)^{**}$ is $A(G\times H).$
\item  the topological centre of $\uc_2(G\times H)^*$ is $B(G\times H)$.
\item the topological centre of $C^*(G\times H)^{\perp}$ is $\{0\}$.
\end{enumerate}
\end{theorem}

\begin{proof}  Note first that the local weight of $G$ is $\eta$ in each of the three cases. And as already mentioned the condition $w(H)\le w(G)$ guarantees that
$A(G\times H)$ has Mazur property of level $\eta.$  Since Theorems \ref{ex5},  \ref{prods} and \ref{connect}  provide F$\ell^1(\eta)$-bases in $A(G)$ for each of these groups $G$, we have the $\eta$-factorization of $A(G\times H)^*$ by Theorem \ref{tensors}. Therefore, Statements (i) and (ii) follow by Theorems \ref{Lcentre}(ii) and \ref{lucentre}(iii), respectively.

Statement (iii) follows from Theorem \ref{lucentre}(ii) since $X\subseteq C^*(G\times H)^{\perp}$ and $C^*(G\times H)\cap B(G\times H)=\{0\}.$
\end{proof}

\begin{corollary}
 Let $G$ be a compact group as in Theorem \ref{GtimesH}.
Then
 $A(G\times H)$ is strongly Arens irregular.
\end{corollary}

\begin{finalremark}\label{finalremark}
	 \begin{enumerate}
	 \item For simplicity let $H$ be trivial. Then Theorem \ref{GtimesH} leaves out the  case when $w(G)=\sup \omega(G_\alpha)$ has countable cofinality and $w(G)>\eta$. In fact,
We still do not know about the interesting case when $G$  is compact, non-metrizable with countable cofinality.
  We also do not know what happens when $\eta$ is finite  and  $w(G)$ is countable as it is the case if $G=SU(3)$, which is still undecided.
On the other hand, $A(G)$ is strongly Arens irregular when $G=SU(2)$ as claimed in \cite{losert2}.

\item Note that in all the examples studied previously, for the cardinality of the F$\ell^1$-bases to be large, it is essential in each case that either $\kappa(G)$, $|G|$ or $w(G)$ to be as large.
So our approach (as well as most of the articles cited above except \cite{IPU87},
\cite{LaLo93} and \cite{Loetal}) cannot be applied to give the strong Arens
irregularity of $L^1(G)$ and $M(G)$ when $G$ is compact, and that of $A(G)$ and $B(G)$ when $G$ is discrete and amenable.
In the case of the group algebra,
one needs to start first with
the topological centre of $\luc(G)^*$ being $M(G)$ (since $\luc(G)=\mathscr{CB}(G)$), then deduce that the topological centre of $L^\infty(G)^*$
is $L^1(G)$ as it can be seen in the proofs provided in \cite{IPU87} and
\cite[Theorem 1]{LaLo88}.

One encounters the same problem with the Fourier algebra of an amenable discrete group. This is proved in \cite[Theorem 6.5 (i)]{LaLo93}.
 By analogy to the group algebra,  one needs  to notice first that $\uc_2(G)^*$ is $B(G)$ (since $\uc_2(G)=C^*(G)$),
  then use \cite[Theorem 6.4]{LaLo93} to deduce that the topological centre of $VN(G)^*$
 is $A(G)$. The approach of \cite[Theorem 6.4]{LaLo93} is not straightforward.

Another approach can be found \cite[Theorem 2.2]{balapy}, where a weakly sequentially complete Banach algebra, containing a bounded  approximate identity and is an ideal in its second dual,
as it is the situation here with $L^1(G)$ and $A(G)$, is always strongly Arens
irregular. See also our recent paper \cite{EFG2}.

The same remark in these situations concerns the theorems proved in Section \ref{algebra}
on the algebraic structure of these algebras.
Our  approach does not cover $L^1(G)^{**}$, $M(G)$  and $M(G)^{**}$ when $G$ is compact, neither $A(G)^{**}$, $B(G)$, $B_\rho(G)$
and $B(G)^{**}$ when $G$ is discrete, and cannot lead to any algebraic structure in these algebras. We hope to return to this situation in the future.
\item
 The other two known groups for which the centres of  $A(G)^{**}$ and $\uc_2(G)^*$ are
$A(G)$ and  $B(G),$ respectively,  are  any locally compact, second countable group $G$ having the closure of its commutator subgroup $\overline{[G,G]}$ not open, and the special unitary group $SU(2)$,   see \cite{LaLo93} and \cite{losert2}.
We have not been able to construct  F$\ell_1$-bases in $A(G)$ with  these two groups, and we believe a further
refinement of our approach is necessary to deal with these two groups. We hope to return to this situation in the future.
\end{enumerate}
\end{finalremark}

\noindent {\sc Acknowledgement.}
This paper was written during several visits of the
first author  to the  University of Jaume I in Castell\'on in
his off-teaching periods. The work was partially supported several times by
the University of Jaume I. This support is gratefully acknowledged. The work would never have
been completed (or even started) without such a support.

 He would also like to thank Jorge Galindo for
his hospitality and all the folks at the department of mathematics
in Castell\'on.

%\begin{thebibliography}{99}

%\begin{thebibliography}{99}

\bibliographystyle{amsplain}
%\bibliography{../bibliografias/bibrep}

\begin{thebibliography}{10}



\bibitem{Arens51}
R.~Arens, \emph{The adjoint of a bilinear operation}, Proc.~Amer.~Math.~Soc.
  \textbf{2} (1951), 839--848.
	
\bibitem{A2} {R. Arens,} {\em Operations induced in function classes,} Monatsh. Math.  55 (1951), 1--19.

	
\bibitem{arsac76} G. Arsac , \emph{Sur l'espace de Banach engendr\'e par les coefficients d'une repr\'esentation uni-
taire,} Publ. D\'ep. Math. (Lyon) \textbf{13} (1976), no. 2, 1--101.



	
	
	\bibitem{BM} J.W. Baker and P. Milnes, {\em.
The ideal structure of the Stone-\v Cech compactification of a group,}
Math. Proc. Cambridge Philos. Soc. \textbf{82} (1977), no. 3, 401--409.

\bibitem{balapy} J.W. Baker, A. T. -M. Lau and J. S. Pym, \emph{ Module
homomorphisms and topological centres associated with weakly sequentially
complete Banach algebras,} J. Funct. Anal. {\bf158} (1998), 186--208.


\bibitem{BJM} {J. F. Berglund, H. D. Junghenn \and P. Milnes,}
 {\em Analysis on Semigroups: Function Spaces, Compactifications,
Representations,}  Wiley, New York (1989).


\bibitem{BD} {F. F. Bonsall and J. Duncan,} \emph{Complete normed algebras,} Springer--Verlag, Berlin, 1973.

\bibitem{BF} A. Bouziad and M. Filali, \emph{On the size of quotients of function spaces on a topological group,} Studia Math. 202 (2011), no. 3, 243--259.

	\bibitem{pym-et-al:one-pt}
T.~Budak, N.~I{\c{s}}{\i}k, and J.~Pym, \emph{Minimal determinants of
  topological centres for some algebras associated with locally compact
  groups}, Bull. Lond. Math. Soc. \textbf{43} (2011), 495--506.


 \bibitem{chou69} C. Chou, {\em On the size of the set of left invariant means on a semigroup,} Proc. Amer. Math. Soc. 23 (1969), 199--205.


 \bibitem{chou82} C. Chou, {\em Topological invariant means on the von Neumann algebra $VN(G),$}
Trans. Amer. Math. Soc. 273 (1982), no. 1, 207--229.

\bibitem{CiYo}
P.~Civin and B.~Yood, \emph{The second conjugate space of a {B}anach algebra as
  an algebra}, Pacific J. Math. \textbf{11} (1961), 847--870.

\bibitem{Comfort1} W. W. Comfort, \emph{A survey of cardinal invariants,} General Topology and Appl. 1
(1971), no. 2, 163--199.


\bibitem{Comfort} W. W. Comfort, \emph{ Topological groups,} In: Handbook of Set-Theoretic Topology, North-Holland, Amsterdam, 1984, pp. 1143--1263.

\bibitem{dales}
H.~G. Dales, \emph{Banach algebras and automatic continuity}, The Clarendon
  Press, Oxford University Press, New York, 2000.

\bibitem{DaLa05}
H.~G.~Dales and A.~T.-M. Lau, \emph{The second duals of {B}eurling algebras},
Memoires Amer. Math. Soc.\textbf{177}(836) (2005).




\bibitem{dales-lau-strauss}
  H. G. Dales, A. T.-M. Lau, D. Strauss, \textit{ Banach algebras on semigroups and on their compactifications,} Mem. Amer. Math. Soc. 205 (2010), vi+165 pp.


 \bibitem{DLS} H. G. Dales, A. T.-M. Lau, D. Strauss, \textit{
Second duals of measure algebras,} Dissertationes Math. 481 (2012), 1--121.


\bibitem{DHi}
D. E. Davenport and N. Hindman, \emph{A proof of van Douwen's right ideal theorem}, Proc. Amer. Math. Soc. 113 (1991), no. 2, 573--580


		
\bibitem{day} M. M. Day, \emph{ Amenable semigroups}, Illinois J. Math. 1 (1957), 509--44.

\bibitem{der} A. Derighetti,  \emph{ Convolution operators on groups}, Lecture Notes of the Unione Matematica Italiana, 11. Springer, Heidelberg; UMI, Bologna, 2011.




\bibitem{dmm} A. Derighetti,  M. Filali, M. Monfared, \emph{
On the ideal structure of some Banach algebras related to convolution operators on $L^p(G)$,} J. Funct. Anal.215 (2004), no.2, 341--365.


\bibitem{DH} J. Duncan and S. A. R. Hosseiniun, \emph{ The second dual
of a Banach algebra,} Proc. Roy. Soc. Edinburgh Sect. A  84 (1979),  309--325.



\bibitem{EfRu90}
E.~G.~Effros and Z-J.~Ruan, \emph{On approximation properties for operator spaces},
Internat. J. Math. \textbf{1} (1990), 163--187.



\bibitem{EngelTop} R. Engelking, \emph{General Topology},  Translated from the Polish by the author, Second edition, Sigma Series in Pure Mathematics 6, Heldermann Verlag, Berlin, 1989.


\bibitem {EF} { M. Eshaghi Gordji and M. Filali,} {\em Arens regularity of module actions,} Studia Math. 181 (2007), no. 3, 237--254.



\bibitem {EFG} {R. Esmailvandi, M. Filali and J. Galindo,} {\emph Arens regularity of ideals of the group algebra of a compact Abelian group,}
 Proceedings of the Royal Society of Edinburgh, Published online by Cambridge University Press:  27 October 2023.

\bibitem {EFG2} {R. Esmailvandi, M. Filali and J. Galindo,} {\emph
Special subsets of discrete groups and  Arens regularity  of ideals in Fourier and group algebras
,}
preprint, 2024.







\bibitem{Eymard64}
P.~Eymard, \emph{L'alg\`ebre de {F}ourier d'un groupe localement compact},
  Bull. Soc. Math. France \textbf{92} (1964), 181--236.
	
	\bibitem{FCamb}
M. Filali, {\em The uniform compactification of a locally compact abelian group,}  Math. Proc. Cambridge Philos. Soc. 108 (1990), no. 3, 527--538.

	

\bibitem{F92}
M. Filali, {\em The ideal structure of some Banach algebras,} Math. Proc. Cambridge Philos. Soc. 111 (1992), no. 3, 567--576.

\bibitem{F99} {M. Filali,}  Finite-dimensional left ideals in some algebras associated with a locally compact group. Proc. Amer. Math. Soc. 127 (1999), no. 8, 2325--2333

\bibitem{Fi99} M. Filali, {\em Finite dimensional right ideals in algebras associated to a locally compact group,}  Proc. Amer. Math. Soc. 127 (1999) 1729--1734.

\bibitem{Fi1999}{  M. Filali,} {\em On a conjecture of Chou and applications,} Topology Proc. 24 (1999) 139--147.

\bibitem{Fi}
M. Filali, {\em On the ideal structure of some algebras with an Arens product}, Recent progress in functional analysis (Valencia, 2000), 289--297, North-Holland Math. Stud., 189, North-Holland, Amsterdam, 2001



\bibitem{F}{  M. Filali,} {\em t-Sets and some algebraic properties in $\beta S$ and $\ell_\infty(S)^*$),} Semigroup Forum  65,
(2002) 285--300.


\bibitem {FG} {M. Filali and J. Galindo,} {\em Extreme Non-Arens Regularity of the Group Algebra},  Forum Math.  30, No. 5, 1193--1208 (2018).

 \bibitem {FG17} {M. Filali and J. Galindo,} {\em Algebraic structure of semigroup compactifications: Pym's and Veech's theorems and strongly prime points,} J. Math. Anal. Appl. 456 (2017), no. 1, 117--150.

\bibitem {FGsurvey}{M. Filali and J. Galindo,} {\em  $\ell^1$
-bases in Banach algebras and Arens irregularities in harmonic
analysis},
Banach algebras and applications. Proceedings of the international conference held at the University of Oulu, July 3--11, 2017,
De Gruyter Proceedings in Mathematics. Berlin: De Gruyter, 37 p. (2020).


\bibitem {enar1} {M. Filali and J. Galindo,} {\em  On the extreme non-Arens regularity of Banach algebras,} J. Lond. Math. Soc. (2) 104 (2021), no. 4, 1840--1860.


\bibitem {enar2} {M. Filali and J. Galindo,} {\em  Orthogonal $\ell_1$-sets and extreme non-Arens regularity of preduals of von Neumann algebras,} J. Math. Anal. Appl. 512 (2022), no. 1, Paper No. 126137, 21 pp.


\bibitem {rideal} {M. Filali and J. Galindo,} {\em Number of right ideals in the second dual
of algebras in harmonic analysis}, preprint 2024.



\bibitem {FiMe}
M. Filali and M. S. Monfared,  Finite-dimensional left ideals in the duals of introverted spaces. Proc. Amer. Math. Soc. 139 (2011), no. 10, 3645--3656


\bibitem {MMM1} {M. Filali, M. Neufang and M. Monfared,} \emph{ Weak factorizations of operators in the group von Neumann algebras of certain amenable groups and applications,} Math. Ann. {\bf351} (2011), no. 4, 935--961.


\bibitem {MMM2} {M. Filali, M. Neufang and M. Monfared,} \emph{
 On ideals in the bidual of the Fourier algebra and related algebras,} J. Funct. Anal. {\bf 258}(2010), no. 9, 3117--3133.





\bibitem{FiPy03}
M.~Filali and J.~S. Pym, \emph{Right cancellation in {LUC}-compactification of
  a locally compact group}, Bull. London. Math. Soc. \textbf{35} (2003),
  128--134.

\bibitem{FiSa07}
M.~Filali and P.~Salmi, \emph{One-sided ideals and right cancellation in the
  second dual of the group algebra and similar algebras}, J. London Math. Soc.
  \textbf{75} (2007), 35--46.

\bibitem{FiSa07a}
\bysame, \emph{Slowly oscillating functions in semigroup compactifications and
  convolution algebras}, J. Funct. Anal. \textbf{250} (2007), 144--166.

\bibitem{FiSa-b}
\bysame, \emph{Topological centres of the weighted convolution algebras}, J. Funct. Anal.
 278, No. 11, 22 p. (2020). 				

 \bibitem{FiSa}
\bysame, \emph{ Right cancellation, factorisation and right isometries}, preprint 2015.

\bibitem{filali-singh}
M.~Filali and A.~I. Singh, \emph{Recent developments on Arens regularity and
  ideal structure of the second dual of a group algebra and some related
  topological algebras}, General topological algebras (Tartu, 1999), Math.
  Stud. (Tartu), vol.~1, Est. Math. Soc., 2001, pp.~95--124.
	
	
\bibitem{FV} {M. Filali and T. Vedenjuoksu,}
 {\em Extreme non-Arens regularity of semigroup algebras}, Topology Proc. {\bf 33} (2009), 185--196.
	
\bibitem{Ghah}	F.
Ghahramani,
\emph{Weighted group algebra as an ideal in its second dual space,}
Proc. Amer. Math. Soc. 90 (1984), no. 1, 71--76.


\bibitem{GhaLa} F. Ghahramani \and J. Laali,
\emph{Amenability and topological centres of the second duals of Banach algebras,}
Bull. Austral. Math. Soc. 65 (2002), no. 2, 191--197.

\bibitem{GMM} F. Ghahramani, J. P. McClure and M. Meng,  \emph{
On asymmetry of topological centers of the second duals of Banach algebras,}
Proc. Amer. Math. Soc. 126 (1998), no. 6, 1765--1768.




\bibitem{Ed}   E. E. Granirer, \textit{Density theorems for some linear subspaces and some C*-subalgebras of $VN(G)$,} Symposia Mathematica, Vol. XXII (Convegno sull'Analisi Armonica e Spazi di Funzioni su Gruppi Localmente Compatti, INDAM, Rome, 1976), pp. 61--70. Academic Press, London, 1977.



\bibitem{Ed2} E. E. Granirer, \emph{
On group representations whose C*-algebra is an ideal in its von Neumann algebra,}
Ann. Inst. Fourier (Grenoble) 29 (1979), no. 4, v, 37--52.



\bibitem{granirer}
  E. E. Granirer, \textit{Day points for quotients of the {F}ourier
    algebra {$A(G)$},  extreme nonergodicity of their duals and
    extreme non-{A}rens  regularity}, Illinois J. Math. 40
  (1996), no. 3, 402--419.



\bibitem{Green} F. P. Greenleaf, {\em Invariant means on topological groups,} Van Nostrand Math. Studies, Van Nostrand Reinhold, 1969.



\bibitem{Gr}
N. Groenbaek, \emph{Amenability of weighted convolution algebras on locally compact groups,}
Trans. Amer. Math. Soc. 319 (1990), no. 2, 765--775

\bibitem{gros} M. Grosser, \emph{Bidualr\"aume und Vervollst\"andigungen von
Banach-moduln,}  Lecture notes in Mathematics, Vol. 717,  Springer-Verlag, Berlin, 1979.



\bibitem{grosser-losert}
M. Grosser and  V. Losert, \emph{The norm
  strict bidual of a Banach algebra and the dual of $C_u(G)$,}
  Manuscripta Math. \textbf{45} (1984), 127--146.



\bibitem{HR} E. Hewitt and K. A. Ross, \emph{ Abstract harmonic analysis I, II},
\newblock Springer-Verlag,  Berlin, 1963, 1970.


 \bibitem{her} C. Herz,  \emph{ Harmonic synthesis for subgroups,} Ann. Inst. Fourier (Grenoble) 23 (1973), no. 3, 91--123.



\bibitem{HS}  N. Hindman and D. Strauss, {\em Algebra in
the Stone-\v Cech Compactification,} de Gruyter Exp. Math. {\bf 27},
Walter de Gruyter, Berlin, (1998).

\bibitem{hoffmorr}
K. H. Hofmann  and S. A. Morris,
\newblock {\em
     The structure of compact groups. A primer for the student---a handbook for the expert.}, 2nd edition.
     \newblock Walter de Gruyter \& Co., Berlin, 2006.




\bibitem{Hu95}
Z.~Hu, \emph{On the set of topologically invariant means on the von {N}eumann
  algbera ${VN}({G})$}, Illinois J. Math. \textbf{39} (1995), 463--490.



\bibitem{Hu97}
\bysame, \emph{Extreme non-{A}rens regularity of quotients of the {F}ourier
  algebra ${A}({G})$}, Colloq. Math. \textbf{72} (1997), 237--249.

\bibitem{Hu02}
\bysame, \emph{Inductive extreme non-{A}rens regularity of the {F}ourier
  algebra ${A}({G})$}, Studia Math. \textbf{151} (2002), 247--264.

\bibitem{Hu06}
\bysame, \emph{Open subgroups and the centre problem for the {F}ourier
  algebra}, Proc. Amer. Math. Soc. \textbf{134} (2006), 3085--3095.

\bibitem{HuNe06}
Z.~Hu and M.~Neufang, \emph{Decomposability of von Neumann algebras and the
  {M}azur property of higher level}, Canad. J. Math. \textbf{58} (2006),
  768--795.

\bibitem{HuMeTr} Z. Hu, M, M. Sangani and T. Traynor, {\em On character amenable Banach algebras,} Studia Math. 193 (2009), no. 1, 53--78.



\bibitem{IPU87}
N.~I\c{s}ik, J.~Pym, and A.~$\ddot{\text U}$lger, \emph{The second dual of the
  group algebra of a compact group}, J. London Math. Soc. \textbf{35} (1987),
  135--158.
	
\bibitem{Jech78}
T. Jech, S\emph{et Theory}, Academic Press, New York, 1978.

\bibitem{KaLa04} E.~Kaniuth and A.~T.-M. Lau, \emph{Fourier algebras and
amenability}, Banach algebras and their applications, 181-192, Contemp.
Math., 363, Amer. Math. Soc., Providence, RI, 2004.

\bibitem{KLP08}
E. Kaniuth, A. Lau, and J. Pym, {\em On $\varphi$-amenability of Banach algebras,} Math. Proc. Cambridge Philos. Soc. 144 (2008), 85--96.

\bibitem{KLPy08}
E. Kaniuth, A. Lau, and J. Pym, {\em On character amenability of Banach algebras,} J. Math. Anal. Appl. 344 (2008), 942--955.


\bibitem{La79}
A.~T.-M. Lau, \emph{Uniformly continuous functionals on the Fourier algebra of any locally compact group,} Trans. Amer. Math. Soc. 251 (1979), 39--59.


\bibitem{Lau83}
A.~T.-M. Lau, \emph{
 Analysis on a class of Banach algebras with applications to harmonic analysis on locally compact groups and semigroups,} Fund. Math. 118 (1983), no. 3, 161--175


\bibitem{Lau86}
A.~T.-M. Lau, \emph{Continuity of Arens multiplication on the dual space of bounded uniformly continuous functions
on locally compact groups and topological semigroups}, Math. Proc. Cambridge Phil. Soc.
\textbf{99}(1986), 273--283.

\bibitem{Lau87}
A.~T.-M. Lau, \emph{Uniformly continuous functionals on {B}anach algebras}, Colloquium Math. \textbf{51}(1987), 195--205.


\bibitem{Lau94}  A.~T.-M. Lau,
\emph{Fourier and {F}ourier-{S}tieltjes algebras of a locally compact group and
amenability},  Topological vector spaces, algebras and related areas
(Hamilton, ON, 1994), 79-92, Pitman Res. Notes Math. Ser., 316, Longman
Sci. Tech., Harlow, 1994.





\bibitem{LaLo88}
A.~T.-M. Lau and V.~Losert, \emph{On the second conjugate algebra of
  ${L}^1({G})$ of a locally compact group}, J. London Math. Soc. \textbf{37}
  (1988), 464--470.

\bibitem{LaLo93}
\bysame, \emph{The ${C}^*$-algbera generated by the operators with compact
  support on a locally compact group}, J.~Funct. Anal. \textbf{112} (1993),
  1--30.

\bibitem{LaLo05}
\bysame, \emph{The center of the second conjugate algebra of the {F}ourier
  algebra for infinite products of groups}, Math. Proc. Camb. Phil. Soc.
  \textbf{138} (2005), 27--39.

  \bibitem{LMP} Lau, A. T.; Milnes, P.; Pym, J. S. Locally compact groups, invariant means and the centres of compactifications. J. London Math. Soc. (2) 56 (1997), no. 1, 77--90.


\bibitem{LaPa}
A. T.-M. Lau and A. L. T. Paterson, {\em The exact cardinality of the set of topological left invariant means on an amenable locally compact group,} Proc. Amer. Math. Soc. 98 (1986) 75--80.

 \bibitem{LauPym90}
A. T. Lau and J. Pym, {\em Concerning the second dual of the group algebra of a locally compact group,} J. London Math. Soc. (2) 41 (1990) 445--460.


 \bibitem{LauPym93} A. T. Lau, A. R. Medghalchi and J. S. Pym,
{\em
On the spectrum of $L^\infty(G)$,}
J. London Math. Soc. (2) 48 (1993), no. 1, 152--166.



 \bibitem{LauPym95} A. T.-M. Lau and J. S. Pym, {\em
The topological centre of a compactification of a locally compact group,}
Math. Z. 219 (1995), no. 4, 567--79.

 \bibitem{laul} A. T.-M. Lau and A. \"Ulger, {\em Topological centers of certain dual algebras,} Trans.
Amer. Math. Soc. 348 (1996), 1191--1212.

\bibitem{losert} V. Losert, \emph{Properties of the Fourier algebra that are equivalent to the amenability,} Proc. Amer. Math. Soc. 92 (1984), 347--354.



 \bibitem{losert84} V. Losert, \emph{ On tensor products of Fourier algebras,} Arch. Math. (Basel) 43 (1984), no. 4, 370--372.





%\bibitem{LosertA}
%V.~Losert, \emph{The centre of the bidual of {F}ourier algebras (discrete
%  groups)}, Preprint.

\bibitem{losert1}
\bysame, \emph{Properties of the {F}ourier algebra that are equivalent to amenability},
Proc. Amer. Math. Soc. \textbf{92} (1984), 347--354.

\bibitem{losert2}
\bysame, Further results on the centre of the bidual of Fourier algebras,
Lectures in Harmonic Analysis, Operator Algebras and representations, CIRM 3-08.11.2008.

\bibitem{Loetal} V.
Losert, \emph{The centre of the bidual of Fourier algebras (discrete groups),} Bull. Lond. Math. Soc. 48 (2016), no. 6, 968--976.
% Losert\emph{centre}, 2017.


\bibitem{Mehdi} M. S. Monfared, \emph{Character amenability of Banach algebras}, Math. Proc. Cambridge Philos. Soc. 144 (2008), no. 3, 697--706.

\bibitem{nebbia} C. Nebbia, \emph{Multipliers and asymptotic behaviour of the Fourier algebra of nonamenable groups,} Proc. Amer. Math. Soc. 84 (1982), 549-554.


\bibitem{Neufang04}
M.~Neufang, \emph{Solution to a conjecture by {H}ofmeier--{W}ittstock}, J.
  Funct. Anal. \textbf{217} (2004), 171--180.

\bibitem{Neufang04B}
\bysame, \emph{A unified approach to the topological centre problem for certain
  {B}anach algebras arising in abstract harmonic analysis}, Arch. Math.
  \textbf{82} (2004), 164--171.

\bibitem{Neufang05}
\bysame, \emph{On a conjecture by Ghahramani--Lau and related problems
  concerning topological centres}, J. Funct. Anal. \textbf{224} (2005),
  217--229.
		
		\bibitem{Neufang08} \bysame, \emph{On Mazur's property and property (X),}
				J. Operator Theory 60 (2008), no. 2, 301--316.
		
	
		
	\bibitem{Neufang08B} \bysame, \emph{ On the topological centre problem for weighted convolution algebras and semigroup compactifications} Proc. Amer. Math. Soc. 136 (2008), no. 5, 1831--1839.
	
		\bibitem{Neufang09} \bysame, \emph{	 On one-sided strong Arens irregularity,} Arch. Math. (Basel) 92 (2009), no. 5, 519--524.
		
		
		
		
	\bibitem{pier}	J. Pier, \emph{Amenable locally compact groups,} Wiley-Interscience, New York, 1984.
		
		
 \bibitem{PyPo} I. V. Protasov and J. S. Pym,
{\em Continuity of multiplication in the largest compactification of a locally compact group,}
Bull. London Math. Soc. 33 (2001), no. 3, 279--282.




\bibitem{rosen} H.P. Rosenthal, \textit{A characterization of Banach spaces containing $\ell^1$,} Proc. Nat. Acad. Sci. U.S.A., 71 (1974), 2411-2413.




\bibitem{Sh} {S. Sherman,} \emph{The second adjoint of a C*-algebra,} Proc. International
Congress Mathematicians, Cambridge, Mass., Volume 1, (1950), 470.




\bibitem{ross} R. Stokke, {\em
On Beurling measure algebras,}
Comment. Math. Univ. Carol. (to appear).


\bibitem{T} {Z. Takeda,} \emph{ Conjugate spaces of operator algebras,} Proc. Japan Academy,
30 (1954), 90--95.


\bibitem{vanDouwen91}
E.~K. van Douwen, \emph{The number of cofinal ultrafilters}, Topol. Appl.
  \textbf{39} (1991), 61--63.
	
	
	
	\bibitem{safoura} S. Zadeh, {\em Isometric isomorphisms of Beurling algebras,} J. Math. Anal. Appl. 438 (2016), no. 1, 1--13.
	
	
\bibitem{Z} Y. Zelenyuk, {\em The number of minimal right ideals of $\beta G$,} Proc. Amer. Math. Soc. 137 (2009), no. 7, 2483--2488.



\end{thebibliography}

\end{document}